\documentclass[12pt]{article}
\usepackage[T1]{fontenc}
\usepackage[latin1]{inputenc}
\usepackage{geometry}
\usepackage{float}
\usepackage{mathrsfs}
\usepackage{amsmath}
\usepackage{amssymb}
\usepackage{graphicx}
\usepackage{setspace}
\usepackage{subfigure}
\usepackage{hyperref}
\hypersetup{
    colorlinks=true,
    linkcolor=blue,
    filecolor=magenta,      
    urlcolor=cyan,
    citecolor = red,
}
\usepackage{gensymb}

\makeatletter

\floatstyle{ruled}
\newfloat{algorithm}{tbp}{loa}
\usepackage{algorithm}
\usepackage{algpseudocode}
\usepackage{caption}

\usepackage{bbm}

\usepackage{amsthm}

\usepackage{mathrsfs}

\usepackage{amsfonts}

\usepackage{epsfig}

\usepackage{bm}

\usepackage{mathrsfs}

\usepackage{enumerate}

\@ifundefined{definecolor}{\@ifundefined{definecolor}
 {\@ifundefined{definecolor}
 {\usepackage{color}}{}
}{}
}{}

\newtheorem{theorem}{Theorem}[section]
\newtheorem{lemma}{Lemma}[section]

\newtheorem{remark}{Remark}[section]
\newtheorem{proposition}{Proposition}[section]
\newtheorem{corollary}{Corollary}[section]
\newtheorem{ass}{Assumption}[section]

\providecommand{\algorithmname}{Algorithm}
\floatname{algorithm}{\protect\algorithmname}

\newcounter{hypA}

\newcounter{hypB}

\newcounter{hypD}

\makeatother

\usepackage{booktabs}
\usepackage{multirow}
\usepackage[table,xcdraw]{xcolor}


\renewenvironment{abstract}
 {\small
  \begin{center}
  \bfseries \abstractname\vspace{-.5em}\vspace{0pt}
  \end{center}
  \list{}{%
    \setlength{\leftmargin}{5mm}% <---------- CHANGE HERE
    \setlength{\rightmargin}{\leftmargin}%
  }%
  \item\relax}
 {\endlist}

\usepackage{blkarray}

\begin{document}
\emergencystretch 3em
%+Title
\begin{center}

{\Large \textbf{Antithetic Multilevel Methods for Elliptic and Hypo-Elliptic Diffusions with Applications}}

\vspace{0.5cm}

BY YUGA IGUCHI$^{1}$, AJAY JASRA$^{2}$, MOHAMED MAAMA$^{3}$ \& ALEXANDROS BESKOS$^{1}$   \vspace{0.2cm}

{\footnotesize $^{1}$Department of Statistical Science, University College London, London, WC1E 6BT, UK.}
{\footnotesize E-Mail:\,} \texttt{\emph{\footnotesize yuga.iguchi.21@ucl.ac.uk, a.beskos@ucl.ac.uk}}\\
{\footnotesize $^{2}$School of Data Science,  The Chinese University of Hong Kong, Shenzhen, Shenzhen, CN.}
{\footnotesize E-Mail:\,} \\ \texttt{\emph{\footnotesize ajayjasra@cuhk.edu.cn}}\\
{\footnotesize $^{3}$Applied Mathematics and Computational Science Program, Computer, Electrical and Mathematical Sciences and Engineering Division, King Abdullah University of Science and Technology, Thuwal, 23955-6900, KSA.} {\footnotesize E-Mail:\,} \texttt{\emph{\footnotesize maama.mohamed@gmail.com}}\\
\end{center}

\begin{abstract}
We present a new antithetic multilevel Monte Carlo (MLMC) method for the estimation of expectations with respect to laws of diffusion processes that can be elliptic or hypo-elliptic. In particular, we consider
the case where one has to resort to time discretization of the diffusion and numerical simulation of such schemes. Inspired by recent works, we introduce a new MLMC estimator of expectations, which does not require any L\'evy area simulation and has a strong error of order 2 and a weak error of order 2. We then show how this approach can be used in the context of the filtering problem associated to partially observed diffusions with discrete time observations. We illustrate that in numerical simulations our new approaches provide efficiency gains for several problems, particularly when the diffusion process is hypo-elliptic, relative to some existing methods. 
%   , particularly when the diffusion process is hypo-elliptic.
\\
\noindent \textbf{Keywords}: Filtering, hypo-elliptic diffusion, multilevel Monte Carlo, stochastic differential equations.
\end{abstract}

%%%%%%%%%%%%%%%%%%%%%%%%%%%%%%%%%
\section{Introduction} \label{sec:intro}

We consider $N$-dimensional stochastic differential equation (SDE): 
\begin{equation} 
	\label{eq:diff}
	d X_t = \sigma_{0} (X_t) dt 
	+ \sum_{1 \le j \le d} \sigma_{j} (X_t) d B_t^j, \quad  X_0 = x \in \mathbb{R}^N, 
\end{equation}
where $\{B_t\}_{t \ge 0}$ is the $d$-dimensional standard Brownian motion defined upon the filtered probability space $(\Omega, \mathcal{F}, \{\mathcal{F}_t\}_{t \ge 0}, \mathbb{P})$, and $\sigma_{j} : \mathbb{R}^N \to \mathbb{R}^N$ satisfies some regularity conditions, to be made precise later, with $\sigma_j = [\sigma_j^1, \ldots, \sigma_j^N]^\top$ for $0 \le j \le d$.
Throughout the paper, the matrix $a = \sigma \sigma^\top$ can be degenerate, with $\sigma \equiv [\sigma_1, \ldots, \sigma_d]$. Thus, this class of diffusion process includes certain elliptic and hypo-elliptic diffusion processes that can be found in applications; see for instance \cite{kloeden}. In particular, \textcolor{black}{a lot of interest is shown recently in the literature for numerical analysis and statistical inference methods for hypo-elliptic diffusions}
%has been highlighted in literature 
(see e.g. \cite{dit, glot, iguchi1, iguchi2}). 
We consider the context that one cannot obtain an exact solution of the SDE, despite its existence, and has to resort to time-discretization of the diffusion and the associated numerical simulation and, again, there are many examples of such processes that are used in practice \cite{kloeden}.

The collection of problems that we focus upon in this article is, firstly, the computation of expectations with respect to (w.r.t.) laws of diffusion processes; we call this the \emph{forward problem}. That is,  given a function $\varphi:\mathbb{R}^N\rightarrow\mathbb{R}$ that is integrable w.r.t.~the transition law of the diffusion, the objective is the computation of a numerical approximation of $\mathbb{E}[\varphi(X_T)]$ for some given terminal time $T>0$.  
%This has numerous applications, e.g.~option pricing in mathematical finance \cite{glass}.
\textcolor{black}{Secondly, we consider the \emph{filtering problem}} for partially observed diffusion processes that are discretely observed in time.  In other words \eqref{eq:diff} is a latent process that is observed through noisy data, only at discrete times (which we take as unit times for simplicity). The objective is then to compute an approximation of the conditional expectation of $X_t$ at each observation time and given all the data available up-to that time. This is a classical problem in engineering, statistics and applied mathematics, see e.g.~\cite{bain,cappe} for further references and applications. 
%Our numerical methods will ultimately rely on (stochastic) Monte Carlo simulation, to be detailed later on.

For both aforementioned problems, one must resort to a time discretization of \eqref{eq:diff} whose properties can be critical for any resulting numerical approximation method relying on it. There are several numerical methods in the literature, such as the Euler-Maruyama (E-M) method and the Milstein scheme; see for instance \cite{kloeden}. The main properties that are often of interest to inform the efficiency of the approximation are the weak and strong error, which we shall define, loosely, as follows -- a full definition can be found later on. \textcolor{black}{For a time discretization on a regular grid of spacing $\Delta>0$, and a corresponding numerical approximation $\{\bar{X}_t\}_{t\geq 0}$
the weak error (assuming it exists) is the discrepancy:
$$
| 
\mathbb{E} [ \varphi (X_T) ] 
- \mathbb{E} [ \varphi (\bar{X}_T) ] 
|
$$
for an appropriate test function $\varphi: \mathbb{R}^{N} \to \mathbb{R}$. 
%where $\mathbb{E}$ is the expectation w.r.t.~the probability law associated to the diffusion process $\{X_t\}_{t\geq 0}$ and its numerical approximation $\{\bar{X}_t\}_{t\geq 0}$ and $|\cdot|$ is the $L_1-$norm. 
We remark that the numerical approximation may be defined in continuous time by interpolation between points on the time grid.}
The strong error\footnote{\textcolor{black}{In the literature, $\bigl\{ \mathbb{E} \bigl[ \| X_T-\bar{X}_T \|^2 \bigr] \bigr\}^{1/2}$ is used as the standard definition of strong error. However, we make use of the squared version of the definition because it aligns with the analysis on the variance of couplings in the context of multilevel Monte Carlo.}} (assuming it exists) is taken as:
$$
\mathbb{E} \bigl[ \| X_T-\bar{X}_T \|^2 \bigr]
$$
where $\|\cdot\|$ is the $L_2-$norm. There are several results for well-known discretization methods; e.g., 
E-M has weak error of $\mathcal{O}(\Delta)$ (weak error 1) and strong error of $\mathcal{O}(\Delta)$ (strong error 1) and the Milstein scheme has weak error 1
%of $\mathcal{O}(\Delta)$  
and strong error 2. %of $\mathcal{O}(\Delta^2)$. 
In the context of the methods to be used in this article, one generally would like \textcolor{black}{the order of weak and strong error} to be `large' at a cost of  $\mathcal{O}(\Delta^{-1})$ for \emph{directly} simulating the approximation. We note that direct simulation,  without for instance solving linear equations of cost of order $\mathcal{O}(N^m)$, $m> 2$, is critical for practical 
problems, especially filtering.

In this work we consider both elliptic and hypo-elliptic diffusion processes and in the latter case we have $N\geq 2$. In such scenarios, the Milstein method (or the strong 1.5 scheme, see \cite{kloeden}, with weak error 2 and strong error 3) cannot often be simulated directly, without a restrictive commutative condition (given later on), as one has to compute an intractable \emph{L\'evy area}. In such cases one resorts to the E-M approach, which can be simulated exactly, but \textcolor{black}{the order of weak and strong error} is comparatively low.  Whilst there
are some higher order discretization methods based upon stochastic Runge-Kutta approaches (see e.g.~\cite{rumm}),
generally for many Monte Carlo simulation-based methods a strong error of 2 generally suffices for `optimal' (to be clarified later on) variance properties. 
\textcolor{black}{An elegant methodology that side-steps sampling of L\'evy areas but preserves \textcolor{black}{strong error 2} was developed in \cite{ml_anti} based upon the multilevel Monte Carlo (MLMC) approach \cite{giles,giles1,hein}.}

MLMC works with a hierarchy of time-discretised diffusions, that is with a collection of step-sizes $0<\Delta_0<\cdots<\Delta_L$, $L
\in\mathbb{N}$. Then one rewrites the expectation of interest as a decomposition of the difference of the exact (no time discretization) expectation and the one with the finest time discretization and then a telescoping sum of differences of expectations associated to increasingly coarse step-sizes. Then, if one can appropriately simulate dependent (\emph{coupled}) time discretizations for pairs of step-sizes it is possible to reduce the cost of a Monte Carlo based algorithm (e.g.~the cost versus a direct simulation of the time discretised diffusion with a single step-size $\Delta_L$) to achieve a pre-specified mean square error (MSE) using MLMC; see e.g.~\cite{giles1} for a review.  \cite{ml_anti} introduce an \emph{antithetic} MLMC \textcolor{black}{(AMLMC)} using the \emph{truncated Milstein scheme} (defined in Section \ref{sec:t-Mil}) which has weak error 1 and stong error 1 \textcolor{black}{without requiring the simulation of intractable L\'evy areas, but the variance of couplings at each level decays w.r.t. the step-size at the same rate as the case of a time discretization having strong error 2, which leads to an optimal computational complexity}. 

In this article we develop a new method (multilevel-based) for time discretization which is \textcolor{black}{effective} in both the elliptic and hypo-elliptic contexts. Motivated by the work in \cite{ml_anti}, we derive a new \textcolor{black}{AMLMC} based on the numerical scheme proposed in \cite{iguchi2} achieving weak error 2 and strong error 1 (the latter is proven in this article), which still gives an optimal computational complexity (for the forward problem). The method can also be simulated directly with a cost of $\mathcal{O}(\Delta^{-1})$ per-pair of levels $0<\Delta<\Delta'$.  
An \textcolor{black}{AMLMC} with a weak error 2 has also been investigated in \cite{nv_anti}, where they used an alternative numerical scheme with a weak error 2 and emphasized its efficiency due to the reduction of the number of time-discretizations, 
which is an advantage over the AMLMC that uses the truncated Milstein scheme (weak error 1). 
A comparison between our proposed AMLMC and the method by \cite{nv_anti} is given later in Section \ref{sec:weak_mlmc_scheme}. 
% \textcolor{black}{We note that the approach does not provide a density of the time discretized transition density in the hypo-elliptic setting,  which can be important in certain applications (e.g.~\cite{iguchi2}),  but we believe our ideas can be extended just as in \cite{iguchi2}.} 
In addition, we show that our new methodology can be used for the filtering problem.  Some of the state-of-the-art numerical methods for this problem are based upon particle filters (e.g.~\cite{cappe,delmoral}) related to the MLMC approach, which are termed \emph{multilevel particle filters} (MLPFs) see e.g.~\cite{mlpf,anti_mlpf}.  Based upon the methodology developed herein, we derive a new MLPF.  To summarize, the main contributions of this article are:
\begin{itemize}
	\item{We introduce a locally non-degenerate scheme of weak error 2 for both elliptic and hypo-elliptic contexts, inspired by \cite{iguchi2}. We \textcolor{black}{prove} that the scheme has strong error~1.} 
	\item{ 
%		\textcolor{black}{Based on the above scheme}, 
		We then develop a new \textcolor{black}{AMLMC} method  
%		for elliptic and hypo-elliptic SDEs
		 \textcolor{black}{that does not contain L\'evy areas} and prove that the variance of the AMLMC estimator decays (w.r.t the step-size) \textcolor{black}{at the same rate as for a discretization scheme that would achieve a strong error 2.}} 
	% this method has a strong error of 2 (\textcolor{black}{in the sense of... more accurate definition}).}
% The weak error of order 2 was proved in \cite{iguchi2}.}
%\item{We show that for \textcolor{black}{the important class of \emph{small-noise}} diffusions, i.e.~SDEs containing a parameter \textcolor{black}{$\mu \ll 1$ at the front of the diffusion coefficient}, the variance of the proposed \textcolor{black}{AMLMC} estimator \textcolor{black}{vanishes as $\mathcal{O}(\mu^2)$ -- in constrast  
%	with the antithetic scheme of \cite{ml_anti} when the corresponding variances do not diminish for small $\mu$.}}
	\item{We show how to use the new \textcolor{black}{AMLMC} method for filtering within the context of MLPFs.}
	\item{
%		We illustrate our approaches in several numerical examples for both the forward and filtering problems.
We present numerical results to show that 
%We show numerically that 
our method can out-perform some competing approaches.}
\end{itemize}
% 
%The first bullet point is needed in the development of the article, to establish the necessity of an antithetic type extension which has weak order 2.
We further elaborate on some of the bullet points above. In the case of the forward problem, the \textcolor{black}{second bullet point leads to the new \textcolor{black}{AMLMC} estimator having a cost of $\mathcal{O}(\epsilon^{-2})$ to give a MSE of $\mathcal{O}(\epsilon^2)$, $\epsilon>0$, %then the cost to achieve such MSE is $\mathcal{O}(\epsilon^{-2})$ 
i.e.~the method attains the optimal cost for (stochastic) Monte Carlo based methods.} 
Such a MSE is also achieved by \cite{ml_anti}, however the higher rate of weak error \textcolor{black}{is expected to provide efficiency gains -- verified in our numerical experiments -- due to the necessity of the use of a finite $L$ (the most precise level) in simulations.} 
%for simulations with  finite $L$ (the most precise level) simulation. 
%Furthermore, the third bullet implies that our new \textcolor{black}{AMLMC method} achieves diminishing variance for small-noise diffusions (many models in applications are within this scope), which also prompts reduction of the computational cost. 
In the case of filtering,  we compare to the \textcolor{black}{MLPF} approaches in \cite{mlpf,anti_mlpf}.
\textcolor{black}{In \cite{mlpf} the authors prove that, in the elliptic case,  to achieve a MSE of $\mathcal{O}(\epsilon^2)$ there is a cost of $\mathcal{O}(\epsilon^{-2.5})$.  In \cite{anti_mlpf} the authors show that in simulations
to achieve a MSE of $\mathcal{O}(\epsilon^2)$ there is a cost of $\mathcal{O}(\epsilon^{-2}\log(\epsilon)^2)$; this
latter MLPF corresponds to an embedding of the multilevel approach of \cite{ml_anti} within the filtering problem.
}
We verify in our simulations that, as one expects based upon \cite{mlpf}, our new MLPF has \textcolor{black}{costs consistent with the anticipated rate $\mathcal{O}(\epsilon^{-2}\log(\epsilon)^2)$} to achieve a MSE of 
$\mathcal{O}(\epsilon^2)$. However, as the discretization schemes \textcolor{black}{underpinning} the methods in  \cite{mlpf,anti_mlpf} have weak error 1, we again observe efficiency gains for finite $L$. 
Finally, we note that our numerical scheme is locally \emph{non-degenerate} under a hypo-elliptic setting, while this is not the case for the truncated Milstein scheme. The existence of the density (non-degeneracy) is important in the filtering problem when utilising guided proposals \textcolor{black}{\cite{chopin}} to improve the performance of particle filters. 

This paper is structured as follows. In Section \ref{sec:num_schemes} we consider several numerical schemes for
SDEs and introduce our approach. In Section \ref{sec:filtering} we describe how our idea can be used in the context of the filtering problem and derive the new MLPF. In Section \ref{sec:numerics} we present our numerical results to illustrate our theoretical derivations. The mathematical proofs of our main results are given in the Appendix. 
% the Supplementary Material \cite{supp}.
\\ 

\noindent
\textbf{Notation:} Let $C_b^{K} (\mathbb{R}^n; \mathbb{R}^m)$, $n, m, K \in \mathbb{N}$, be the space of $K$-times differentiable functions $f : \mathbb{R}^n \to \mathbb{R}^m$ such that partial derivatives up to order $K$ are bounded. 
% Also, $\mathcal{B}_b (\mathbb{R}^n), \, n \in \mathbb{N},$ denotes the space of bounded-Borel functions $f: \mathbb{R}^n \to \mathbb{R}$. 
For a vector $y \in \mathbb{R}^N$, we define the norm $\|\cdot\|$ as $\| y \| \equiv \sqrt{\sum_{1 \le i \le N} y_i^2}$.  

\section{Numerical schemes}\label{sec:num_schemes} 
\subsection{Basic assumptions and error} 
\label{sec:setting}
To study a broad class of SDEs including the case where the matrix $a = \sigma \sigma^\top$ is degenerate, we consider the following structure for model (\ref{eq:diff}):
\begin{align} \label{eq:sde} 
d X_t = 
\begin{bmatrix}
d X_{S, t} \\ 
d X_{R, t} 
\end{bmatrix}
=
\begin{bmatrix}
\sigma_{S, 0} (X_t)  \\
\sigma_{R, 0} (X_t) 
\end{bmatrix}  dt 
+ \sum_{1 \le j \le d}
\begin{bmatrix}
\mathbf{0}_{N_S} \\
\sigma_{R, j} (X_t) 
\end{bmatrix} 
dB_t^j, \qquad X_0 = x \in \mathbb{R}^N, 
\end{align}
where we have set  
$
\sigma_{S, 0} : \mathbb{R}^N \to \mathbb{R}^{N_S},  
\, 
\sigma_{R, j} : \mathbb{R}^N \to \mathbb{R}^{N_R}, 
\, 0 \le j \le d, 
$ 
with integers $N_S \ge 0$, $N_R \ge 1$ such that $N_S + N_R = N$. We write for $x \in \mathbb{R}^N$:    
\begin{align*}
\sigma_0 (x) = 
\bigl[ 
\sigma_{S, 0} (x)^\top, 
\sigma_{R, 0} (x)^\top 
\bigr]^\top, 
\qquad 
\sigma_j (x) = 
\bigl[ 
\mathbf{0}_{N_S}^\top, 
\sigma_{R, j} (x)^\top  
\bigr]^\top,  
\quad 1 \le j \le d,
\end{align*}
and $a \equiv \sigma \sigma^\top$ with $\sigma \equiv [\sigma_1, \ldots, \sigma_d]$. Notice that when $N_S \ge 1$, the matrix $a$ is degenerate. We write  
$
[\sigma_0, \sigma_j] (x) \equiv \sum_{1 \le k \le  N} 
\bigl\{ \widetilde{\sigma}_0^k (x) 
\partial_{x_k} \sigma_j (x) 
-  {\sigma}_j^k (x)
\partial_{x_k} \widetilde{\sigma}_0 (x)
\bigr\}$, $1 \le j \le d$,     
where $\widetilde{\sigma}_0 : \mathbb{R}^N \to \mathbb{R}^N$ is the drift function when \textcolor{black}{the It\^o-type SDE (\ref{eq:sde}) is written as a Stratonovich one, specifically}, 
$
\textcolor{black}{\widetilde{\sigma}_0 (x) \equiv \sigma_0 (x) - \tfrac{1}{2} \sum_{1 \le i \le N} \sum_{1 \le j \le d} \sigma_j^i (x) \partial_i \sigma_j (x).}   
$ 

We introduce the following assumptions related to \emph{H\"ormander's condition} (see e.g. \cite{malliavin}).  
\begin{ass} \label{ass:coeff}
$\sigma_j \in C_b^{\infty} (\mathbb{R}^N; \mathbb{R}^N)$, $0 \le j \le d$. 
\end{ass} 
\begin{ass} \label{ass:hor}
(i) Ellipticity. When $N_S = 0$, it holds that for any $x \in \mathbb{R}^N$: 
\begin{align*}
\mathrm{Span} \bigl\{ \sigma_1 (x), \ldots, \sigma_d (x) \bigr\} = \mathbb{R}^N. 
\end{align*} 
% 
%which is equivalent to the matrix $a (x) = \sigma (x) \sigma (x)^\top$ being positive definite for any $x \in \mathbb{R}^N$.  \\ 
(ii) Hypo-ellipticity. When $N_S \ge 1$, it holds that for any $x \in \mathbb{R}^N$:
\begin{gather*}
\mathrm{Span} \bigl\{\sigma_{R, 1}(x), \ldots, \sigma_{R, d}(x) \bigr\} = \mathbb{R}^{N_R}, \,   
\mathrm{Span} \bigl\{\sigma_1 (x), \ldots, \sigma_d (x), [\sigma_0, \sigma_1] (x),\ldots, [\sigma_0, \sigma_{d}] (x)
\bigr\} = \mathbb{R}^{N}.  
\end{gather*}  
%s
\end{ass}
\noindent 
% 
% ***This Section needs a lot of work***
For a numerical scheme $\{ \bar{X}_{k \Delta} \}_{k = 0,1, \ldots, 2^{\ell}}$ with constant step-size $\Delta = T/2^{\ell}$ and non-negative integer $\ell$, \textcolor{black}{weak and strong error of order $m\ge 1$, respectively, are given as follows:} 
$$
\left|\mathbb{E}\left[\varphi (X_T)\right] - \mathbb{E}\left[ \varphi (\bar{X}_T) \right]
\right| 
= \mathcal{O}(\Delta^{m}), \qquad 
\mathbb{E} 
\bigl[ 
\max_{0 \le k \le 2^{\ell}}
\bigl\|  X_{k \Delta} - \bar{X}_{k \Delta}  \bigr\|^{2p} 
\bigr] 
= \mathcal{O} (\Delta^{mp}), \quad p \ge 1, 
$$ 
for some appropriate test function $\varphi : \mathbb{R}^N \to \mathbb{R}$.  
\subsection{Discretizations}

We introduce a discretization scheme of weak order 2 
% which will have a key involvement in our methodology}. 
and also mention other popular discretization schemes (e.g.~the Milstein scheme) for comparison. 
Let $T > 0$ be a terminal time and $\Delta_\ell = T/2^{\ell}$ be a step-size of discretization with a non-negative integer $\ell$. We make use of the notation $t_{k} = k \Delta_{\ell}$, $1 \le k \le 2^{\ell}$, and $\Delta B_{t, s} = B_{t} - B_{s}, \, 0 \le s \le t$. For a sufficiently smooth function $\varphi: \mathbb{R}^N \to \mathbb{R}$, we set: 
\begin{align*}
\mathcal{L}_0 \varphi(x) 
& = \sum_{1 \le k \le N}
\sigma_{0}^k (x) \partial_{x_k} \varphi(x) 
+ \tfrac{1}{2} 
\sum_{1 \le j \le d}  \sum_{1 \le k_1, k_2 \le N} 
\sigma_{j}^{k_1} (x) 
\sigma_{j}^{k_2} (x) 
\partial_{x_{k_1}} \partial_{x_{k_2}} \varphi(x); \\ 
\mathcal{L}_i \varphi(x) 
& = \sum_{1 \le k \le N}
\sigma_{i}^k (x) \partial_{x_k} \varphi (x), \qquad 1 \le i \le d.   
\end{align*}
We define, for $1 \le i_1, i_2 \le d$:
\begin{align*}
[\sigma_{i_1}, \sigma_{i_2}] (x) 
=  
\mathcal{L}_{i_1} \sigma_{i_2} (x) 
- \mathcal{L}_{i_2} \sigma_{i_1} (x), 
\qquad x \in \mathbb{R}^N,     
\end{align*} 
where 
$
\mathcal{L}_{i_1} \sigma_{i_2} (x) 
= [
\mathcal{L}_{i_1} \sigma_{i_2}^1 (x), 
\ldots, 
\mathcal{L}_{i_1} \sigma_{i_2}^N (x)  
]^\top \in \mathbb{R}^N.
$ 
\subsubsection{Milstein scheme}
{The Milstein scheme, of weak order 1 and strong order 2, writes as follows, for $0 \le k \le 2^{\ell} -1$}: 
\begin{align} \label{eq:milstein}
\begin{aligned}
& \bar{X}_{t_{k+1}}^{\mathrm{Mil}}
= \bar{X}_{t_k}^{\mathrm{Mil}}
+ \sigma_{0} (\bar{X}_{t_k}^{\mathrm{Mil}}) \Delta_\ell
+ \sum_{1 \le j \le d} 
\sigma_{j} (\bar{X}_{t_k}^{\mathrm{Mil}}) 
\Delta B_{t_{k+1}, t_k}^j \\ 
& \qquad 
+ \tfrac{1}{2} \sum_{1 \le j_1, j_2 \le d}
\mathcal{L}_{j_2} \sigma_{j_1} 
(\bar{X}_{t_k}^{\mathrm{Mil}}) 
\bigl(
\Delta B_{t_{k+1}, t_k}^{j_1} 
\Delta B_{t_{k+1}, t_k}^{j_2} 
- \Delta_\ell \cdot \mathbf{1}_{j_1 = j_2} 
- \Delta A_{t_{k+1}, t_k}^{j_1j_2} \bigr), 
\end{aligned}   \tag{Milstein}  
\end{align}
with $\bar{X}_{0}^{\mathrm{Mil}} 
= x$, where $\Delta A_{t_{k+1},t_k}^{j_1j_2}$ is a L\'evy area specified as 
$ 
\Delta A_{t_{k+1}, t_k}^{j_1j_2}  
\equiv  
{\textstyle 
\int_{t_k}^{t_{k+1}} 
\int_{t_k}^{s} d B_u^{j_1} d B_{s}^{j_2}
} 
- 
{\textstyle 
\int_{t_k}^{t_{k+1}} 
\int_{t_k}^{s} d B_{u}^{j_2} d B_{s}^{j_1}
}.   
$
Note that in general there is no effective way to directly simulate the L\'evy area. However, if the \emph{commutative condition} holds, i.e. for any $x \in \mathbb{R}^N$, 
\begin{align} \label{eq:comm}
\bigl[\sigma_{j_1}, \sigma_{j_2} \bigr] (x) = 0, 
% \mathcal{L}_{j_1} \sigma_{j_2} (x) - \mathcal{L}_{j_2} \sigma_{j_1} (x) = 0, 
\qquad 1 \le j_1, j_2 \le d, \ \ j_1 \neq j_2, 
\end{align}
then the L\'evy area does not appear in the Milstein scheme and the latter becomes tractable. 
\subsubsection{Truncated Milstein scheme} 
\textcolor{black}{The truncated Milstein scheme, used by the AMLMC method of \cite{ml_anti}, has weak and strong errors both of order equal to 1, and writes as follows, for $0 \le k \le 2^{\ell} -1$}:
\label{sec:t-Mil} 
\begin{align} \label{eq:t_mil}
\begin{aligned} 
& \bar{X}_{t_{k+1}}^{\textrm{T-Mil}}
= \bar{X}_{t_k}^{\textrm{T-Mil}}
+ \sigma_{0} (\bar{X}_{t_k}^{\textrm{T-Mil}}) 
\Delta_\ell 
+ \sum_{1 \le j \le d}
\sigma_{j} (\bar{X}_{t_k}^{\textrm{T-Mil}}) 
\Delta B_{t_{k+1}, t_k}^j   \\ 
& \qquad 
+ \tfrac{1}{2} \sum_{1 \le j_1, j_2 \le d}
\mathcal{L}_{j_1} \sigma_{j_2} 
(\bar{X}_{t_k}^{\textrm{T-Mil}}) 
\bigl(
\Delta B_{t_{k+1}, t_k}^{j_1} 
\Delta B_{t_{k+1}, t_k}^{j_2} 
- \Delta_\ell \cdot \mathbf{1}_{j_1 = j_2} 
\bigr),
\end{aligned}  \tag{T-Milstein} 
\end{align}
with $\bar{X}_{0}^{\textrm{T-Mil}} = x$. Since the scheme omits the L\'evy area $\Delta A_{t_{k+1}, t_k}^{j_1 j_2}$, the strong convergence rate is the same as for the E-M scheme unless the commutative condition holds. 

\subsubsection{Second order weak scheme} \label{sec:weak2}
Motivated from  \cite{iguchi2}, we introduce two nondegenerate discretization schemes for elliptic ($N_S = 0$) and hypo-elliptic ($N_S \ge 1$) cases, separately. That is, for $0 \le k \le 2^{\ell}-1$: 
\begin{align} \label{eq:scheme}
\begin{aligned}
& \bar{X}_{t_{k+1}} 
= \bar{X}_{t_k} 
+ \sigma_{0} (\bar{X}_{t_k}) 
\Delta_{\ell}
+ \sum_{1 \le j \le d}
\sigma_{j} (\bar{X}_{t_k}) \Delta B_{t_{k+1}, t_k}^j \\  
&  \qquad  + 
\sum_{0 \le j_1, j_2 \le d}
\mathcal{L}_{j_1} \sigma_{j_2} (\bar{X}_{t_k}) \Delta \eta^{j_1 j_2}_{t_{k+1}, t_k} 
+ \tfrac{1}{2} \sum_{1 \le j_1 < j_2 \le d}
[\sigma_{j_1}, \sigma_{j_2}] (\bar{X}_{t_k}) \Delta \widetilde{A}^{j_1 j_2}_{t_{k+1}, t_k}, 
\end{aligned} \tag{Weak-2}   
\end{align} 
with $\bar{X}_{0}  = x$, where the random variables 
$\Delta \eta^{j_1 j_2}_{t_{k+1}, t_k}$ and 
$\Delta \widetilde{A}^{j_1 j_2}_{t_{k+1}, t_k}$ are given as:  
\begin{align*}
\Delta \eta^{j_1 j_2}_{t_{k+1}, t_k}  
= 
\begin{cases}
\Delta \eta^{\textrm{Ell}, j_1 j_2}_{t_{k+1}, t_k}  & (N_S = 0);  \\[0.2cm]
\Delta \eta^{\textrm{H-Ell}, j_1 j_2}_{t_{k+1}, t_k}  & (N_S \ge 1), 
\end{cases}
\qquad  
\Delta \widetilde{A}^{j_1 j_2}_{t_{k+1}, t_k} 
= \Delta B_{t_{k+1}, t_k}^{j_1} \Delta \widetilde{B}_{t_{k+1}, t_k}^{j_2},  
\end{align*}
% 
%  \begin{gather*}
% \Delta \xi^{\textrm{Ell}, j_1 j_2}_{t_{i+1}} 
%  =  \xi_{0 j_1 , i+1}  = \tfrac{1}{2}
%  \delta B_{i+1}^{j_1} h, \qquad \xi_{j_1 j_1, i+1} = \tfrac{1}{2} (\delta B_{i+1}^{j_1})^2 - h;  \\[0.2cm] 
%  \xi_{j_1, j_2, i+1} = 
%  \begin{cases}
%  \tfrac{1}{2} \delta B_{i+1}^{j_1} \delta B_{i+1}^{j_2} + 
%  \tfrac{1}{2} \delta B_{i+1}^{j_1} \delta \widetilde{B}_{i+1}^{j_2}, &  0 < j_1 < j_2; \\[0.1cm]
%  \tfrac{1}{2} \delta B_{i+1}^{j_1} \delta B_{i+1}^{j_2} 
%  -  \tfrac{1}{2} \delta B_{i+1}^{j_2} \delta \widetilde{B}_{i+1}^{j_1}, &  0 < j_2 < j_1
%  \end{cases}, 
%  \end{gather*}
% 
where $\widetilde{B}_t = (\widetilde{B}_t^2, \ldots, \widetilde{B}_t^d), \, t \ge 0$, is a standard $(d-1)$-dimensional Brownian motion independent of 
$\{B_t\}_{t \ge 0}$ and:
\begin{align}
\Delta \eta^{\textrm{Ell}, j_1 j_2}_{t_{k+1}, t_k} 
& = 
\tfrac{1}{2} \bigl( 
\Delta B_{t_{k+1}, t_k}^{j_1} \Delta B_{t_{k+1}, t_k}^{j_2} - \Delta_\ell \cdot \mathbf{1}_{j_1 = j_2 \neq 0}
\bigr), 
\quad 0 \le j_1, j_2 \le d;
\nonumber
\\[0.2cm]   
\Delta \eta^{\textrm{H-Ell}, j_1 j_2}_{t_{k+1}, t_k}  
& = 
\begin{cases}
\Delta \eta^{\textrm{Ell}, j_1 j_2}_{t_{k+1}, t_k}  
& (1 \le j_1, j_2 \le d \ \textrm{or} \ j_1 = j_2 = 0);  \\[0.1cm]
{\textstyle \int_{t_k}^{t_{k+1}} \int_{t_k}^{s}} du dB_s^{j_2} 
& (j_1 = 0, 1 \le j_2 \le d); \\[0.1cm]
{\textstyle \int_{t_k}^{t_{k+1}} \int_{t_k}^{s}} dB_u^{j_1} ds 
& (1 \le j_1 \le  d, j_2 = 0).    \\   
\end{cases} 
\nonumber
%\label{eq:hypo_rv}
\end{align}
In the above specification of $\Delta \eta_{t_{k+1}, t_k}^{j_1 j_2}$, we use the interpretation $\Delta B^0_{t_{k+1}, t_k}  = \Delta_\ell$. The definition of the scheme under the hypo-elliptic setting slightly differs from the original one given in \cite{iguchi2}. In particular, the latter includes additional random variables in the approximation of the smooth component $X_{S, t}$ for the purpose of improving the performance of parameter estimation. \textcolor{black}{Without such additional variables, it is shown that (\ref{eq:scheme}) achieves a weak error 2 since the random variables used in the scheme satisfy the moment conditions outlined in \cite[Lemma 2.1.5]{mil} that are sufficient for the attained order of weak convergence.} 

We give several remarks on scheme (\ref{eq:scheme}). First, 
comparing with the truncated Milstein scheme, we observe that the scheme contains the terms $\Delta \widetilde{A}$ and random variables of size $\mathcal{O} (\Delta^{3/2}_{\ell})$ or $\mathcal{O} (\Delta^2_{\ell})$. 
Due to the inclusion of these terms, (\ref{eq:scheme}) is shown to achieve weak error 2. 
In particular, variable $\Delta \widetilde{A}$ is interpreted as a proxy to the L\'evy area in the distributional (but not pathwise) sense. 
Thus, as we will show in Section \ref{sec:strong_weak2}, the order of strong convergence for (\ref{eq:scheme}) is not as good as that of the Milstein scheme which uses the true L\'evy area (though the latter cannot be exactly simulated in general). 
Second, under the hypo-elliptic setting ($N_S \ge 1$), the scheme, in particular $\Delta \eta_{t_{k+1}, t_k}^{\textrm{H-Ell}}$, involves $\textstyle \int_{t_k}^{t_{k+1}} \int_{t_k}^{s} du dB_s^{j_2}, \, \int_{t_k}^{t_{k+1}} \int_{t_k}^{s} dB_u^{j_1} ds$ that can be directly simulated by Gaussian variables that preserve the covariance structure between these integrals and the Brownian motion. Together with Assumption~\ref{ass:hor}, use of these variables leads to the current state $\bar{X}_{t_{k+1}}$ given $\bar{X}_{t_k}$ containing a locally Gaussian approximation with non-degenerate covariance, that is: 
\begin{align} \label{eq:lg}
\begin{aligned}
\bar{X}_{S, t_{k+1}} 
& \approx 
\bar{X}_{S, t_k} 
+ \sigma_{S, 0} (\bar{X}_{t_k}) \Delta_\ell 
+ \sum_{1 \le j \le d} \mathcal{L}_j \sigma_{S, 0} (\bar{X}_{t_k}) 
\Delta \eta^{\textrm{H-Ell}, j 0}_{t_{k+1}, t_k};  \\
\bar{X}_{R, t_{k+1}} 
& \approx 
\bar{X}_{R, t_k} 
+ \sigma_{R, 0} (\bar{X}_{t_k}) \Delta_\ell 
+ \sum_{1 \le j \le d} 
\sigma_{R, j} (\bar{X}_{t_k}) 
\Delta B^{j}_{t_{k+1}, t_k}. 
\end{aligned} 
\end{align} 
Note that if $\Delta \eta^{\textrm{H-Ell}, j 0}_{t_{k+1}, t_k}$ above is replaced with $\Delta \eta^{\textrm{Ell}, j 0}_{t_{k+1}, t_k}
\equiv \tfrac{1}{2} \Delta B_{t_{k+1}, t_k}^{j} \Delta_\ell$ which is used in the elliptic setting ($N_S = 0$), then the covariance of the right hand side (R.H.S.) of (\ref{eq:lg}) is no longer positive definite. 
\subsection{Strong convergence of the weak second order scheme and summary} \label{sec:strong_weak2}
The strong error rate  of scheme 
(\ref{eq:scheme}) is the same as for the truncated Milstein and the E-M scheme. The proof of the following result is in Appendix \ref{app:s_rate}.
\begin{proposition} \label{prop:s_rate}
For any $p \ge 1$, there exists a constant $C > 0$ such that
$$
\mathbb{E} 
\bigl[ \max_{0 \le k \le 2^\ell} 
\| X_{t_k} - \bar{X}_{t_k} \|^{2p} 
\bigr] 
\le C \Delta_\ell^{p}. 
$$
\end{proposition}
Table \ref{table:literature} summarises the weak and strong errors for some of the most popular discretization schemes. 
%Those marked in black are linked to those schemes discussed above.  
The result for the strong error of scheme (\ref{eq:scheme}) is new. 
% and its derivation is given in Section \ref{sec:strong_weak2} below.

\begin{table}[!h]
\centering
\caption{Numerical scheme for general SDEs (i.e.~commutative condition (\ref{eq:comm}) not assumed to hold).} % title of Table
\label{table:literature}
% used for centering table
\begin{tabular}{c c c} % centered columns (4 columns)
% \hline\hline %inserts double horizontal lines
\toprule
Scheme
& Rate of weak/strong convergence 
& Is L\'evy area required?
\\  % inserts table
\midrule
% \citet{sam:12} & 
% \begin{tabular} {c}
	%     (\ref{eq:hypo-I}) with $N_S = N_R = 1$. \\ 
	%     $\mu_{S} (X_t,  \beta_S) = X_{S,t}$.
	% \end{tabular} & $\Delta_n = o (n^{-1/2})$ &  $\circ$ \\
% [0.3cm]
% \hline 
% \\[-0.2cm] 
%Euler-Maruyama
%& 1.0 / 1.0
%& No 
%\\[0.1cm] % inserting body of the table
% \hline
% \\[-0.2cm] 
\ref{eq:milstein} 
& 1.0 / 2.0 
& Yes 
\\[0.1cm]
% \\[0.1cm]
% \hline 
% \\[-0.2cm] 
\ref{eq:t_mil} 
& 1.0 / 1.0 
& No 
\\[0.1cm]
% \hline 
% \\[-0.2cm]   
\ref{eq:scheme} 
& 2.0  / 1.0
& No
\\ 

% \\[0.1cm]  
% \hline
% \\[-0.2cm]    
% Strong 1.5 order (\cite{kloeden})
% & 2.0 
% & 3.0
% & Yes 
% \\[0.1cm]   
\bottomrule   
\end{tabular}
\label{table:nonlin} % is used to refer this table in the text
\end{table}  
\subsection{Antithetic MLMC with weak second order scheme}\label{sec:weak_mlmc_scheme}

The aim is to combine the weak order 2 method (\ref{eq:scheme}) with the ideas of \cite{ml_anti} and consider a new antithetic MLMC (AMLMC) estimator so that the variance of couplings at each level decays w.r.t. the step-size at the same rate as the case of a time discretization having strong error 2.  
%  We will define an antithetic MLMC estimator by making use of the non-degenerate weak second order scheme (\ref{eq:scheme}) for elliptic/hypo-elliptic SDEs.  
% 
Throughout this subsection, let $\ell = 0, \ldots, L$ be the level of discretization ($2^{-\ell}$), where $L \in \mathbb{N}$ is the finest level of discretization. We write $T > 0$ for the length of the time interval and $\Delta_{\ell} = T/2^{\ell}$ for the discretization step-size. 
To define the antithetic estimator, we design discretizations on coarse/fine grids based upon scheme (\ref{eq:scheme}). 
For a fixed $\ell \le L$, we define the coarse grids 
$\mathbf{g}^{c, [\ell-1]} =\{t_{k} \}_{k=0,1, \ldots, 2^{\ell-1}}  $ and the fine grids $\mathbf{g}^{f, [\ell]} = \{t_k, t_{k+1/2}\}_{k = 0,1,\ldots, 2^{\ell-1}-1} \cup \{T\}$, where 
$ 
t_{k} = k \Delta_{\ell -1}, \, 
t_{k + 1/2} = \bigl(k+ 1/2 \bigr)\Delta_{\ell -1}.
$
% For a time instance $t \in \mathbf{g}^{c,[\ell-1]}$ or $t \in \mathbf{g}^{f,[\ell]}$ with its previous time instance being denoted by   $t^\prime$, we write: 
% % 
% % 
% \begin{gather*}
%   \Delta B_{t, t^\prime}^0 = t - t^\prime, 
%   \quad  
%   \Delta B_{t, t^\prime}^{j_1} = B_{t}^{j_1} - B_{t^\prime}^{j_1},
%   \quad 
%   \Delta \widetilde{B}_{t, t^\prime}^{j_1} = \widetilde{B}_{t}^{j_1} - \widetilde{B}_{t^\prime}^{j_1},
%   \quad 
%   \Delta \eta_{t, t^\prime}^{00} = \tfrac{(t-t^\prime)^2}{2};  \\ 
%   \Delta \eta_{t, t^\prime}^{j_10} 
%   = \int_{t^\prime}^{t} B_{s-t^\prime}^{j_1} ds, 
%   \quad  
%   \Delta \eta_{t, t^\prime}^{0j_2} 
%   = \int_{t^\prime}^{t} (s-t^\prime) dB_s^{j_2}, 
%   \quad 
%   \Delta \eta_{t, t^\prime}^{j_1 j_2}
%   = \tfrac{1}{2}
%     \bigl\{ 
%       \Delta B_{t, t^\prime}^{j_1} 
%       \Delta B_{t, t^\prime}^{j_2}
%       - (t-t^\prime) \cdot \mathbf{1}_{j_1 = j_2}
%       \bigr\} 
% \end{gather*}
% %  
% for $1 \le j_1, j_2 \le d$. 
For notational simplicity, we introduce two integrators $\overline{\mathcal{I}}_{t, s}, \widetilde{\mathcal{I}}_{t, s}: \mathbb{R}^N \to \mathbb{R}^N$, $0 \le s \le t $,  associated with scheme (\ref{eq:scheme}), so that for $x \in \mathbb{R}^N $: 
\begin{align*}
& \overline{\mathcal{I}}_{t, s} (x) \equiv 
x  +\sum_{0 \le j \le d} 
\sigma_{j} 
( x  )
\Delta B_{t, s}^j  + \sum_{0 \le  j_1, j_2 \le d}
\mathcal{L}_{j_1} \sigma_{j_2} (x) 
\Delta \eta_{t, s}^{j_1 j_2}  
+  \tfrac{1}{2} \sum_{1 \le j_1 < j_2 \le d}
\bigl[\sigma_{j_1}, \sigma_{j_2} \bigr]	(x) 
\Delta \widetilde{A}^{j_1 j_2}_{t, s};   \\ 
& \widetilde{\mathcal{I}}_{t, s} (x) \equiv 
x  +\sum_{0 \le j \le d} 
\sigma_{j} 
( x  )
\Delta B_{t, s}^j  
 + \sum_{0 \le  j_1, j_2 \le d} \mathcal{L}_{j_1} \sigma_{j_2} (x) 
\Delta \eta_{t, s}^{j_1 j_2}   -  \tfrac{1}{2} \sum_{1 \le j_1 < j_2 \le d}
\bigl[\sigma_{j_1}, \sigma_{j_2} \bigr]	(x)  \Delta \widetilde{A}^{j_1 j_2}_{t, s}. 
\end{align*}
Notice that the difference between the above two integrators is in the sign of the last term. 
On the coarse grids $\mathbf{g}^{c, [\ell-1]}$, we define a discretization scheme 
$\{\bar{X}_{t}^{c, [\ell-1]}\}_{t \in \mathbf{g}^{c, [\ell-1]}}$ and 
its \emph{antithetic version} 
$\{\widetilde{X}_{t}^{c, [\ell - 1]}\}_{t \in \mathbf{g}^{c, [\ell-1]}}$ as follows. For $0 \le k \le 2^{\ell-1}-1$: 
%for $0 \le k \le 2^{\ell - 1} - 1$,  
%
\begin{align} \label{eq:X_c}   
& \bar{X}_{t_0}^{c, [\ell - 1]}  = x, \quad 
\bar{X}_{t_{k+1}}^{c, [\ell - 1]}  = \overline{\mathcal{I}}_{t_{k+1}, t_k} ( \bar{X}_{t_{k}}^{c, [\ell - 1]}); \\ 
& \widetilde{X}_{t_0}^{c, [\ell - 1]}  = x, \quad 
\widetilde{X}_{t_{k+1}}^{c, [\ell - 1]}  = \widetilde{\mathcal{I}}_{t_{k+1}, t_k} ( \widetilde{X}_{t_{k}}^{c, [\ell - 1]}). 
\label{eq:X_c_a} 
\end{align}
%\begin{align}
%	\begin{aligned} \label{eq:X_c}  
%		& \bar{X}_{t_{k+1}}^{c, [\ell - 1]} 
%		= \bar{X}_{t_{k}}^{c, [\ell - 1]}
%		+\sum_{j = 0}^d 
%		\sigma_{j} 
%		( \bar{X}_{t_{k}}^{c, [\ell - 1]} )
%		\Delta B_{t_{k+1}, t_{k}}^j  		 \\  
%		& + \sum_{ j_1, j_2 = 0 }^d 
%		\mathcal{L}_{j_1} \sigma_{j_2} (\bar{X}_{t_{k}}^{c, [\ell - 1]}) 
%		\Delta \eta_{t_{k+1}, t_{k}}^{j_1 j_2}  
%        +  \tfrac{1}{2} \sum_{1 \le j_1 < j_2 \le d}
%		\bigl[\sigma_{j_1}, \sigma_{j_2} \bigr]
%		(\bar{X}_{t_{k}}^{c, [\ell - 1]}) 
%		\Delta \widetilde{A}^{j_1 j_2}_{t_{k+1}, t_k},  
%	\end{aligned} 
%\end{align}
%% 
%with its initial point $\bar{X}_{t_0}^{c, [\ell - 1]}  = x$,  
% 
%\begin{align}
%	\begin{aligned}
%		& \widetilde{X}_{t_{k+1}}^{c, [\ell - 1]} 
%		= \widetilde{X}_{t_k}^{c, [\ell - 1]}
%		% + \sigma_{0} ( \bar{X}^{c, [\ell], (j)}_{i}) h_\ell 
%		+ \sum_{j = 0} 
%		\sigma_{j} 
%		( \widetilde{X}_{t_k}^{c, [\ell - 1]} )
%		\Delta B_{t_{k+1}, t_{k}}^j \\ 
%		& + 
%		\sum_{ j_1, j_2 = 0 }^d  \mathcal{L}_{j_1} \sigma_{j_2} 
%		(\widetilde{X}_{t_{k-1}}^{c, [\ell - 1]}) 
%		\Delta \eta_{t_{k+1}, t_{k}}^{j_1 j_2}  
%		- \tfrac{1}{2} \sum_{1 \le j_1 < j_2 \le d}
%		\bigl[\sigma_{j_1}, \sigma_{j_2} \bigr]
%		( \widetilde{X}_{t_{k}}^{c, [\ell - 1]} )  
%		\Delta \widetilde{A}^{j_1 j_2}_{t_{k+1}, t_k},  
%	\end{aligned} %\label{eq:X_c_a}
%\end{align}
%% 
%with $\widetilde{X}_{t_0}^{c, [\ell - 1]}  = x$. 
%  
Similarly, on the fine grids $\mathbf{g}^{f, [\ell]}$, we define a numerical scheme $\{\bar{X}^{f, [\ell]}_t\}_{t \in \mathbf{g}^{f, [\ell]}}$ and its antithetic version $\{\widetilde{X}^{f, [\ell]}_t\}_{t \in \mathbf{g}^{f, [\ell]}}$ as follows. For  $0 \le k \le 2^{\ell-1}-1$:
\begin{align} 
& \bar{X}^{f, [\ell]}_{t_0} = x,  \quad 
\bar{X}^{f, [\ell]}_{t_{k+1/2}}  = \overline{\mathcal{I}}_{t_{k+1/2}, t_{k}}  (\bar{X}^{f, [\ell]}_{t_{k}}),  \quad 
\bar{X}^{f, [\ell]}_{t_{k+1}}  = 
\overline{\mathcal{I}}_{t_{k+1}, t_{k+1/2}}  (\bar{X}^{f, [\ell]}_{t_{k+1/2}});
\label{eq:X_f}    \\ 
&  \widetilde{X}^{f, [\ell]}_{t_0} = x,  \quad 
\widetilde{X}^{f, [\ell]}_{t_{k+1/2}}  = \widetilde{\mathcal{I}}_{t_{k+1}, t_{k+1/2}}  (\widetilde{X}^{f, [\ell]}_{t_{k}}),  \quad 
\widetilde{X}^{f, [\ell]}_{t_{k+1}}  = 
\widetilde{\mathcal{I}}_{t_{k+1/2}, t_{k}}  (\widetilde{X}^{f, [\ell]}_{t_{k+1/2}}).  
\label{eq:X_a}  
\end{align} 
The antithetic scheme (\ref{eq:X_a}) features the following two key properties: 1. On the interval $[t_{k}, t_{k+1}]$, the Gaussian increments used in the standard discretisation (\ref{eq:X_f}) are exchanged between the first and second halfs; 2. The sign of the last term in the integrator is \emph{minus}.  We note that the first point (exchange of Gaussian increments) is featured in the antithetic truncated Milstein scheme proposed in \cite{giles} as well, but the second point (change of sign) is not. Let $\varphi: \mathbb{R}^N \to \mathbb{R}$ be some suitable test function.  In the next section
we will define an antithetic estimator based upon the weak second order scheme (\ref{eq:scheme}) and use  the identity: 
\begin{align} \label{eq:cpl_weak2} 
\mathbb{E} [\varphi (\bar{X}_T^{f, [L]})] 
= \mathbb{E} [\mathcal{P}_{0}^{\varphi}] 
+ \sum_{1 \le \ell \le L} 
\mathbb{E} 
\bigl[ \mathcal{P}_{f, \ell}^{\varphi} 
- \mathcal{P}_{c, \ell -1}^{\varphi} 
\bigr],  
\end{align}
where we have set, for $1 \le \ell \le L$:
\begin{align} 
\mathcal{P}_{f, \ell}^{\varphi} 
\equiv 
\tfrac{1}{2} \bigl( 
\varphi(\bar{X}_T^{f, [\ell]}) + \varphi(\widetilde{X}_T^{f, [\ell]})
\bigr),  
\; 
\mathcal{P}_{c, \ell-1}^{\varphi} \equiv 
\tfrac{1}{2} \bigl( 
\varphi(\bar{X}_{T}^{c, [\ell-1]}) + \varphi(\widetilde{X}_{T}^{c, [\ell-1]})
\bigr),  
\;    \mathcal{P}_{0}^{\varphi} \equiv 
\varphi(\bar{X}_T^{c, [0]}). \label{eq:p_def_l}
% \label{eq:p_def_0}
\end{align} 
We notice that 
$
\mathbb{E} [\mathcal{P}_{0}^\varphi] 
= \mathbb{E} [\mathcal{P}_{c, 0}^\varphi], \  
\mathbb{E} [\mathcal{P}_{f, \ell}^\varphi] 
= \mathbb{E} [\mathcal{P}_{c, \ell}^\varphi],  \   1 \le \ell \le L-1. 
$
For $s_1 \in \mathbf{g}^{c, [\ell-1]}$ and $s_2 \in \mathbf{g}^{f, [\ell]}$ with $\ell = 1, \ldots, L$, we define: 
\begin{align} \label{eq:ave_scheme}
\hat{X}^{c, [\ell-1]}_{s_1} = \tfrac{1}{2} \bigl( \bar{X}_{s_1}^{c, [\ell-1]} 
+ \widetilde{X}_{s_1}^{c, [\ell-1]} \bigr), 
\qquad 
\hat{X}^{f, [\ell]}_{s_2} = 
\tfrac{1}{2} \bigl(
\bar{X}_{s_2}^{f, [\ell]} 
+ \widetilde{X}_{s_2}^{f, [\ell]} 
\bigr),  
\end{align}
% % 
and study the $L^p$ bound for the coupling $\mathcal{P}_{f, \ell}^{\varphi} 
- \mathcal{P}_{c, \ell -1}^{\varphi}$.  

\begin{lemma} \label{lemma:s_bd_cpl}
Let $\varphi \in C_b^2 (\mathbb{R}^N; \mathbb{R})$ and $1 \le \ell \le L$. For any $p \ge 2$, there exist constants $C_1, C_2, C_3  > 0$ such that
\begin{align} 
\mathbb{E} \bigl[ 
\bigl( 
\mathcal{P}_{f, \ell}^{\varphi} 
- \mathcal{P}_{c, \ell -1}^{\varphi} 
\bigr)^p \bigr]  
& \le C_1 \mathbb{E} \bigl[ \| \hat{X}_{T}^{f, [\ell]} - \hat{X}_{T}^{c, [\ell-1]} \|^p  \bigr]
+ C_2 \mathbb{E} \bigl[ \| \bar{X}_{T}^{f, [\ell]}
- \widetilde{X}_{T}^{f, [\ell]} \|^{2p} \bigr] \nonumber  \\[0.1cm] 
& \qquad 
+ C_3 \mathbb{E} 
\bigl[ 
\| \bar{X}_{T}^{c, [\ell-1]} - \widetilde{X}_{T}^{c, [\ell-1]} \|^{2p} 
\bigr]. 
\label{eq:s_bd_cpl}
\end{align} 
\end{lemma}
\begin{proof}
%	Recall the definition of $\hat{X}^{f, [\ell]}$ and $\hat{X}^{c, [\ell-1]}$ in (\ref{eq:ave_scheme}). 
The second order Taylor expansion yields: $\mathcal{P}_{f, \ell}^\varphi = \varphi(\hat{X}_{T}^{f, [\ell]}) 
+ F_{1}^{\xi_1, \xi_2}$ and $\mathcal{P}_{c, \ell-1}^\varphi = \varphi(\hat{X}_{T}^{c, [\ell-1]}) + F_{2}^{\xi_3, \xi_4}$, where we have set:  
\textcolor{black}{  
\begin{align*}
F_{1}^{\xi_1, \xi_2} & \equiv \tfrac{1}{16} 
( 
\bar{X}_{T}^{f, [\ell]} - \widetilde{X}_{T}^{f, [\ell]} )^\top 
\bigl( \partial^2  \varphi (\xi_1) + \partial^2  \varphi (\xi_2) \bigr) 
( 
\bar{X}_{T}^{f, [\ell]} - \widetilde{X}_{T}^{f, [\ell]} ); \\[0.1cm] 
F_{2}^{\xi_3, \xi_4}  
& \equiv \tfrac{1}{16} 
( 
\bar{X}_{T}^{c, [\ell-1]} 
- \widetilde{X}_{T}^{c, [\ell-1]} )^\top 
\bigl( {\partial^2}  \varphi (\xi_3) 
+ {\partial^2}  \varphi (\xi_4) \bigr)
( 
\bar{X}_{T}^{c, [\ell-1]} 
- \widetilde{X}_{T}^{c, [\ell-1]}), 
\end{align*}
} 
\textcolor{black}{for some $\xi_1, \xi_2, \xi_3, \xi_4 \in \mathbb{R}^N$, where  $\partial^2\varphi(\cdot)$ is the $N\times N$ matrix of 2nd derivatives.} Thus: 
\textcolor{black}{
\begin{align} \label{eq:cpl_diff}
 & \mathcal{P}_{f, \ell}^{\varphi} 
- \mathcal{P}_{c, \ell -1}^{\varphi} 
= \partial \varphi (\xi_5)^{\top}   \bigl( 
\hat{X}_{T}^{f, [\ell]}  
- \hat{X}_{T}^{c, [\ell-1]} \bigr) 
 + F_{1}^{\xi_1, \xi_2} 
 - F_{2}^{\xi_3, \xi_4}, 
% & \qquad 
% - \tfrac{1}{16} 
% \bigl( 
% \bar{X}_{T}^{c, [\ell-1]} 
% - \widetilde{X}_{T}^{c, [\ell-1]} \bigr)^\top 
% \bigl( {\partial^2}  \varphi (\xi_3) 
% + {\partial^2}  \varphi (\xi_4) \bigr)
% \bigl( 
% \bar{X}_{T}^{c, [\ell-1]} 
% - \widetilde{X}_{T}^{c, [\ell-1]} \bigr), \nonumber  
\end{align}
}
\textcolor{black}{for some $\xi_5 \in \mathbb{R}^N$, where $\partial\varphi(\cdot)$ is the $N\times 1$ vector of 1st derivatives. } Due to the boundedness of the test function $\varphi$ and \textcolor{black}{the standard inequality given in (\ref{eq:bd})}, we conclude from (\ref{eq:cpl_diff}). 
%	The rest of the proof is the same as that of \cite[Lemma 2.2]{ml_anti}.  
\end{proof} 

Our objective is to derive bounds for each term in the R.H.S.~of (\ref{eq:s_bd_cpl}) over a coarse time step $\Delta_{\ell-1}$. 
For the first term, we have the following result: 
%  with proof given in Section \ref{app:strong_err_coupl}.  
% 
\begin{theorem}
\label{thm:strong_err_coupl}
Let $1 \le \ell \le L$. For all $p \ge 2$, there exists a constant $C$ such that
\begin{align*}
\mathbb{E} \bigl[ \, 
\max_{t \in \mathbf{g}^{c, [\ell-1]}} 
\| 
\hat{X}^{f, [\ell]}_{t} - \hat{X}^{c, [\ell-1]}_{t} 
\|^p
\bigr] \le C \, \Delta_{\ell-1}^p. 
\end{align*}
\end{theorem}
\begin{proof}
Let $0 \le n \le 2^{\ell- 1}$ and  $ \hat{\mathcal{S}}_n \equiv  \mathbb{E} \bigl[ \max_{0 \le k \le n}
\| \hat{X}_{t_k}^{f, [\ell]} - \hat{X}_{t_k}^{c, [\ell-1]} \|^p \bigr]. 
$ 
It holds that for any $p \ge 2$, there exists a constant $C_p > 0$ such that: 
\begin{align} \label{eq:ineq_S}
\hat{\mathcal{S}}_n \le C_p \sum_{1 \le j \le N} 
\mathbb{E} 
\bigl[ \max_{0 \le k \le n} 
| 
\hat{X}_{t_k}^{f, [\ell], j} - \hat{X}_{t_k}^{c, [\ell-1], j} 
|^p  
\bigr]. 
\end{align}
Then, it suffices to show that there exists a constant $C > 0$ such that:  
\begin{align} \label{eq:bd_cpl}
\mathbb{E} 
\bigl[ \max_{0 \le k \le n} |  
\hat{X}_{t_k}^{f, [\ell], j} 
- \hat{X}_{t_k}^{c, [\ell-1], j} 
|^p  \bigr] 
\le 
C \Bigl(\Delta_{\ell-1}^p + \Delta_{\ell-1} 
\sum_{0 \le k \le n-1} \hat{\mathcal{S}}_k \Bigr), 
\end{align}
which leads to the desired result by applying the discrete Gr\"onwall inequality to (\ref{eq:ineq_S}). Recursive application of  (\ref{eq:X_hat_f}) and (\ref{eq:X_hat_c}) given in Lemmas \ref{lemma:f_anti_diff} and \ref{lemma:coare}  respectively yields: 
\allowdisplaybreaks
\begin{align} 
& \hat{X}^{f, [\ell], j}_{t_k} - \hat{X}^{c, [\ell-1], j}_{t_k}
= 
\sum_{\substack{0 \le i \le k-1 \\  0 \le m \le d} } 
\bigl( \sigma_{m}^j (\hat{X}^{f, [\ell]}_{t_i})  - \sigma_{m}^j (\hat{X}^{c, [\ell-1]}_{t_i}) \bigr) 
\Delta B_{t_{i+1}, t_i}^{m}  \label{eq:key_diff} \\ 
& \! \!  \! \! 
+ \!  \! \sum_{\substack{ 0 \le i \le k-1 \\ 1 \le m_1, m_2 \le d}}
\! \!  \!  \!  \! 
\bigl( \mathcal{L}_{m_1} \sigma_{m_2}^j 
(\hat{X}^{f, [\ell]}_{t_i}) 
- \mathcal{L}_{m_1} \sigma_{m_2}^j 
(\hat{X}^{c, [\ell-1]}_{t_i}) \bigr) 
\Delta \eta_{t_{i+1}, t_i}^{m_1 m_2} 
+ \sum_{0 \le i \le k-1}  
\bigl( \hat{\mathcal{M}}_{t_{i+1}, t_i}^{j}  + 
\hat{\mathcal{N}}_{t_{i+1}, t_i}^{j}   \bigr),  \nonumber 
%		& \qquad + \sum_{0 \le i \le k-1}
%		\bigl\{ 
%		\hat{\mathcal{M}}_{t_{i+1}, t_i}^{f,j} 
%		+ \hat{\mathcal{N}}_{t_{i+1}, t_i}^{f, j}
%		+ \hat{\mathcal{M}}_{t_{i+1}, t_i}^{c,j}
%		+ \hat{\mathcal{N}}_{t_{i+1}, t_i}^{c,j}  
%		\bigr\},  
\end{align}
where the remainder terms are such that
$
\mathbb{E} 
\bigl[ 
\hat{\mathcal{M}}^{j}_{t_{i+1}, t_i} 
| \mathcal{F}_{t_{i}} 
\bigr] = 0,  \   0 \le i \le 2^{\ell-1} - 1,
$ 
and for any $p \ge 2$ there exist constants $C_1, C_2 > 0$ such that 
$ 
\max_{0 \le i \le 2^{\ell-1} - 1} 
\mathbb{E} \bigl[ | \hat{\mathcal{M}}^{j}_{t_{i+1}, t_i}|^p \bigr] 
\le C_1 \Delta_{\ell-1}^{3p/2}$ and   
$\max_{0 \le i \le 2^{\ell-1} - 1} 
\mathbb{E} \bigl[ | \hat{\mathcal{N}}^{j}_{t_{i+1}, t_i} |^p \bigr] 
\le C_2 \Delta_{\ell -1}^{2p}. 
$   
Given (\ref{eq:key_diff}), the bound (\ref{eq:bd_cpl})  holds by following the same argument as in the proof of \cite[Theorem 4.10]{ml_anti}, and we conclude. 
\end{proof}
Also, we have that, for any $\textcolor{black}{p \ge 1}$ there exist constants $C_1, C_2 > 0$ such that: 
\begin{align} 
\begin{aligned} \label{eq:s_bd} 
& \mathbb{E} \bigl[ \, \max_{t \in \mathbf{g}^{c, [\ell-1]}}  
\| \bar{X}_{t}^{f, [\ell]} - \widetilde{X}_{t}^{f, [\ell]} \|^{2p} \bigr] \le C_1 \Delta_{\ell-1}^p;  \\ 
& \mathbb{E} 
\bigl[ \, \max_{t \in \mathbf{g}^{c, [\ell-1]}}
\| \bar{X}_{t}^{c, [\ell-1]}
- \widetilde{X}_{t}^{c, [\ell-1]} 
\|^{2p} \bigr] 
\le C_2 \Delta_{\ell-1}^p,  
\end{aligned}
\end{align} 
\textcolor{black}{which are obtained from the strong convergence rate of scheme (\ref{eq:scheme}) and the same argument used in the proof of \cite[Lemma 4.6]{ml_anti}.} 
Hence, from Theorem \ref{thm:strong_err_coupl}, Lemma \ref{lemma:s_bd_cpl} and (\ref{eq:s_bd}), we obtain the following result.
\begin{corollary}
\label{cor:fn_bd}
Let $\varphi \in C_b^2 (\mathbb{R}^N; \mathbb{R})$ and $1 \le \ell \le L$. 
For any $p \ge 2$ there exists constant  $C > 0$ such that
$ 
\mathbb{E} [ 
( 
\mathcal{P}_{f, \ell}^{\varphi} 
- \mathcal{P}_{c, \ell -1}^{\varphi} 
)^p ]
\le C \, \Delta_{\ell-1}^p. 
$ 
\end{corollary}
\begin{remark}
The AMLMC estimator under scheme (\ref{eq:scheme}) is designed to have \textcolor{black}{four different integrations}, as given in (\ref{eq:X_c})-(\ref{eq:X_a}), while the antithetic estimator under the truncated Milstein scheme \cite{ml_anti} uses \textcolor{black}{three types of integrators} without the antithetic coarse approximation $\widetilde{X}^{c, [\ell-1]}$. In the case of (\ref{eq:scheme}), use of only \textcolor{black}{three integrators} would lead to no improvement in strong convergence due to the presence of the term involving $\Delta \widetilde{A}_{t_{k+1}, t_k}^{j_1 j_2}$ with a size of $\mathcal{O} (\Delta_\ell)$. $\widetilde{X}^{c, [\ell-1]}$ is exploited to deal with the above $\mathcal{O} (\Delta_\ell)$-term and obtain the higher rate of strong convergence (Theorem \ref{thm:strong_err_coupl}).
\end{remark}
\begin{remark} 
\cite{nv_anti} constructed an AMLMC method based on the Ninomiya-Victoir (N-V) scheme \cite{nv}, an alternative scheme of weak error 2. They showed that the strong error of the N-V scheme is 1 and then improved it with the technique of the antithetic multilevel estimator. The advantages of the proposed AMLMC based on (\ref{eq:scheme}) against that of the N-V scheme are summarized as follows: 
(i) Scheme (\ref{eq:scheme}) is always explicit while the N-V is a semi-closed scheme in the sense that it requires solving ODEs defined via the SDE coefficients and their solvability depends on the definition of coefficients; 
(ii) Our antithetic scheme uses \textcolor{black}{four different integrators} (\ref{eq:X_c})-(\ref{eq:X_a}), while the antithetic estimator with the N-V scheme uses six \textcolor{black}{integrators}; 
(iii) Our (\ref{eq:scheme}) scheme is designed to be locally non-degenerate for both elliptic/hypo-elliptic settings (Section \ref{sec:setting}) as we explained in Section \ref{sec:weak2}. Such a non-degenerate scheme is beneficial for the filtering problem as we described in Section \ref{sec:intro}.       
\end{remark}
% 
% \subsection{Proof of Theorem \ref{thm:strong_err_coupl}} \label{app:strong_err_coupl}
% 
%\begin{remark}
%	The improvement in the variance bound for AW2 in Theorem \ref{thm:m_bd_sd} comes from the inclusion of higher-order stochastic Taylor expansion terms from the drift coefficient $\sigma_0$ in scheme (\ref{eq:scheme}) (note that the \ref{eq:t_mil} scheme contains higher order terms from the diffusion coefficients only).  For instance, the error bound for the truncated Milstein scheme is affected by the term $\mathcal{L}_0 \sigma_0 (\cdot) \Delta_{\ell}^2/2$ being independent of the small diffusion parameter $\mu$.
%while this is not the case for the weak second order scheme due to the inclusion of the term in the scheme. The details can be found in the proofs in Appendix \ref{app:m_bd_sd} in Supplementary Material \cite{supp}. 
%\end{remark}
% 
\subsection{AMLMC for forward problem}
\label{sec:mlmc_for}
In order to estimate $\mathbb{E}\left[\varphi (X_T)\right]$, one simply needs to sample the systems \eqref{eq:X_c}-\eqref{eq:X_a} using the same source of randomness (i.e.~the same Brownian motion and Gaussian variates) as implied in \eqref{eq:X_c}-\eqref{eq:X_a}. We will sample these afore-mentioned systems multiple times (independently) so will use an argument `$(i)$' to indicate the $i^{th}$-sample.  For instance,  from \eqref{eq:X_c},  we will write 
$\bar{X}_{t_k}^{c,[\ell]}(i)$ for the $i^{th}$-sample associated to recursion \eqref{eq:X_c} where the associated Brownian motion and Gaussians variates have been generated anew for each sample.  Similarly, in the context of \eqref{eq:p_def_l} 
%-\eqref{eq:p_def_0} 
we will write $\mathcal{P}^{\varphi}_{f,\ell}(i)$,  $\mathcal{P}^{\varphi}_{c,\ell-1}(i)$ and 
$\mathcal{P}^{\varphi}_{0}(i)$.

The \textcolor{black}{AMLMC} procedure is as follows.  We first set $L$ and the sample sizes $M_0,\ldots,M_L$ to be used at each pair of levels; we will state below how this can be done.  Then one can follow the approach in Algorithm \ref{alg:mlmc}. The new \textcolor{black}{AMLMC} estimator is given in \eqref{eq:ml_est} that is contained in Algorithm \ref{alg:mlmc} and can be computed using any test function of interest when the underlying quantity $\mathbb{E}\left[\varphi (X_T)\right]$ is well defined.

To specify $L$ and $M_0,\dots,M_L$ one can appeal to the results of Theorem \ref{thm:strong_err_coupl},
Corollary~\ref{cor:fn_bd}, as well as the weak error of the scheme (\ref{eq:scheme}) and follow standard computations in MLMC (e.g.~\cite{giles1}). That is,  when considering the MSE, 
$
\mathbb{E} \bigl[\bigl(\widehat{\mathbb{E}[\varphi(X_T)]}-\mathbb{E}[\varphi(X_T)]
\bigr)^2 \bigr],
$
then under the assumptions made above, one has an upper-bound on the MSE as 
$
\mathcal{O}\bigl(
\sum_{0 \le \ell \le L}  \Delta_{\ell}^2/{M_{\ell}} + \Delta_L^{4}
\bigr).
$
Therefore, for $\epsilon>0$ given, one can achieve a MSE of $\mathcal{O}(\epsilon^2)$ by choosing $L=\mathcal{O}(\log(\epsilon^{-{1}/{2}}))$ and $M_{\ell}=\mathcal{O}(\epsilon^{-2}\Delta_{\ell}^{3/2})$. The cost to achieve
this MSE is 
$
\textcolor{black}{\sum_{0 \le \ell \le L} \Delta_{\ell}^{-1}M_{\ell} = \mathcal{O}(\epsilon^{-2})}
$
which is the best possible using stochastic Monte Carlo methods and was also obtained in \cite{ml_anti}. In most practical simulations, one generally sets $L$ as on standard computing equipment it is not feasible to generate beyond $L=10$ and this determines $\epsilon$.  Therefore, as the bias (weak error) of this method is $\mathcal{O}(\Delta_L^2)$, versus $\mathcal{O}(\Delta_L)$ in the antithetic Milstein method in \cite{ml_anti}, one might expect to see benefits for $L$'s that are used in practice.  We consider this in Section \ref{sec:numerics}.

\begin{algorithm}[h]
\begin{enumerate}
\item{Input $L\geq 1$ and $M_0,\dots,M_L$.  Set $\ell=0$ and go to 2..}
\item{For $i=1,\dots,M_0$ independently simulate \eqref{eq:X_c} to produce $\bar{X}_{T}^{c,[0]}(1),\dots,\bar{X}_{T}^{c,[0]}(M_0)$. Set $\ell=\ell+1$ and go to 3..}
\item{
For $i=1,\ldots,M_{\ell}$,  independently simulate \eqref{eq:X_c}-\eqref{eq:X_a} to produce
$ \textstyle 
\{ \bar{X}_{T}^{c,[\ell-1]}(i) \}_{i = 1}^{M_\ell}$, 
$\{ \widetilde{X}_{T}^{c,[\ell-1]}(i) \}_{i = 1}^{M_\ell}$, 
$ \{ \bar{X}_{T}^{f,[\ell-1]}(i) \}_{i = 1}^{M_\ell}$,  
$\{ \widetilde{X}_{T}^{f,[\ell-1]}(i) \}_{i = 1}^{M_\ell}$. 
%\bigl( \bar{X}_{T}^{c,[\ell-1]}(1),\dots,\bar{X}_{T}^{c,[\ell-1]}(M_{\ell}) \bigr) \quad
%\bigl( \widetilde{X}_{T}^{c,[\ell-1]}(1),\dots,\widetilde{X}_{T}^{c,[\ell-1]}(M_{\ell}) \bigr)
%$$
%and
%$$
%\bigl(\bar{X}_{T}^{f,[\ell-1]}(1),\dots,\bar{X}_{T}^{f,[\ell-1]}(M_{\ell}) \bigr) \quad
%\bigl(\widetilde{X}_{T}^{f,[\ell-1]}(1),\dots,\widetilde{X}_{T}^{f,[\ell-1]}(M_{\ell}) \bigr).
If $\ell\leq L-1$, set $\ell=\ell+1$ go to the start of 3.~otherwise go to 4..}
\item{Compute the MLMC estimator:
\begin{align}
	\widehat{\mathbb{E}[\varphi(X_T)]} := \mathcal{P}^{\varphi,M_0}_{0} + \sum_{1 \le \ell \le L} \bigl\{
	\mathcal{P}^{\varphi,M_{\ell}}_{f,\ell} - 
	\mathcal{P}^{\varphi,M_{\ell}}_{c,\ell-1}
	\bigr\}\label{eq:ml_est}
\end{align}
where
$ 
\mathcal{P}^{\varphi,M_0}_{0}:= \tfrac{1}{M_0} \sum_{1 \le i \le M_0}
\mathcal{P}^{\varphi}_{0}(i), \ \  \mathcal{P}^{\varphi,M_{\ell}}_{f,\ell} :=  \tfrac{1}{M_{\ell}} \sum_{1 \le i \le M_{\ell}}\mathcal{P}^{\varphi}_{f,\ell}(i),  \ \   
\mathcal{P}^{\varphi,M_{\ell}}_{c,\ell-1} :=  \tfrac{1}{M_{\ell}} \sum_{1 \le i \le M_{\ell}}\mathcal{P}^{\varphi}_{c,\ell-1}(i). 
$ 
Return \eqref{eq:ml_est} and stop.}
\end{enumerate}
\caption{AMLMC using the weak second order scheme \eqref{eq:scheme}.}
\label{alg:mlmc}
\end{algorithm}
\section{Application to filtering}\label{sec:filtering}

\subsection{State-space model}

We consider a sequence of observations obtained sequentially and at unit times, $Y_1,Y_2,\dots$, $Y_k \in \textcolor{black}{\mathbb{R}^N}$, $k\in\mathbb{N}$. The assumption of unit times is mainly for simplicity of notation and any time grid could be considered.  Associated to this sequence is an unobserved diffusion process exactly of the type \eqref{eq:diff}.  For the data, we shall assume that, at any time $k\in\mathbb{N}$, $Y_k$ has a (bounded) positive probability density that depends only on the position, $X_k$, of the diffusion process at time $k$ and is denoted $g(x_k,y_k)$. We denote the transition kernel of the diffusion process over a unit time and starting at $z\in\mathbb{R}^N$ as $Q(z,\cdot)$, for instance
$
\textstyle 
\mathbb{E}[\varphi(X_1)] = \int_{\mathbb{R}^N}\varphi(x_1)Q(x,dx_1), 
$ 
where the expectation on the R.H.S. is w.r.t.~the law of the diffusion \eqref{eq:diff},  which we recall starts at $x\in\mathbb{R}^N$, and $\varphi:\mathbb{R}^N\rightarrow\mathbb{R}$ is bounded, measurable (the collection of such functions is denoted $\mathcal{B}_b(\mathbb{R}^N)$).

The object of interest is the filtering distribution.  For any $k\in\mathbb{N}$ we define the filtering expectation:
\begin{equation}
\label{eq:filt}
\pi_k(\varphi) := \tfrac{\mathbb{E}\bigl[\varphi(X_k)\bigl\{\prod_{p=1}^k g(X_p,y_p)\bigr\}\bigr]}{\mathbb{E}
\bigl[\bigl\{\prod_{p=1}^k g(X_p,y_p)\bigr\}\bigr]}.
\end{equation}
Note that the fact that $\varphi$ and $g(\cdot,y)$ are bounded (for any $y \in \textcolor{black}{\mathbb{R}^N} $) ensure that the filter is well-defined, but these assumptions are not needed in general -- again we seek to simplify the discussion.  We will compute a numerical approximation of \eqref{eq:filt} sequentially in time, as an exact computation is seldom possible.

In practice we often cannot \textcolor{black}{(i)} simulate from $Q(z,\cdot)$ and/or we may not have an 
 \textcolor{black}{(ii)} 
explicit expression for the density of $Q(z,\cdot)$ 
%(if it even exists) 
%and even if it does there may not be an unbiased estimate of the density.  
or \textcolor{black}{(iii)} an unbiased estimate of such density.
One of the afore-mentioned properties \textcolor{black}{(i)-(iii)} is needed in order to deploy numerical methods which are
used in the approximation of the filter \eqref{eq:filt} in continuous time (see e.g.~\cite{mlpf} for an explanation). 
Therefore we consider time discretization via
the weak second order method \eqref{eq:scheme}, with step-size $\Delta_{\ell}=2^{-\ell}$.  Now,  for any starting point $z\in\mathbb{R}^N$ and ending at a time 1 we denote the time discretised transition kernel as $Q^{[\ell]}(z,\cdot)$,  for instance,
$ \textstyle
\mathbb{E}^{[\ell]}[\varphi(\bar{X}_1)] = \int_{\mathbb{R}^N}\varphi(x_1)Q^{[\ell]}(x,dx_1), 
$
where we have modified the notation of the expectation operator to $\mathbb{E}^{[\ell]}[\cdot]$ to emphasize dependence on the discretization level.
We  consider the approximation of the time discretised filter, $k\in\mathbb{N}$:
\begin{equation}\label{eq:disc_filt}
\pi_k^{[\ell]}(\varphi) := \tfrac{\mathbb{E}^{[\ell]}\left[\varphi(X_k)\left\{\prod_{p=1}^k g(X_p,y_p)\right\}\right]}{\mathbb{E}^{[\ell]}\left[\left\{\prod_{p=1}^k g(X_p,y_p)\right\}\right]}.
\end{equation}
Note, to clarify,  the R.H.S.~of the above equation can be alternatively written as:
$$
\tfrac{\int_{\mathbb{R}^{Nk}}
\varphi(x_k)\left\{\prod_{p=1}^k g(x_p,y_p)\right\}\prod_{p=1}^k Q^{[\ell]}(x_{p-1},dx_p)
}{\int_{\mathbb{R}^{Nk}}\left\{\prod_{p=1}^k g(x_p,y_p)\right\}\prod_{p=1}^k Q^{[\ell]}(x_{p-1},dx_p)}
$$
where $x_0=x$. Even with time discretization, one still needs to resort to numerical methods to approximate \eqref{eq:disc_filt}.

\subsection{Multilevel particle filters}

Our objective is now to approximate the time discretised filter \eqref{eq:disc_filt}.  We start with the ordinary
particle filter (PF) which can do exactly the former task and is described in Algorithm \ref{alg:pf}. This algorithm presents the most standard and well-known PF with several possible extensions.  Also note that the estimates of the filter, in equation \eqref{eq:pf_est} of Algorithm \ref{alg:pf}, are typically returned recursively in time.

The PF on its own is typically much less efficient than using a multilevel version,  which has been developed and extended in several works; see 
e.g.~\cite{mlpf,mlpf1,levymlpf,ub_pf,anti_mlpf} and \cite{ml_rev} for a review.  We describe the method of \cite{mlpf}, except replacing the Euler-Maruyama discretization with the weak second order scheme.  The basic idea is based upon the identity:
\begin{align}\label{eq:ml_filt_id}
\pi^{[L]}_k(\varphi) = \pi^{[0]}_k(\varphi) + \sum_{1 \le \ell \le L} \bigl\{\pi^{[\ell]}_k(\varphi)-\pi^{[\ell-1]}_k(\varphi)\bigr\}.
\end{align}
We remark that on the R.H.S.~of \eqref{eq:ml_filt_id} one need not start at level 0, but we adopt this choice for ease of exposition.
The idea is to use the PF to recursively approximate $\pi^{[0]}_k(\varphi)$ and then to use a coupled particle filter (CPF) for the approximation of $\pi^{[\ell]}_k(\varphi)-\pi^{[\ell-1]}_k(\varphi)$, independently for each index $\ell$.
The coupling is described in Algorithm \ref{alg:max_couple}  and then the CPF is given in Algorithm \ref{alg:cpf}, which are presented in Section \ref{sec:supp_algo} in the Appendix. 

Algorithm \ref{alg:max_couple} presents a way to simulate a maximal coupling of two positive probability mass functions with the same support. It allows one to couple the resampling operation across two different levels of discretization as is done for a single level in Algorithm \ref{alg:pf}. This is then incorporated in Algorithm \ref{alg:cpf} which provides a way to approximate $\pi^{[\ell]}_k(\varphi)-\pi^{[\ell-1]}_k(\varphi)$ recursively in time. 

The overall multilevel Particle Filter (MLPF) can be summarized as follows, given $L$ the maximum level and the number of samples $M_0,\dots,M_L$; we show how these parameters can be chosen below.
\begin{enumerate}
\item{Run Algorithm \ref{alg:pf} at level $\ell=0$ with $M_0$ samples.}
\item{Independently of 1.~for $\ell=1,\dots,L$, independently run Algorithm \ref{alg:cpf} in the Appendix with $M_{\ell}$ samples.}
\end{enumerate}
Based on this process, a biased approximation of $\pi_k(\varphi)$ is then
$$
\widehat{\pi_k(\varphi)} := \pi^{[0],M_0}_k(\varphi) + \sum_{1 \le \ell \le L}
\bigl\{\pi^{[\ell],M_{\ell}}_k(\varphi)-\pi^{[\ell-1],M_{\ell}}_k(\varphi)\bigr\},
$$
\textcolor{black}{where $\pi_k^{[\ell],M}(\varphi)$ is the PF estimate of $\pi_k^{[\ell]}(\varphi)$ with the number of particles $M$ specifically given in (\ref{eq:pf_est})}. 
The bias of this approximation is from the discretization level $L$ and the bias of the PF/CPF approximation, 
e.g.~that in general, 
$
\mathbb{E} \bigl[\pi^{[\ell],M_{\ell}}_k(\varphi)-\pi^{[\ell-1],M_{\ell}}_k(\varphi) \bigr] 
\neq \pi^{[\ell]}_k(\varphi)-\pi^{[\ell-1]}_k(\varphi)
$,  where $\mathbb{E}$ is used to denote the expectation w.r.t.~the probability law used in generating our estimators.
Now, if one combines the theory in \cite{iguchi2} for the weak error, the strong error result in Proposition \ref{prop:s_rate} and the results in \cite{mlpf} one can consider the MSE,
$
\mathbb{E}\bigl[\bigl(\widehat{\pi_k(\varphi)}-\pi_k(\varphi)\bigr)^2\bigr]
$. 
Under the assumptions in the current paper and in \cite{mlpf} it can be proved that the MSE has an upper-bound which is:
\begin{equation} \label{eq:mlpf_bound}
\mathcal{O}\bigl(\sum_{0 \le \ell \le L} \Delta_\ell^{{1}/{2}}/{M_\ell} + \Delta_L^{4}\bigr).
\end{equation}
We do not prove this bound as it is a fairly trivial application of the results in the afore-mentioned papers. The exponent of
$\Delta_\ell$, in the summand, is $1/2$ and this reduction of the strong error of Euler-Maruyama is due to the resampling mechanism that has been employed; we do not know of any general method that can maintain the strong error rate. We also remark that there is an additional additive term on the R.H.S., but this term is much smaller than the term given above, so we need not consider it.  
Using the standard approach that has been adopted in MLMC (i.e.~as discussed in Section \ref{sec:mlmc_for}) one can show that for $\epsilon>0$ given, setting $L=\mathcal{O}(\log(\epsilon^{-{1}/{2}}))$, $M_{\ell}=\epsilon^{-2}\Delta_{\ell}^{3/4}\Delta_L^{-1/4}$ gives a MSE of $\mathcal{O}(\epsilon^2)$ for a cost (per time step $k$) of $\mathcal{O}(\epsilon^{-2.25})$. This is lower than the cost of the approach in \cite{mlpf} due to the increased weak error relative to the Euler-Maruyama discretization used in \cite{mlpf}.

In the recent work of \cite{anti_mlpf}, the authors show how to use the antithetic Milstein scheme within the context of the MLPF; we abbreviate to AMMLPF (antithetic Milstein MLPF).  They show empirically that to achieve a MSE (associated to their estimator) of  $\mathcal{O}(\epsilon^2)$ there is a cost (per time step $k$) of $\mathcal{O}
\bigl( \epsilon^{-2}\log(\epsilon)^2 \bigr)$. The objective now is to show how our new antithetic MLMC method can be extended to MLPFs.
As in the case of MLMC, we expect for this new method the error-cost calculation to be of the same order as the AMMLPF,  but when using smaller $L$, as would be adopted in practice, that improvements are seen in simulations,  due to the decreased weak error.

\begin{algorithm}[h]
\begin{enumerate}
\item{Input: level of discretization $\ell\in\mathbb{N}_0$, final time $T\in\mathbb{N}$ and number of samples $M$. Set $\bar{X}_0^{[\ell]}(i)=x$, $i=1,\dots,M$ and $k=1$. Go to 2..}
\item{Sampling: For $i=1,\dots,M$, simulate $\bar{X}_k^{[\ell]}(i)|\bar{x}_{k-1}^{[\ell]}(i)$ using the dynamics \eqref{eq:scheme} up-to time 1, with starting point $\bar{x}_{k-1}^{[\ell]}(i)$ and step-size $\Delta_{\ell}$. Go to 3..}
\item{Resampling: For $i=1,\dots,M$ compute:
$ \textstyle
w_k^{[\ell]}(i) := \tfrac{g(\bar{X}_{k}^{[\ell]}(i),y_k)}{\sum_{j=1}^M g(\bar{X}_{k}^{[\ell]}(j),y_k)}.
$
For any $\varphi\in\mathcal{B}_b(\mathbb{R}^{N})$ we have the estimate:
\begin{equation}\label{eq:pf_est}
	\pi_k^{[\ell],M}(\varphi) := \sum_{1 \le i \le M} w_k^{[\ell]}(i)\varphi(\bar{X}_k^{[\ell]}(i)).
\end{equation}
For $i=1,\dots,M$ sample an index $j(i)\in\{1,\dots,M\}$ using the probability mass function $w_k^{[\ell]}(\cdot)$ and set $\check{X}_k^{[\ell]}(i)=\textcolor{black}{\bar{X}_k^{[\ell]}(j(i))}$.
For $i=1,\dots,M$, set $\bar{X}_k^{[\ell]}(i)=\check{X}_k^{[\ell]}(i)$, $k=k+1$, if $k=T+1$ go to 4.~otherwise go to 2..}
\item{Return the estimates $\pi_1^{[\ell],M}(\varphi),\dots,\pi_T^{[\ell],M}(\varphi)$ from \eqref{eq:pf_est}.}
\end{enumerate}
\caption{Particle Filter using the weak second order scheme \eqref{eq:scheme}. The algorithm is stopped at a time $T$, but need not be.}
\label{alg:pf}
\end{algorithm}

\subsection{New multilevel particle filter}

Our new MLPF,  which we shall call the antithetic multilevel Particle Filter (AMLPF), is similar to the approach that was illustrated in the previous section. At level 0, we shall use a PF to approximate $\pi_k^{[0]}(\varphi)$. 
To approximate the differences $\pi_k^{[\ell]}(\varphi)-\pi_k^{[\ell-1]}(\varphi)$ we shall use a combination of the 
antithetic MLMC weak second order scheme of Section \ref{sec:weak_mlmc_scheme}, which will be the `sampling' part of a PF and a type of `coupling' for the `resampling step'.  As we have already introduced the former, we introduce the latter as Algorithm \ref{alg:max_couple_four} in Section \ref{sec:supp_algo} in the Appendix.  As has been commented by \cite{ub_pf} in the context of coupling two probability mass functions (as in Algorithm \ref{alg:max_couple}) there is nothing that is optimal about using Algorithm \ref{alg:max_couple_four}.  Indeed, when used as part of a MLPF, we expect just as in the case of Algorithm \ref{alg:max_couple} when used for Algorithm \ref{alg:cpf}, the strong error rate from the forward problem is reduced by a factor of two; see \eqref{eq:mlpf_bound}.  It remains an open problem to find a general coupling method which can maintain the forward error rate (as was the case in \cite{ballesio} in dimension 1 only) and a linear complexity in terms of the samples $M$.

Given Algorithm \ref{alg:max_couple_four}, we are now in a position to give our new coupled particle filter in 
Algorithm \ref{alg:new_cpf}.  Just as in the previous section, the AMLPF can be summarized as follows, given $L$ the maximum level and the number of samples $M_0,\dots,M_L$; we show how these parameters can be chosen below.
\begin{enumerate}
\item{Run Algorithm \ref{alg:pf} at level $\ell=0$ with $M_0$ samples.}
\item{Independently of 1.~for $\ell=1,\dots,L$, independently run Algorithm \ref{alg:new_cpf} with $M_{\ell}$ samples.}
\end{enumerate}
Thus our new approximation of $\pi_k(\varphi)$ is:
$$
\widetilde{\pi_k(\varphi)} :=\pi^{[0],M_0}_k(\varphi) + \sum_{1 \le \ell \le L} 
\bigl\{\hat{\pi}^{[\ell],M_{\ell}}_k(\varphi)-\hat{\pi}^{[\ell-1],M_{l}}_k(\varphi)\bigr\}.
$$
where we recall that $\hat{\pi}^{[\ell],M_{\ell}}_k(\varphi)-\hat{\pi}^{[\ell-1],M_{l}}_k(\varphi)$ is given in 
\eqref{eq:new_cpf_est} in Algorithm \ref{alg:new_cpf}.

We can again consider the MSE 
$
\mathbb{E}\bigl[\bigl(\widetilde{\pi_k(\varphi)}-\pi_k(\varphi)\bigr)^2\bigr].
$
As noted in \cite{anti_mlpf}, which considers the AMMLPF,  although it is fairly easy to establish a bound
on the R.H.S.~which is of the type (up-to some other terms which are smaller)
$
\mathcal{O} (\sum_{0 \le \ell \le L} {\Delta_l^{\nu}}/{M_l} + \Delta_L^{4}), 
$
obtaining the value of $\nu$ that is observed in simulation is not easy to achieve with the current proof method that has been adopted in \cite{mlpf,anti_mlpf}.  As a result, we do not give a theoretical analysis in this paper.  However, as we shall see in Section \ref{sec:numerics}, it appears that the correct value of $\nu=1$ and hence we use this as our guideline to choose $L,M_0,\dots,M_L$.  Following the arguments that have been used previously, for $\epsilon>0$ given, setting $L=\mathcal{O}(\log(\epsilon^{-{1}/{2}}))$, $M_{\ell}=\epsilon^{-2}\Delta_{\ell}L$ gives a MSE of $\mathcal{O}(\epsilon^2)$ for a cost (per time step $k$) of $\mathcal{O} \bigl( \epsilon^{-2}\log(\epsilon)^2 \bigr)$.  

%\begin{algorithm}[h]
%	\begin{enumerate}
%		\item{Input $M\in\mathbb{N}$ the cardinality of the state-space and four positive probability mass functions
%			$W^1(1),\dots,W^1(M),\dots,W^4(1),\dots,W^4(M)$ on $\{1,\dots,M\}$.  Go to 2..}
%		\item{Sample $U\sim\mathcal{U}_{[0,1]}$. If $U<\sum_{i=1}^M\min\{W^1(i),W^2(i),W^3(i),W^4(i)\}$ go to 3.~otherwise go to 4..}
%		\item{Sample an index $i_1$ using the probability mass function
%			$$
%			\mathbb{P}(i_1) = \tfrac{\min\{W^1(i_1),W^2(i_1),W^3(i_1),W^4(i_1)\}}{\sum_{j_1=1}^M\min\{W^1(j_1),W^2(j_1),W^3(j_1),W^4(j_1)\}}
%			$$
%			set $i_4=i_3=i_2=i_1$ and go to 5..}
%		\item{Sample the indices $(i_1,\dots,i_4)$ using the probability mass function
%			$$
%			\mathbb{P}(i_1,\dots,i_4) = \prod_{1 \le j \le 4} \tfrac{W^j(i_j)-\min\{W^1(i_j),W^2(i_j),W^3(i_j),W^4(i_j)\}}{1-\sum_{j_1=1}^M\min\{W^1(j_1),W^2(j_2),W^3(j_1),W^4(j_1)\}}
%			$$
%			and go to 5..}
%		\item{Return the indices $(i_1,\dots,i_4)\in\{1,\dots,M\}^4$.}
%	\end{enumerate}
%	\caption{Simulating a Four-Way Coupling.}
%	\label{alg:max_couple_four}
%\end{algorithm}

\begin{algorithm}[h]
\begin{enumerate}
\item{Input: level of discretization $\ell\in\mathbb{N}$, final time $T\in\mathbb{N}$ and number of samples $M$. Set $\bar{X}_0^{c,[\ell-1]}(i)=\widetilde{X}_0^{c,[\ell-1]}(i)=\bar{X}_0^{f,[\ell]}(i)=\widetilde{X}_0^{f,[\ell]}(i)=x$, $i=1,\dots,M$ and $k=1$. Go to 2..}
\item{Sampling: For $i=1,\dots,M$, simulate 
$$
\bigl(\bar{X}_k^{c,[\ell-1]}(i),\widetilde{X}_k^{c,[\ell-1]}(i),\bar{X}_k^{f,[\ell]}(i),
\widetilde{X}_k^{f,[\ell]}(i)\bigr)|
\bigl(\bar{x}_{k-1}^{c,[\ell-1]}(i),\widetilde{x}_{k-1}^{c,[\ell-1]}(i),\bar{x}_{k-1}^{f,[\ell]}(i),
\widetilde{x}_{k-1}^{f,[\ell]}(i) \bigr)
$$ 
using the coupled dynamics \eqref{eq:X_c}-\eqref{eq:X_a} up-to time 1, with:
\begin{itemize}
	\item{starting point $\bar{x}_{k-1}^{c,[\ell-1]}(i)$, step-size $\Delta_{\ell-1}$ for \eqref{eq:X_c}}
	\item{starting point $\widetilde{x}_{k-1}^{c,[\ell-1]}(i)$, step-size $\Delta_{\ell-1}$ for \eqref{eq:X_c_a}}
	\item{starting point $\bar{x}_{k-1}^{f,[\ell]}(i)$ and step-size $\Delta_{\ell}$ for \eqref{eq:X_f}}
	\item{starting point $\widetilde{x}_{k-1}^{f,[\ell]}(i)$ and step-size $\Delta_{\ell}$ for \eqref{eq:X_a}.}
\end{itemize}
Go to 3..}
\item{Resampling: For $i=1,\dots,M$ compute
\begin{align*}
	& w_k^{c,[\ell-1]}(i) := \tfrac{g(\bar{X}_{k}^{c,[\ell-1]}(i),y_k)}{\sum_{j=1}^M g(\bar{X}_{k}^{c,[\ell-1]}(j),y_k)}, \  \ 
	w_k^{f,[\ell]}(i) := \tfrac{g(\bar{X}_{k}^{f,[\ell]}(i),y_k)}{\sum_{j=1}^M g(\bar{X}_{k}^{f,[\ell]}(j),y_k)};  \\
	& \widetilde{w}_k^{c,[\ell-1]}(i) := \tfrac{g(\widetilde{X}_{k}^{c,[\ell-1]}(i),y_k)}{\sum_{j=1}^M g(\widetilde{X}_{k}^{c,[\ell-1]}(j),y_k)}, \  \ 
	\widetilde{w}_k^{f,[\ell]}(i) := \tfrac{g(\widetilde{X}_{k}^{f,[\ell]}(i),y_k)}{\sum_{j=1}^M g(\widetilde{X}_{k}^{f,[\ell]}(j),y_k)}.
\end{align*} 
For any $\varphi\in\mathcal{B}_b(\mathbb{R}^{N})$ we have the estimate:
\begin{align}
	& \hat{\pi}_k^{[\ell],M}(\varphi) - \hat{\pi}_k^{[\ell-1],M}(\varphi) := 
	\tfrac{1}{2}\sum_{1 \le i \le M} \left\{w_k^{f,[\ell]}(i)\varphi(\bar{X}_k^{f,[\ell]}(i)) + \widetilde{w}_k^{f,[\ell]}(i)\varphi(\widetilde{X}_k^{f,[\ell]}(i))\right\}
	\nonumber \\ 
	&
	\qquad 
	- \tfrac{1}{2}
	\sum_{1 \le i \le M} \left\{w_k^{c,[\ell-1]}(i)\varphi(\bar{X}_k^{c,[\ell-1]}(i)) + %\nonumber \\ & &
	\widetilde{w}_k^{c,[\ell-1]}(i)\varphi(\widetilde{X}_k^{c,[\ell-1]}(i))\right\}. \label{eq:new_cpf_est}
\end{align}
For $i=1,\dots,M$ sample 
indices $\left(j^{c,[\ell-1]}(i),\widetilde{j}^{c,[\ell-1]}(i),j^{f,[\ell]}(i),\widetilde{j}^{f,[\ell]}(i)\right)\in\{1,\dots,M\}^4$
using Algorithm \ref{alg:max_couple_four} in Appendix with probability mass functions $(w_k^{c,[\ell-1]}(\cdot),\widetilde{w}_k^{c,[\ell-1]}(\cdot),w_k^{f,[\ell]}(\cdot),\widetilde{w}_k^{f,[\ell]}(\cdot))$, cardinality $M$
and set 
%			$\check{X}_k^{c,[\ell-1]}(i)=\bar{X}_k^{c,[\ell-1]}(j^{c,[\ell-1]}(i))$,  
%			$\acute{X}_k^{c,[\ell-1]}(i)=\widetilde{X}_k^{c,[\ell-1]}(\widetilde{j}^{c,[\ell-1]}(i))$,
%			$\check{X}_k^{f,[\ell]}(i)=\bar{X}_k^{f,[\ell]}(j^{f,[\ell]}(i))$, 
%			$\acute{X}_k^{f,[\ell]}(i)=\widetilde{X}_k^{f,[\ell]}(\widetilde{j}^{f,[\ell]}(i))$.

$$\check{X}_k^{c,[\ell-1]}(i)=\bar{X}_k^{c,[\ell-1]}(j^{c,[\ell-1]}(i)),  \ \ 
\acute{X}_k^{c,[\ell-1]}(i)=\widetilde{X}_k^{c,[\ell-1]}(\widetilde{j}^{c,[\ell-1]}(i));$$ 
$$ \check{X}_k^{f,[\ell]}(i)=\bar{X}_k^{f,[\ell]}(j^{f,[\ell]}(i)),  \ \ 
\acute{X}_k^{f,[\ell]}(i)=\widetilde{X}_k^{f,[\ell]}(\widetilde{j}^{f,[\ell]}(i)). $$
For $i=1,\dots,M$, set $\bar{X}_k^{c,[\ell-1]}(i)=\check{X}_k^{c,[\ell-1]}(i)$,
$\widetilde{X}_k^{c,[\ell-1]}(i)=\acute{X}_k^{c,[\ell-1]}(i)$, 
$\bar{X}_k^{f,[\ell]}(i)=\check{X}_k^{f,[\ell]}(i)$, 
$\widetilde{X}_k^{f,[\ell]}(i)=\acute{X}_k^{f,[\ell]}(i)$.  
Set $k=k+1$, if $k=T+1$ go to 4.~otherwise go to 2..}
\item{Return the estimates $\hat{\pi}_1^{[\ell],M}(\varphi) - \hat{\pi}_1^{[\ell-1],M}(\varphi),\dots,\hat{\pi}_T^{[\ell],M}(\varphi) - \hat{\pi}_T^{[\ell-1],M}(\varphi)$ from \eqref{eq:new_cpf_est}.}
\end{enumerate}

\caption{New Coupled Particle Filter using the antithetic weak second order scheme. The algorithm is stopped at a time $T$, but need not be.}
\label{alg:new_cpf}
\end{algorithm} 
\section{Numerical results}\label{sec:numerics}
In this section, we provide a series of numerical illustrations detailing our methodology for both forward and filtering problems. 
% through the implementation of the new AMLMC and AMLPF algorithms. 
Specifically, we compare their performance against both multilevel and standard Monte Carlo (Std MC) methods and particle filters. 
%We will outline and rigorously test these algorithms upon a neuroscience and a financial model. Through this exploration, we highlight the advantages of using both AMLMC and AMLPF, showcasing their benefits.
% 
We here summarise the labels of the algorithms that we use in the numerics: 
%that follow:
% 
\begin{itemize}
\item Forward problem: Std MC, MLMC (standard method with scheme (\ref{eq:scheme})), AMLMC (the new antithetic method with scheme (\ref{eq:scheme})) and AMMLMC (the antithetic method of \cite{ml_anti} with scheme (\ref{eq:t_mil})).
\item Filtering problem: PF, MLPF (standard method, using scheme (\ref{eq:scheme})), AMLPF (the new antithetic PF method with scheme (\ref{eq:scheme})) and AMMLPF (the antithetic PF method studied in \cite{anti_mlpf} with scheme (\ref{eq:t_mil})).
\end{itemize}
\subsection{Models}
We consider two SDE models in our experiments. The first model is the stochastic FitzHugh-Nagumo (FHN) model, which is a well-known hypo-elliptic model in neuroscience:  
% is a simplified two-dimensional model derived from the Hodgkin-Huxley (HH) model for spike generation. 
%, introduced by Richard FitzHugh and Jinichi Nagumo. 
%Unlike the more complex HH equations, the FHN equations offer simplicity while effectively elucidating neuronal dynamics. 
% The aim was to isolate the fundamental mathematical elements responsible for excitation and transmission properties, detaching them from the intricacies of sodium and potassium biophysics analysis. 
%Mathematically, the FHN model is a system of first-order nonlinear ordinary differential equations with two coupled equations, one governing a voltage-like variable $X_t$ having a cubic nonlinearity and a recovery variable $Y_t$. 
%It can be presented as:
% with the addition of white noise $ \{B_t\}_{t \geq 0}$ as:
\begin{equation*}
%	\left\{\begin{array}{lllclcll}
%		dX_t= \frac{1}{\epsilon} \big(X_t - X_{t}^{3} -  Z_t - s \big) \; dt ,\vspace{2mm}\\
%		dZ_t =   ( \gamma X_t - Z_t + \beta ) \; dt  + \sigma dB_t^1 .
%	\end{array}\right.
dX_t= \tfrac{1}{\epsilon} \big(X_t - X_{t}^{3} -  Z_t - s \big) \; dt , 
\quad 
dZ_t =   ( \gamma X_t - Z_t + \beta ) \; dt  + \sigma dB_t^1 .
\label{FHN}
\end{equation*}
The  values of the parameters in the simulations are set as follows:  
${X_{t_0} = 0} $,  $Y_{t_0} = 0, $  $ \epsilon = 0.1 $,  $ \sigma = 0.3$, $ \gamma = 1.5$,  $ \beta = 0.3$ and  $ s = 0.01 $. 
%  \textcolor{black}{from which one implies that the  model can be considered as lying within the class of small-noise diffusions (\ref{eq:small_diffusion})}. 
For the forward problem, we estimate the value of $\mathbb{E} [X_T]$ with $T=100$ time units. 
For the filtering case, we estimate
% the value of the filtering distribution 
 $\mathbb{E} [ X_n | y_{0:n}]$ with $n = 100$. The observation data $y_k$ we choose is $y_k \mid( X_{k\delta},  Z_{k\delta} )\sim \mathcal{N}( X_{k\delta} ,\tau^2)$ with $\delta = 1, \, \tau = 0.1$, where 
$\mathcal{N}(m, \sigma^2)$ denotes the Gaussian distribution of mean $m$ and variance $\sigma^2$.

%\begin{table}%[H]
%	\caption{Parameter choices of the FHN model.}
%	\begin{center}
%		\begin{tabular}{ |c| c| c| c|c| c| c| c| } 
%			\toprule
%			${X_{t_0} = 0}\, \text{mV}$ & $Y_{t_0} = 0 $   &  $ \epsilon = 0.1 $  & $ \sigma = 0.3$ 
%			& $ \gamma = 1.5$ & $ \beta = 0.3$   & $ s = 0.01 $ & $T = 100.0$
%			\\
%			\bottomrule 
%		\end{tabular}
%		\label{tab:fhn}
%	\end{center}
%\end{table}

% \begin{figure}%[H]
% \centering
% \subfigure{\includegraphics[width=12cm,height=5.5cm]{Voltage.pdf}}
% % \subfigure{\includegraphics[width=10cm,height=4.0cm]{phase.pdf}}
% \caption{Simulation of the stochastic FHN model via weak second-order scheme. Time evolution of the membrane potential $X_t$ and the adaptation variable $Y_t$.}
% % Top: time evolution of the membrane potential $X_t$ and the adaptation variable $Y_t$. Bottom: phase portrait $(X,Y)$}
% \label{fig:Sim}
% \end{figure}
% 
% \subsubsection{Heston model}

The second model example is the Heston model \cite{Heston} 
%widely recognized in finance as an asset price model with stochastic volatility, making it one of the most popular SDEs for %European 
%option pricing. 
%Governed by two coupled SDEs, 
given as an elliptic SDE not satisfying the commutative condition (\ref{eq:comm}): 
\begin{equation*}
%\left\{\begin{array}{lllclcll}
%	d S_t  = r S_t dt + \sqrt{v_t} S_t dB_{t}^1 ,\vspace{2mm}\\
%	d v_t  = \alpha (\theta - v_t) dt 
%	+  \mu \sqrt{v_t} ( \rho dB_{t}^1 + \sqrt{1 - \rho^2} dB_{t}^2).
%\end{array}\right.
d S_t  = r S_t dt + \sqrt{v_t} S_t dB_{t}^1 ,  \quad 
d v_t  = \alpha (\theta - v_t) dt 
+  \mu \sqrt{v_t} ( \rho dB_{t}^1 + \sqrt{1 - \rho^2} dB_{t}^2). 
\end{equation*} 
% Where $ \{B_{1,t} \}_{t \in [0,1]}$ and $\{ B_{2,t}\}_{t \in [0,1]}$ are two one-dimensional independent Brownian motions. 
The values of the parameters used in the simulations are set as: 		
$S_{t_0} = 100$, $v_{t_0} = 0.09 $,  $r= 0.04$, $\alpha = 2.0$,  $\theta = 0.09$,   $\mu = 0.1$ and  $\rho = 0.7$.
For the forward problem, our target quantity is $\mathbb{E} [ S_T]$ with $T = 1.0$.
For the filtering case, we estimate $\mathbb{E} [S_{n} | y_{0:n}]$ with $n = 100$, where each observation $y_k$ is obtained as  
$y_k \mid( S_{k\delta},  v_{k\delta} )\sim \mathcal{N}(  S_{k\delta} ,\tau^2)$ with $\delta = 0.01$ and $\tau = 2$.
\textcolor{black}{We stress here that 
	in the above model settings, the test functions are unbounded for the filtering problem, while we have assumed boundedness in Section \ref{sec:filtering}. As we will show in the numerical results below, such a discrepancy can be negligible under  suitable scenarios; e.g. the case where the moments of underlying process are uniformly bounded in the time-interval. 
	}
%\begin{table}%[H]
%\caption{Parameter choices of the Heston model.}
%\label{tab:heston}
%\begin{center}
%	\begin{tabular}{ | c| c| c| c |c| c| c| c |} 
%		\toprule 
%		$S_{t_0} = 100$ & $v_{t_0} = 0.09 $ 
%		&  $r= 0.04$   &  $\alpha = 2.0$  & $\theta = 0.09$
%		& $\mu = 0.1$ & $\rho = 0.7$   & $T = 1.0$ 
%		\\
%		\bottomrule 
%	\end{tabular}
%\end{center}
%\end{table}
% 
\subsection{Set-Up and results}

For our numerical experiments, we applied our algorithms to obtain the multilevel estimators. 
% at levels $L \in \{1,2,3,4\}$. 
Given the unavailability of an analytical solution, \textcolor{black}{we will use std MC and PF with a high-resolution $L = 9$ 
% a high-resolution simulation of the single level 
to approximate the ground truth for the forward and filtering problem, respectively} that shall serve as the benchmark solution. 
%  for the forward \textcolor{black}{and} filtering problems.
%our simulations.
%
For the filtering problem, 
\textcolor{black}{though we did not discuss about stochastic resampling for our proposed AMLPF in Section \ref{sec:filtering}, we will run particle filters with adaptive resampling to showcase the practical extendability of AMLPF.} Specifically, resampling is performed when the effective sample size (ESS) is less than $\tfrac{1}{2}$ of the particle numbers. For the coupled filters, we use the ESS of the coarse filter as the measurement of discrepancy.  The error within the estimators in our simulations will be evaluated using the mean square error (MSE), which will be computed by conducting $50$ independent simulations for each method (Std MC, MLMC, AMLMC and AMMLMC) for the forward problem, and (PF, MLPF, AMLPF and AMMLPF) for the filtering case with the ground truth obtained as described above. 

The primary target is to compare the costs of these methods at the same MSE level. In the AMLMC and AMLPF, one needs to determine the number of samples to approximate the multilevel estimators at levels $\ell$ and $\ell-1$, denoted by $M_\ell$.
% as specified in the paper. 
% We set for a given $\epsilon>0$, $L=\mathcal{O}(|\log(\epsilon^{-0.5})|)$, $L\in\mathbb{N}$. 
In particular, we set $M_\ell$ for the AMLMC and AMLPF as 
$M_\ell = c_{1, \ell} \times \varepsilon^{-2} \Delta_\ell^{3/2}$ and $M_\ell  = c_{2, \ell} \times \varepsilon^{-2} \Delta_\ell L$, respectively, for some constants $c_{1, \ell}, c_{2, \ell} >0$ and a given $L$ to attain a target MSE of $\mathcal{O} (\varepsilon^2)$, $\varepsilon > 0$, with a cost of $\mathcal{O} (\varepsilon^{-2})$ for AMLMC and $\mathcal{O}(\epsilon^{-2}\log(\epsilon)^2)$ for AMLPF. 
For the AMMLMC and AMMLPF, we also choose $M_\ell$ as above. 
% In practice we suppose that at level $L$, $\text{MSE} = \mathcal{O} (2^{-2L}) = \text{Var}_L + \text{B}_{L}^2$, where $\text{B}_{L}$ is bias when estimating the filter and $\text{Var}_L$ is variance. 
% The computational cost of these numerical models is determined by the theoretical count of operations required to compute each approximation. 
% This is typically calculated by integrating a step counter within the numerical implementation process. 
In our experiments, we initially simulate the Std MC and PF algorithms with $L \in \{1,2,3,4\}$ and obtain the corresponding MSE and cost values, where the computational cost 
%corresponding to each MSE
is computed as $ \sum_{\ell = 0}^{L} {M_\ell}/{\Delta_\ell}$.  
Subsequently, we use the MLMC and MLPF estimators to achieve identical MSE levels and record their corresponding cost values. 
Finally, we compute the AMLMC and AMLPF estimators to attain similar MSE levels and note their respective cost values. Due to the lower order of weak convergence, the AMMLMC and AMMLPF estimators are computed with $L = \{2, 4, 6, 8\}$.  
% For the AMMLMC, this approach aligns with the methodology outlined in reference \cite{giles}, wherein the optimal number of samples at each level is determined proportionally to $\sqrt{V_l C_l^{-1}}$. 
% Here, $V_l$ represents the variance across multiple levels, while $C_l$ denotes the cost associated with a single sample at level $l$. For the AMMLPF algorithm, we set the number of samples $M_l$ as  $M_l = \mathcal{O} (\epsilon^{-2} \Delta_l \; L)$, where $l \in \{1,..., L\}$. In our simulations, we explore scenarios where $L$ takes values from the set $ \{2,4,6,8 \}$.

We present our numerical simulations to show the benefits of applying AMLMC/AMLPF to the above SDE models, compared to Std MC, MLMC, AMMLMC/PF, MLPF, AMMLPF. 
Figures \ref{fig:MSEvsCost1}-\ref{fig:MSEvsCost2} show the MSE against the cost. %to present the cost rates through each targeted MSE.  
% $\mathbb{E}[\phi(X_{n \delta}) | y_{1:n}]$,. 
The figures show that as we increase the levels from $L=1$ to $L=4$, the difference in the cost between the methods also increases. 
% These figures show the advantage and accuracy of using the AMLMC and AMLPF. 
Table \ref{tab:mse} presents the \textcolor{black}{estimated change rates of $\log$(cost) against $\log$(MSE)} for both problems. \textcolor{black}{The reported rates align} with our theoretical expectations. 
We observe that the  computational costs are \textcolor{black}{of sizes consistent to the theoretical ones of} $\mathcal{O}(\epsilon^{-3})$ for the Std MC and PF, $\mathcal{O}(\epsilon^{-2})$ for the AMLMC, and $\mathcal{O}(\epsilon^{-2} \log (\epsilon)^2)$ for the AMLPF. 
% In comparing the MSE against the computational cost, between the antithetic weak second-order scheme and the antithetic Milstein algorithms, 
% the setting of the parameter $L$ plays a crucial role, particularly in Milstein's case. With paths nearly identical in both algorithms, the lower cost arises from the choice of $L$. 
Moreover, we see from the bottom two plots of Figures \ref{fig:MSEvsCost1}-\ref{fig:MSEvsCost2} that AMLMC/AMLPF (using the weak second order scheme) outperformed AMMLMC/AMMLPF (using the truncated Milstein scheme) in terms of \textcolor{black}{cost vs MSE}. 
We note that when choosing the number of samples $M_\ell$ in the experiments, the constants $c_{1, \ell}$ and $c_{2, \ell}$ to determine $M_\ell$ (indicated above) are allowed to be set lower for the case of the weak second order scheme compared with that of the truncated Milstein scheme. 
We expect this is due to the tighter variance bounds for the couplings of the AMLMC under a \textcolor{black}{small-noise} diffusion setting, i.e. the case some small parameter is contained in the diffusion coefficient, which we detail in Section \ref{sec:B} in Appendix. 
%\textcolor{red}{Section \ref{sec:sd}}. 
% the lower number of samples are This leads to the clear differences between AMLMC/AMLPF and AMMLMC/AMMLPF.   
% This can be observed through Figures \ref{fig:MSEvsCost1} - \ref{fig:MSEvsCost2} and Table \ref{tab:mse}. 
% Moreover, it's worth noting that while the antithetic weak second order achieves comparable MSE to its Milstein counterpart, the computational cost of Antithetic Milstein is notably higher. This difference demonstrates the improvement of the antithetic weak second order over Milstein, in terms of MSE vs. cost.

\begin{figure}[h]
\centering
\subfigure{\includegraphics[width=7.5cm, height=4.1cm]{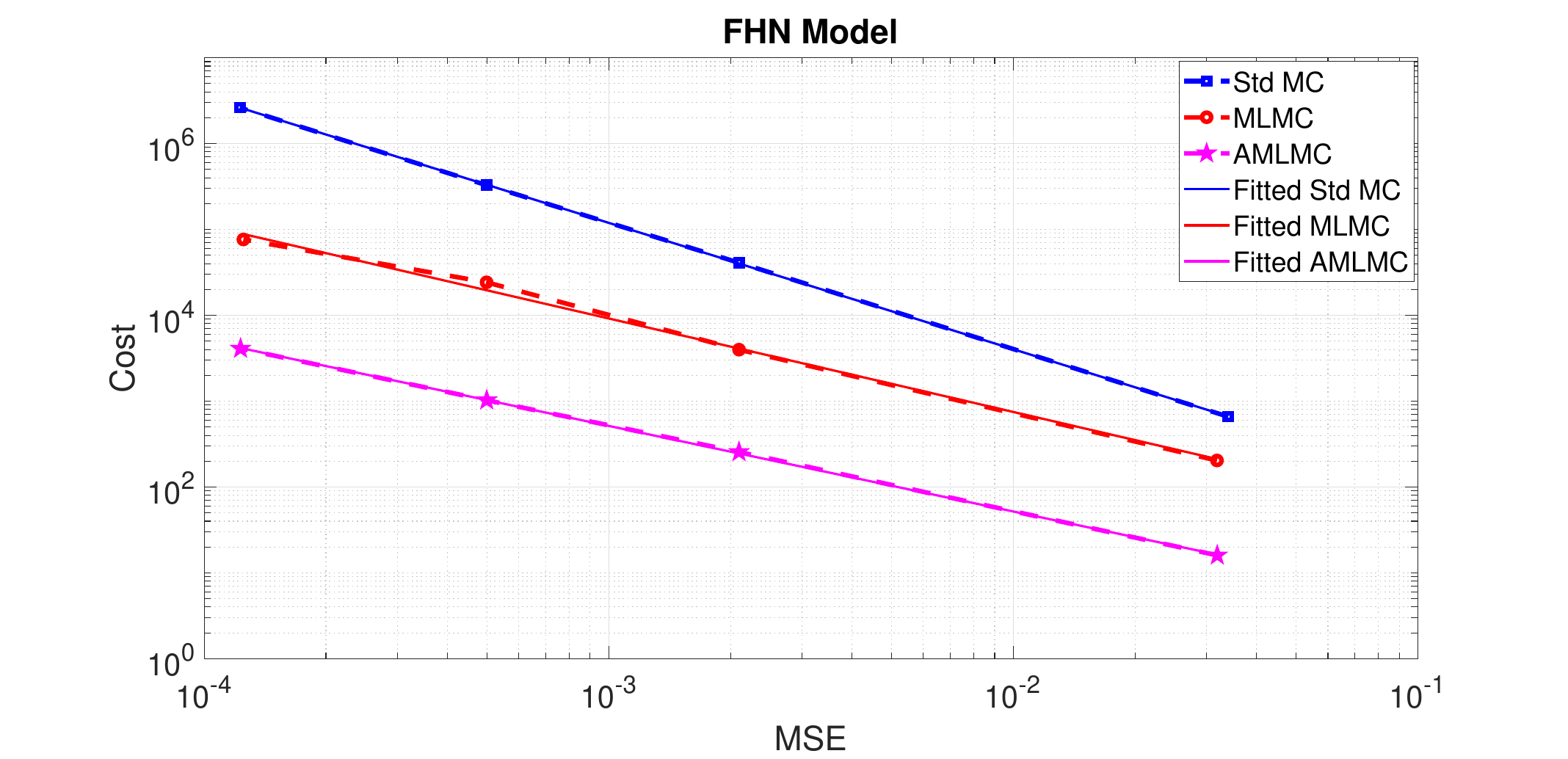}}
\subfigure{\includegraphics[width=7.5cm, height=4.1cm]{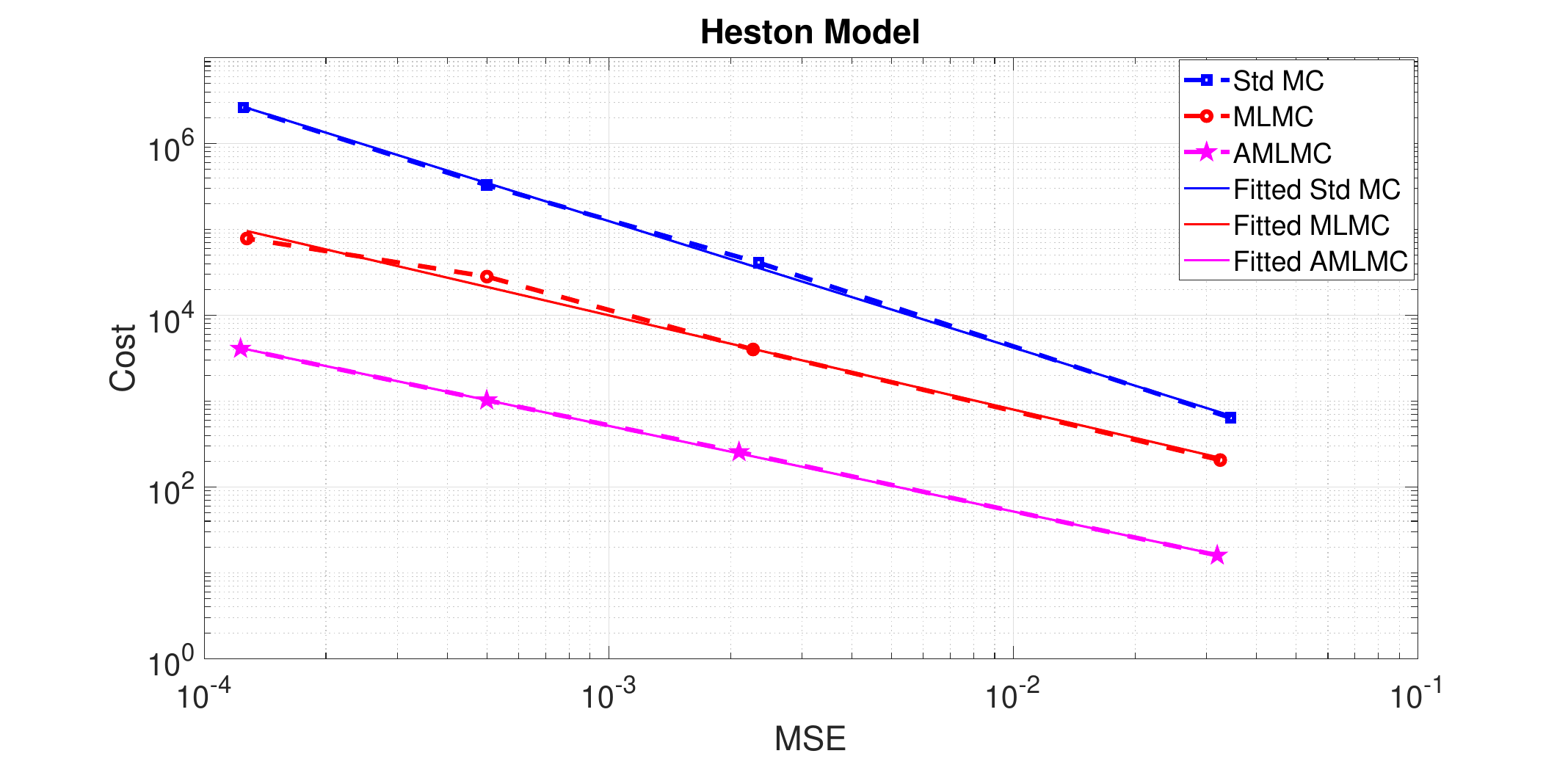}}
\\ 
\subfigure{\includegraphics[width=7.5cm, height=4.4cm]{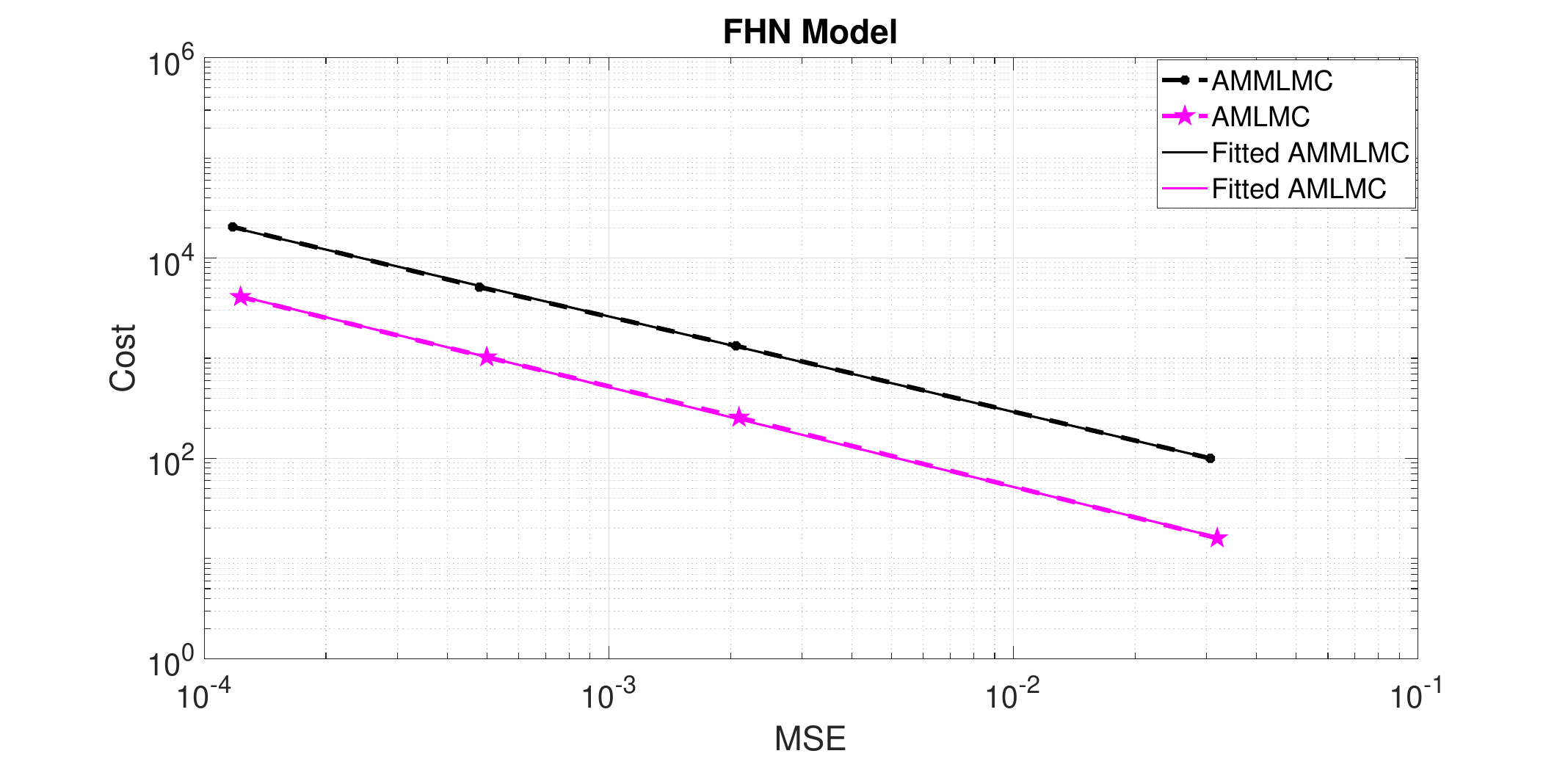}}
\subfigure{\includegraphics[width=7.5cm, height=4.4cm]{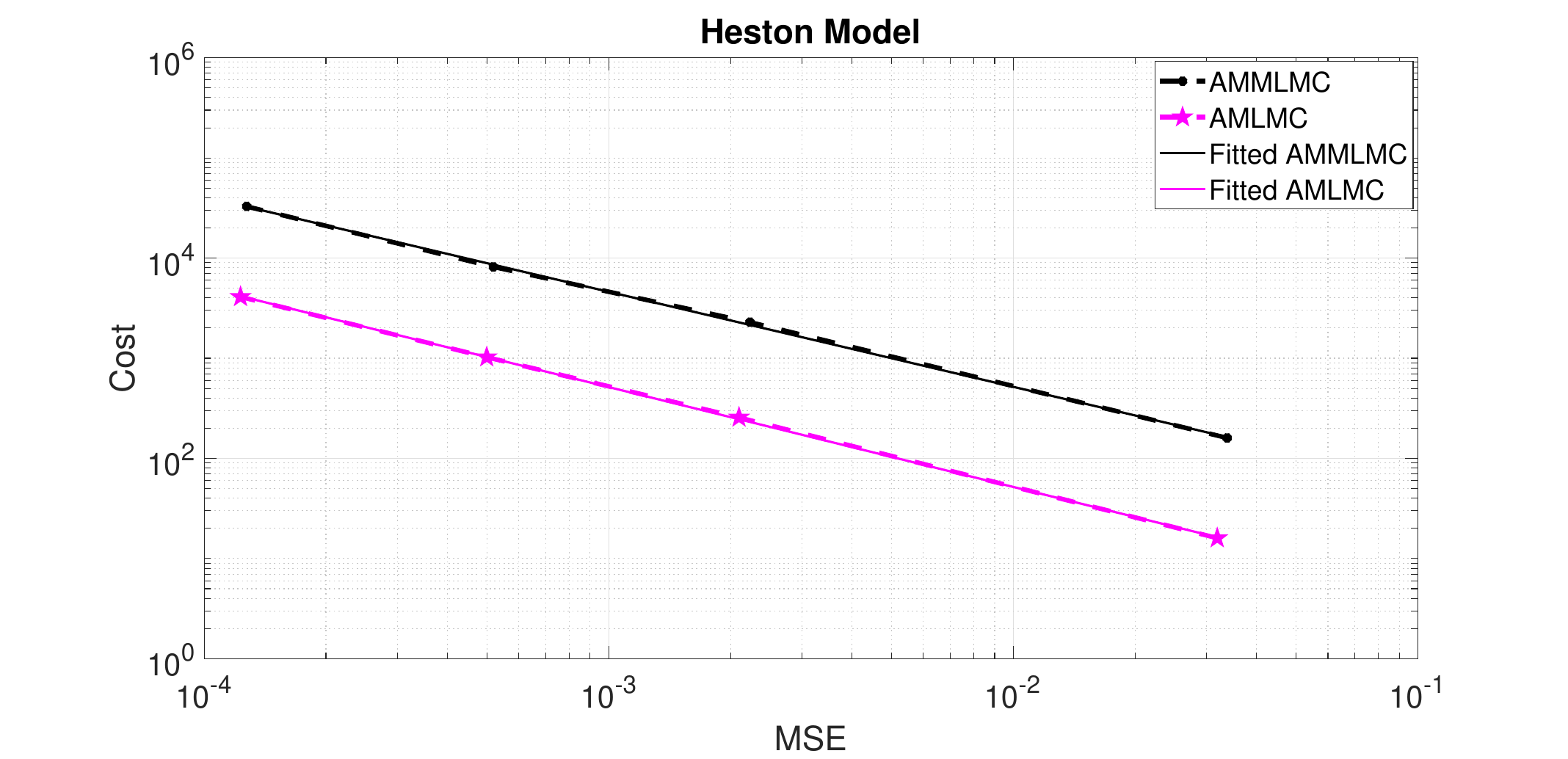}} 
\caption{\textcolor{black}{Cost versus MSE for the forward problem.}}
\label{fig:MSEvsCost1}
\end{figure}
\begin{figure}[h]
\centering
\subfigure{\includegraphics[width=7.5cm, height=4.1cm]{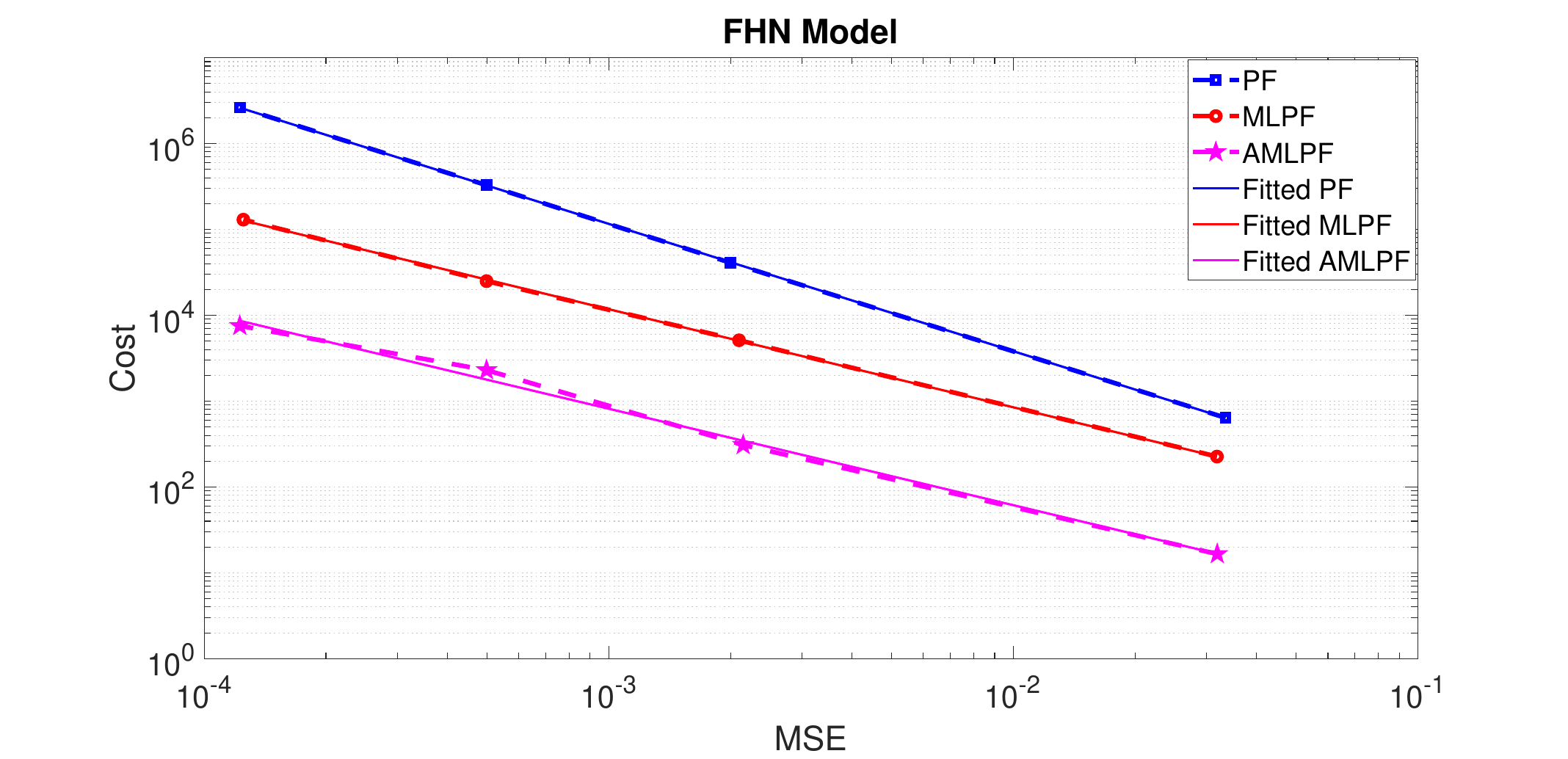}}
\subfigure{\includegraphics[width=7.5cm, height=4.1cm]{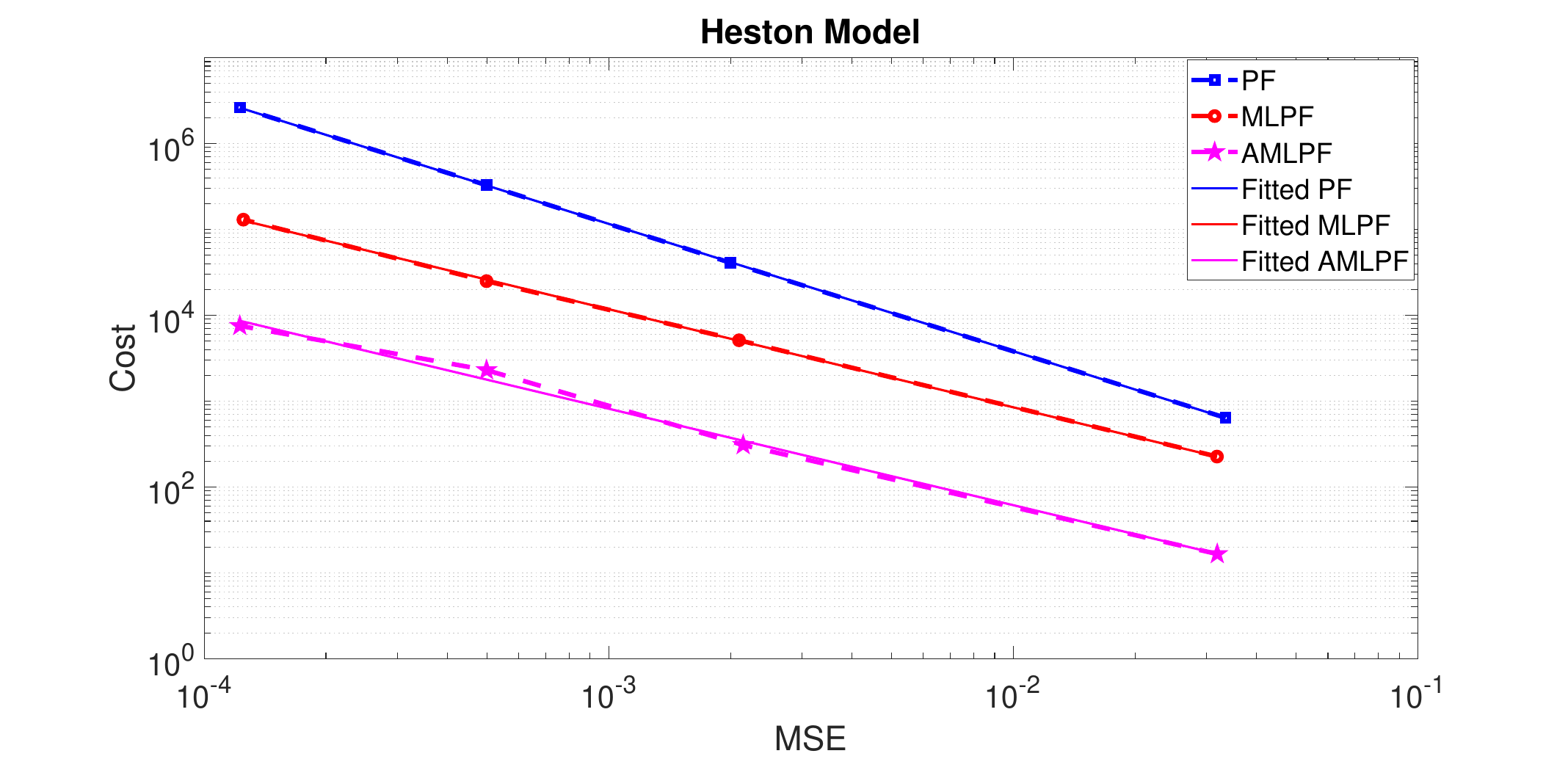}}
\\ 
\subfigure{\includegraphics[width=7.5cm, height=4.4cm]{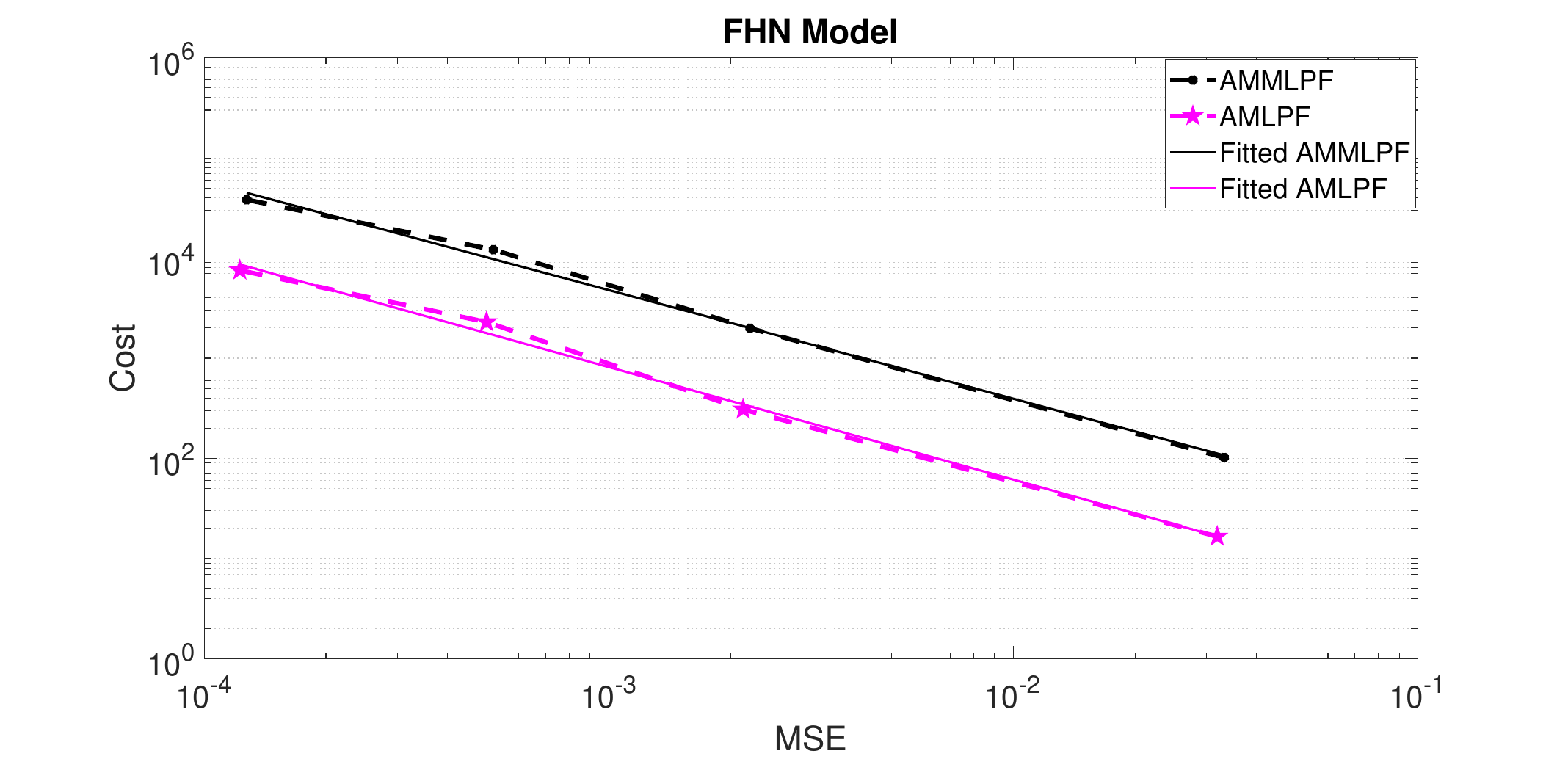}}
\subfigure{\includegraphics[width=7.5cm, height=4.4cm]{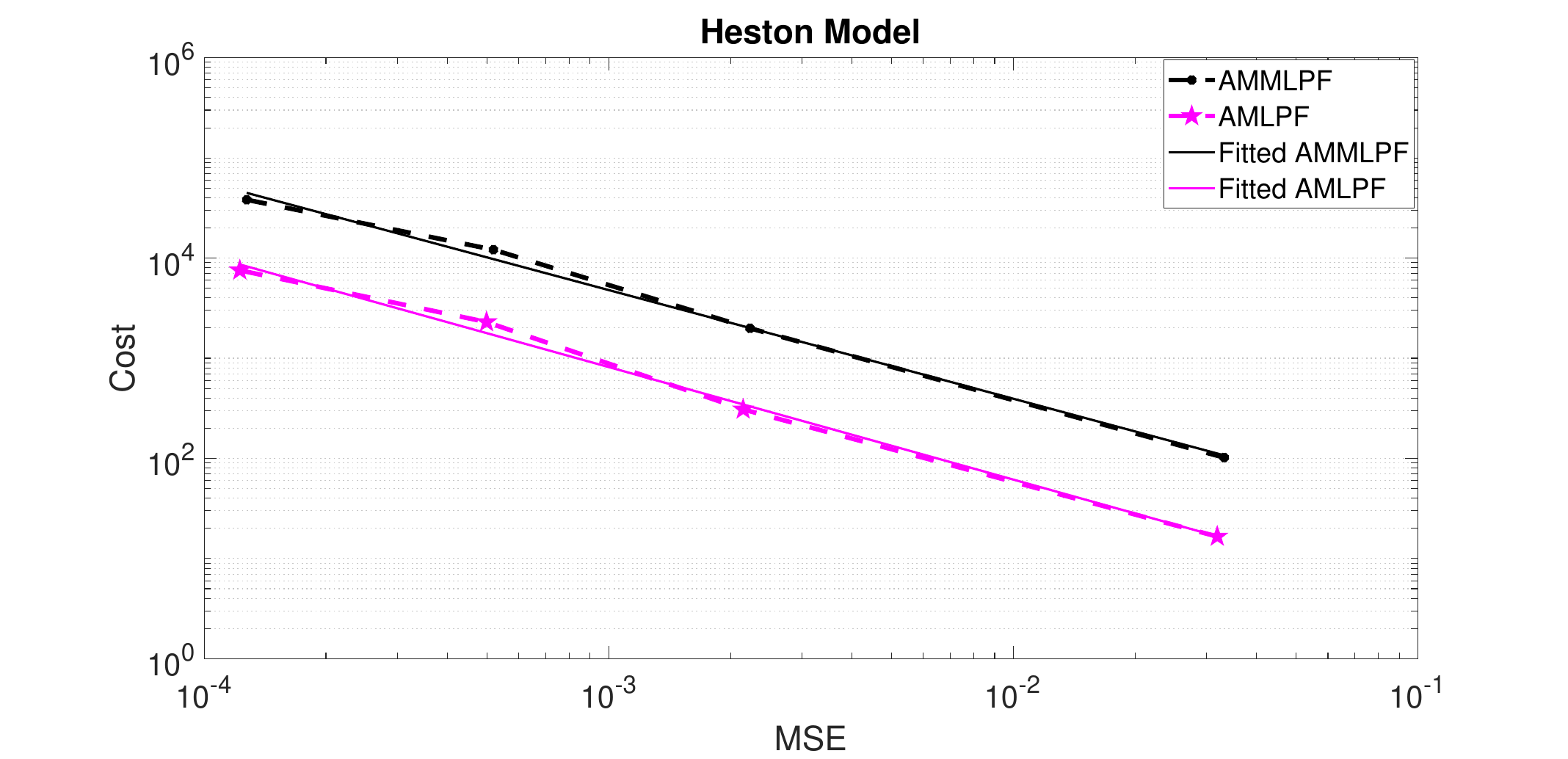}} 
\caption{\textcolor{black}{Cost versus MSE} for the filtering problem.}
\label{fig:MSEvsCost2}
\end{figure}

\begin{table} 
\caption{\textcolor{black}{Estimated change rate of $log$(cost) against $log$(MSE)}. Left: Forward problem. Right: Filtering problem.}  
\label{tab:mse} 
\begin{minipage}[t]{.49\textwidth}
\begin{center}
\begin{tabular}{ c c c c  } 
	\toprule 
	Model  & \hspace{-0.4cm} Std MC & \hspace{-0.4cm} MLMC  & \hspace{-0.4cm} AMLMC \\
	\midrule
	\textbf{FHN} & \hspace{-0.4cm}  -1.48 & \hspace{-0.4cm}  -1.1 & \hspace{-0.4cm}  -1.03 \\ 
	\textbf{Heston} & \hspace{-0.4cm}  -1.47 & \hspace{-0.4cm}  -1.11 & \hspace{-0.4cm}  -1.05\\
	\bottomrule 
\end{tabular}
\end{center}
% \subcaption{Forward problem}
\end{minipage}
%
%\hfill
%
\begin{minipage}[t]{.49\textwidth}
\begin{center}
\begin{tabular}{ c c c c  } 
	\toprule
	Model  & \hspace{-0.4cm}  PF & \hspace{-0.4cm}  MLPF  & \hspace{-0.4cm}  AMLPF \\
	\midrule
	\textbf{FHN} & \hspace{-0.2cm}  -1.46 & \hspace{-0.4cm}  -1.17 & \hspace{-0.4cm}   -1.11 \\ 
	
	\textbf{Heston} & \hspace{-0.2cm}  -1.49 & \hspace{-0.4cm}  -1.24 & \hspace{-0.4cm}  -1.14\\
	
	\bottomrule 
\end{tabular}
\end{center}
% \caption{Filtering problem}
\label{}
\end{minipage}
\end{table} 

% \begin{table}%[H]
% \caption{Estimated rates of MSE with respect to the cost for the filtering problem.} 
% \begin{center}
% \begin{tabular}{ c c c c  } 
% \toprule
% Model  & PF & MLPF  & AMLPF \\
% \midrule
% \textbf{FHN} & -1.46 &-1.17 & -1.11 \\ 

% \textbf{Heston} & -1.49 &-1.24 & -1.14\\

% \bottomrule 
% \end{tabular}

% \label{tab:tabfiltering}
% \end{center}
% \end{table}

\section{Conclusion}
\textcolor{black}{Our work has investigated the use of a 
% non-degenerate 
weak second order scheme within the multilevel Monte Carlo (MLMC) framework. 
%, under the objective of an efficient estimation of expectations~w.r.t the law of a wide class of diffusion processes, including \emph{hypo-elliptic} diffusions}. \textcolor{black}{This latter is an important class of SDEs with numerous uses in applications. 
We first proved that our scheme has a strong error~1}. 
Then, in the context of MLMC, we developed a new antithetic estimator based on our weak second order scheme which achieves the optimal cost rate $\mathcal{O} (\varepsilon^{-2})$, $\varepsilon > 0$, to obtain a MSE of $\mathcal{O} (\varepsilon^2)$. Such an optimal cost rate is also reported for the different antithetic MLMC  approach of \cite{ml_anti} which makes use of a truncated Milstein scheme of weak error 1.
The new antithetic estimator is shown to possess a benefit versus the one of \cite{ml_anti}, that is, our estimator is expected to be more efficient for a finite maximum level of discretization $L$ used in practice due to the higher order weak convergence. 
% We also note that in the Supplementary Material, we have analytically shown that for small-noise diffusions the variance of the new antithetic estimator will be smaller than that of \cite{ml_anti}}.
As an application, \textcolor{black}{we have proposed an antithetic multilevel particle filter (AMLPF) by building upon previous works \cite{mlpf, anti_mlpf} for the purposes of efficient filtering of diffusion processes from observations. Our simulation studies are in support of the anticipated cost of the proposed AMLPF being $\mathcal{O} (\varepsilon^{-2} \log (\varepsilon)^2)$ to achieve an MSE of $\mathcal{O} (\varepsilon^2)$}. \textcolor{black}{Also, all our numerics support the understanding that the new antithetic estimator using the weak second order scheme outperforms the antithetic Milstein scheme-based estimator in both forward/filtering problems.}  We emphasize that our numerical scheme is locally non-degenerate under both elliptic/hypo-elliptic settings, whereas the truncated Milstein scheme is degenerate in the hypo-elliptic case. The non-degeneracy of the scheme makes possible its deployment within particle filters with guided proposals so that stochastic weights required to be assigned to particles are well-defined and available as the ratio of products involving the density expression for the numerical scheme and the proposal, though the exploration of this direction is left for future work.  

%\subsection*{Acknowledgements}
%YI was supported by the Additional Funding Programme for Mathematical Sciences, delivered by EPSRC (EP/V521917/1) and the Heilbronn Institute for Mathematical Research. 
%AJ was supported by SDS CUHK, Shenzhen. 

\appendix 

\section{Proof of Proposition \ref{prop:s_rate}} \label{app:s_rate}

\begin{proof}
	Let $1 \le i \le 2^{\ell}$. We have that  
	$$
		\mathcal{S}_i \equiv 
		\mathbb{E}  
		\Bigl[ \max_{0 \le k \le i} 
		\| X_{t_k} - \bar{X}_{t_k} \|^p 
		\Bigr]
		\le N^{p-1} \sum_{1 \le j \le N}
		\mathbb{E} 
		\Bigl[ \max_{0 \le k \le i} 
		| X_{t_k}^j - \bar{X}^{j}_{t_k} |^p 
		\Bigr],
	$$
	where we made use of the following inequality:
	\begin{align} \label{eq:bd}
		\Bigl( \sum_{1 \le j \le N}|x_j| \Bigr)^p \le N^{p-1} 
		\sum_{1 \le j  \le N} |x_j|^{p}, 
		\quad x = (x_1, \ldots, x_N) \in \mathbb{R}^N.  
	\end{align}
	We will show that for any $p \ge 2$, there exists a constant $C > 0$ such that: 
	\begin{align} \label{eq:bd_grn}
		\mathcal{S}_i \le C 
		\Bigl(
		\Delta_\ell^{p/2} 
		+  
		\sum_{0 \le n \le i-1} \mathcal{S}_n \cdot \Delta_\ell 
		\Bigr),
	\end{align}
	which leads to the conclusion due to the discrete Gronwall's inequality.
%	 We have for $0 \le k \le i$, 
%	% 
%	\begin{align}
%		& X_{t_k}^j - \bar{X}^{j}_{t_k}  
%		= \sum_{0 \le n \le k-1} \sum_{0 \le m \le d}
%		\int_{t_n}^{t_{n+1}}
%		\bigl( 
%		\sigma_m^j (X_s) - \sigma_m^j (\bar{X}_{t_n})
%		\bigr) dB_s^m  \nonumber \\ 
%		& \quad
%		- \sum_{0 \le n \le k-1}  \sum_{0 \le m_1, m_2 \le d} \mathcal{L}_{m_1} \sigma_{m_2}^j (\bar{X}_{t_n})
%		\Delta \eta_{t_{n+1}, t_n}^{m_1 m_2} 
%		- \tfrac{1}{2} 
%		\sum_{0 \le n \le k-1}  
%		\sum_{1  \le m_1 < m_2 \le d} 
%		\bigl[\sigma_{m_1}, \sigma_{m_2} \bigr]^j (\bar{X}_{t_n}) \Delta \widetilde{A}_{t_{n+1}, t_n}^{m_1 m_2}. \nonumber
%	\end{align}
	% 
	% 
	Application of \textcolor{black}{stochastic} Taylor expansion for $X_{t_n}^{j}, \, 0 \le n \le k-1,$ yields that, for $1 \le j \le N$:  
	\begin{align*}
		& X_{t_k}^j - \bar{X}^{j}_{t_k}  
		= \sum_{ \substack{0 \le n \le k-1 \\ 0 \le m \le d}}   \int_{t_n}^{t_{n+1}}
		\bigl( \sigma_m^j (X_s) - \sigma_m^j (\bar{X}_{t_n})
		\bigr) dB_s^m  \nonumber \\ 
		& \quad
		- \sum_{\substack{0 \le n \le k-1 \\ 0 \le m_1, m_2 \le d }}   \mathcal{L}_{m_1} \sigma_{m_2}^j (\bar{X}_{t_n})
		\Delta \eta_{t_{n+1}, t_n}^{m_1 m_2}  - \tfrac{1}{2} 
		\sum_{\substack{0 \le n \le k-1 \\ 1  \le m_1 < m_2 \le d}}  
		\bigl[\sigma_{m_1}, \sigma_{m_2} \bigr]^j (\bar{X}_{t_n}) \Delta \widetilde{A}_{t_{n+1}, t_n}^{m_1 m_2} \\ 
		& = 
		\sum_{ \substack{0 \le n \le k-1 \\ 0 \le m \le d}}  
		\bigl( 
		\sigma_m^j (X_{t_n}) - \sigma_m^j (\bar{X}_{t_n}) 
		\bigr) \Delta B^m_{t_{n+1}, t_n} 
		\nonumber \\
		& \quad 
		+ \tfrac{1}{2}
		\sum_{\substack{0 \le n \le k-1 \\ 1 \le m_1, m_2 \le d}} 
		\bigl(
		\mathcal{L}_{m_1} \sigma_{m_2}^j (X_{t_n}) 
		- \mathcal{L}_{m_1} \sigma_{m_2}^j (\bar{X}_{t_n}) 
		\bigr) 
		\bigl\{ 
		\Delta B_{t_{n+1}, t_n}^{m_1} \Delta B_{t_{n+1}, t_n}^{m_2} 
		- \Delta_\ell \mathbf{1}_{m_1 = m_2} 
		\bigr\} 
		\nonumber \\ 
		& \quad 
		- \tfrac{1}{2} \sum_{\substack{0 \le n \le k-1 \\  1 \le m_1 < m_2 \le d }} 
		\bigl\{ 
		[\sigma_{m_1}, \sigma_{m_2}]^j (X_{t_n}) \Delta {A}_{t_{n+1}, t_n}^{m_1 m_2}
		+  
		[\sigma_{m_1}, \sigma_{m_2}]^j (\bar{X}_{t_n}) 
		\Delta \widetilde{A}_{t_{n+1}, t_n}^{m_1 m_2} 
		\bigr\} 
		\nonumber \\
		& \quad 
		+ \sum_{0 \le n \le k-1} 
		\bigl(
		\mathcal{M}_{t_{n+1}, t_n}^j + \mathcal{N}_{t_{n+1}, t_n}^j  
		\bigr), 
	\end{align*}
	where the terms $\mathcal{M}_{t_{n+1}, t_n}^j$,  
	$\mathcal{N}_{t_{n+1}, t_n}^j$ are such that 
	$\mathbb{E} \bigl[ \mathcal{M}_{t_{n+1}, t_n}^j | \mathcal{F}_{t_n} \bigr] = 0$ for $0 \le n \le k-1$ and it holds under Assumption \ref{ass:coeff} that for any $p \ge 2$, there exist constants $C_1, C_2 > 0$ such that
	\begin{align}
		\max_{0 \le n \le k-1} \mathbb{E} 
		\bigl[
		| \mathcal{M}_{t_{n+1}, t_n}^j |^p
		\bigr] \le C_1 \Delta_\ell^{3p/2}, 
		\qquad 
		\max_{0 \le n \le k-1} \mathbb{E} 
		\bigl[
		| \mathcal{N}_{t_{n+1}, t_n}^j |^p
		\bigr] \le C_2 \Delta_\ell^{2p}. 
	\end{align}
	Thus, inequality (\ref{eq:bd}) yields
	$
	\mathbb{E} \bigl[  \max_{0 \le k \le i}  \bigl| X_{t_k}^j - \bar{X}_{t_k}^j \bigr|^p  \bigr]  \le C_p \sum_{1 \le \alpha \le 6}  \mathcal{T}_i^{(\alpha), j} 
	$ 
	for some constant $C_p > 0$, where we have set: 
	\allowdisplaybreaks
	\begin{align}
		&\hspace{-0.4cm} \mathcal{T}_i^{(1), j} 
		= \mathbb{E} 
		\bigl[ 
		\max_{0 \le k \le i} 
		\bigl| 
		\sum_{0 \le n \le k-1} 
		\bigl(
		\sigma_0^j (X_{t_n}) - \sigma_0^j (\bar{X}_{t_n}) 
		\bigr) \Delta_\ell  
		\bigr|^p
		\bigr]; 
		\nonumber \\ 
		& \hspace{-0.4cm} \mathcal{T}_i^{(2), j} 
		= \mathbb{E}
		\bigl[ 
		\max_{0 \le k \le i}  
		\bigl|
		\sum_{\substack{0 \le n \le k-1 \\ 1 \le m \le d}}  
		\bigl( 
		\sigma_m^j (X_{t_n}) 
		- \sigma_m^j (\bar{X}_{t_n})  
		\bigr) \Delta B^m_{t_{n+1}, t_n} 
		\bigr|^p
		\bigr]; 
		\nonumber \\ 
		& \hspace{-0.4cm}
		 \mathcal{T}_i^{(3), j}  
		 \! 
		= 
		\mathbb{E}
		\bigl[
		\max_{0 \le k \le i}  
		\bigl|
		\! \! \!  \! 
		\sum_{\substack{0 \le n \le k-1 \\ 1 \le m_1, m_2 \le d}} 
		\! \! \! \! \! \! \! 
         \bigl(
		\mathcal{L}_{m_1} \sigma_{m_2}^j (X_{t_n}) 
		- 
		\mathcal{L}_{m_1} \sigma_{m_2}^j (\bar{X}_{t_n}) 
		\bigr) 
		\bigl( 
		\Delta B_{t_{n+1}, t_n}^{m_1} \Delta B_{t_{n+1}, t_n}^{m_2} 
		\! 
		- \Delta_\ell \mathbf{1}_{m_1 = m_2} 
		\bigr)
		\bigr|^p
		\bigr];  
		\nonumber \\   
		& \hspace{-0.4cm} 
		\mathcal{T}_i^{(4), j}
		=  \mathbb{E} 
		\bigl[ \max_{0 \le k \le i}  
		\bigl|\sum_{\substack{0 \le n \le k-1 \\ 1 \le m_1 < m_2 \le d}} 
		\bigl\{ 
		[\sigma_{m_1}, \sigma_{m_2}]^j (X_{t_n}) \Delta {A}_{t_{n+1}, t_n}^{m_1 m_2}  +  
		[\sigma_{m_1}, \sigma_{m_2}]^j (\bar{X}_{t_n}) 
		\Delta \widetilde{A}_{t_{n+1}, t_n}^{m_1 m_2} 
		\bigr\}  
		\bigr|^p \bigr]; 
		\nonumber \\[0.1cm]    
		&\hspace{-0.4cm} 
		 \mathcal{T}_i^{(5), j}  = 
		\mathbb{E} 
		\bigl[ \max_{0 \le k \le i} 
		\bigl|
		\sum_{0 \le n \le k-1} \mathcal{M}_{t_{n+1}, t_n}^j
		\bigr|^p 
		\bigr], \quad 
		\mathcal{T}_i^{(6), j} = 
		\mathbb{E} 
		\bigl[ \max_{0 \le k \le i} 
		\bigl|
		\sum_{0 \le n \le k-1} \mathcal{N}_{t_{n+1}, t_n}^j
		\bigr|^p 
		\bigr]. \nonumber   
	\end{align}
	Applying inequality (\ref{eq:bd}), we have under Assumption \ref{ass:coeff} that:  
	\begin{align*}
		\mathcal{T}^{(1), j}_i 
		& \le i^{p-1} 
		\sum_{0 \le n \le i-1} \mathbb{E} 
		\bigl[ 
		| \sigma_0^j (X_{t_n}) - \sigma_0^j (\bar{X}_{t_n}) 
		|^p
		\bigr] \Delta_\ell^p 
%		\le c_1 (i \Delta_\ell)^{p-1} \sum_{0 \le n \le i-1} \mathcal{S}_n \cdot \Delta_\ell  
		\le c_1 T^{p-1} \sum_{0 \le n \le i-1} \mathcal{S}_n \cdot \Delta_\ell 
	\end{align*}
	for some constant $c_1 > 0$ independent of $\Delta_\ell$ since $i \Delta_\ell \le T$. Similarly, we have:
	\begin{align}
		\mathcal{T}^{(6),j}_i \le i^{p-1} \sum_{0 \le n \le i-1} \mathbb{E} 
		\bigl[ 
		| \mathcal{N}_{t_{n+1}, t_n}^j |^p
		\bigr]  
		\le c_6 T^p \Delta_\ell^p 
	\end{align}
	for some constant $c_6 > 0$. We consider the other four terms. Since they involve martingales, we make use of the discrete Burkholder-Davis-Gundy inequality to obtain: 
	\begin{align*}
		\mathcal{T}_{i}^{(2),j} 
		& \le c_{2,1} \mathbb{E} 
		\bigl[ 
		\bigl( 
		\sum_{0 \le n \le i-1} \sum_{1 \le m \le d}
		\bigl\{ 
		(\sigma_{m}^j (X_{t_n}) - \sigma_{m}^j (\bar{X}_{t_n})) \Delta B_{t_{n+1}, t_n}^m
		\bigr\}^2
		\bigr)^{p/2}
		\bigr] 
		\nonumber \\ 
%		& \le i^{p/2 - 1} \sum_{0 \le n \le i-1}
%		\mathbb{E}
%		\Bigl[
%		\Bigl|
%		\sum_{1 \le m \le d} \bigl\{ 
%		(\sigma_{m}^j (X_{t_n}) - \sigma_{m}^j (\bar{X}_{t_n})) \Delta B_{t_{n+1}, t_n}^m
%		\bigr\}^2
%		\Bigr|^{p/2}
%		\Bigr]
%		\nonumber \\ 
		& \le c_{2, 2} \, i^{p/2 - 1} \sum_{0 \le n \le i-1}
		\sum_{1 \le m \le d}
		\mathbb{E}
		\bigl[
		\bigl|
		(\sigma_{m}^j (X_{t_n}) - \sigma_{m}^j (\bar{X}_{t_n})) \Delta B_{t_{n+1}, t_n}^m
		\bigr|^{p}
		\bigr]
		\nonumber \\ 
		& \le c_{2, 3} \, i^{p/2 - 1} \sum_{0 \le n \le i-1} \mathcal{S}_n \cdot \Delta_\ell^{p/2}
		\le c_{2, 3} \, T^{p/2 - 1} \sum_{0 \le n \le i-1} \mathcal{S}_n \cdot \Delta_\ell
	\end{align*}
	for some constants $c_{2,1}, c_{2,2}, c_{2,3} > 0$, where we applied (\ref{eq:bd}) in the second  inequality. Similarly, we have that:
	\begin{align*}
		\mathcal{T}_i^{(3), j}  
		& \le c_3 i^{p/2 - 1} 
		\sum_{0 \le n \le i-1} \mathcal{S}_n \cdot \Delta_\ell^{p} 
		\le c_3 T^{p/2 - 1} \Delta_\ell^{p/2}
		\sum_{0 \le n \le i-1} \mathcal{S}_n \cdot \Delta_\ell;  
		\\ 
		\mathcal{T}_i^{(5), j} 
		& \le c_{5, 1} \, i^{p/2 -1} \sum_{0 \le n \le i-1}
		\mathbb{E} 
		\bigl[
		\bigl| 
		\mathcal{M}_{t_{n+1}, t_n}^j
		\bigr|^p
		\bigr]
		\le c_{5,2} \, i^{p/2} \Delta_\ell^{3p/2}
		= c_{5,2} \, T^{p/2} \Delta_\ell^{p}
	\end{align*} 
	for some constants $c_3, c_{5,1}, c_{5,2} > 0$. Finally, for $\mathcal{T}_i^{(4), j}$, we obtain:
	\begin{align}
		\mathcal{T}_i^{(4), j} 
		\le c_{4,1} i^{p/2 -1} 
		\sum_{0 \le n \le i-1} 
		\sum_{1 \le m_1 < m_2 \le d} \mathbb{E} 
		\bigl[ 
		| \Delta A_{t_{n+1}, t_n}^{m_1 m_2} |^p 
		+  
		| \Delta \widetilde{A}_{t_{n+1}, t_n}^{m_1 m_2} |^p 
		\bigr]
		\le c_{4, 2} T^{p/2} \Delta_\ell^{p/2}
	\end{align}
	for constants $c_{4, 1}, c_{4 ,2} > 0$, where we used 
	that $\mathbb{E} |\Delta A_{t_{n+1}, t_n}^{m_1 m_2}|^p  
	= \mathcal{O} (\Delta_\ell^p)$, $\mathbb{E} |\Delta \widetilde{A}_{t_{n+1}, t_n}^{m_1 m_2}|^p  
	= \mathcal{O} (\Delta_\ell^p)$ for any $p \ge 2$. 
	Note that $\sum_{0 \le n \le i-1} \mathcal{S}_n$ does not appear in the upper bound of $\mathcal{T}_i^{(4),j}$. Thus, we obtain inequality (\ref{eq:bd_grn}) and conclude. 
\end{proof}
% 
%\begin{remark} 
%	The error term $\mathcal{T}_i^{(4),j}$ appeared in the above proof is induced from the use of tractable random variables $\Delta \widetilde{A}_{t_{n+1}, t_n}$ instead of intractable L\'evy area $\Delta A_{t_{n+1}, t_n}$ in the weak second order scheme (\ref{eq:scheme}). Due to the existence of the error term, the weak second order scheme attains the same rate of strong convergence as the E-M scheme or the truncated Milstein scheme (\ref{eq:t_mil}). 
%\end{remark} 

\section{Auxiliary results for Theorem \ref{thm:strong_err_coupl}} 
\label{app:tech}
Throughout this section, let $1 \le j \le N$, $1 \le \ell \le L$, $0 \le k \le 2^{\ell - 1} - 1$ and $t_k = k \Delta_{\ell-1}$. 
\begin{lemma} \label{lemma:diff_fine} 
	 It holds that:
	\allowdisplaybreaks
	\begin{align*}
		& \bar{X}^{f, [\ell], j}_{t_{k+1}} 
		 = \bar{X}^{f, [\ell], j}_{t_k} 
		+ \sum_{0 \le m \le d} 
		\sigma_{m}^j \bigl(\bar{X}^{f, [\ell]}_{t_k} \bigr) 
		\Delta B_{t_{k+1}, t_k}^{m} 
		+ \sum_{0 \le m_1, m_2 \le d} 
		\mathcal{L}_{m_1} \sigma_{m_2}^j \bigl(\bar{X}^{f, [\ell]}_{t_k} \bigr) 
		\Delta \eta_{t_{k+1}, t_k}^{m_1 m_2} \nonumber \\ 
		& \quad 
		- \tfrac{1}{2} \sum_{1 \le m_1, m_2 \le d} 
		\mathcal{L}_{m_1} \sigma_{m_2}^j 
		\bigl( \bar{X}^{f, [\ell]}_{t_k} \bigr)
		\bigl( 
		\Delta B_{t_{k+1}, t_{k+1/2}}^{m_1}  
		\Delta B_{t_{k+1/2}, t_k}^{m_2}
		-
		\Delta B_{t_{k+1/2}, t_k}^{m_1} 
		\Delta B_{t_{k+1}, t_{k+1/2}}^{m_2} 
		\bigr) \nonumber \\ 
		& \quad 
		+ \tfrac{1}{2} \sum_{1 \le  m_1 <  m_2 \le d} 
		\bigl[\sigma_{m_1}, \sigma_{m_2} \bigr]^j
		\bigl( \bar{X}^{f, [\ell]}_{t_k} \bigr) 
		\bigl( 
		\Delta \widetilde{A}^{m_1 m_2}_{t_{k+1/2}, t_k}
%		\Delta B^{m_1}_{t_{k+1/2}, t_k} 
%		\Delta \widetilde{B}^{m_2}_{t_{k+1/2}, t_k} 
		+ 
%		\Delta B^{m_1}_{t_{k+1}, t_{k+1/2}}
%		\Delta \widetilde{B}^{m_2}_{t_{k+1}, t_{k+1/2}}  
        \Delta \widetilde{A}^{m_1 m_2}_{t_{k+1}, t_{k+1/2}} 
		\bigr)
		% & \quad - \tfrac{1}{2} \sum_{1 \le  k_2 <  k_1 \le d} 
		% \mathcal{L}_{k_1} \sigma_{k_2, \theta}^j (\bar{X}^{f}_{i})
		% \bigl( \delta B^{k_1}_{(i+1/2, i)} \delta \widetilde{B}^{k_2}_{(i+1/2, i)} 
		% + \delta B^{k_1}_{(i+1, i+1/2)} \delta \widetilde{B}^{k_2}_{(i+1, i+1/2)}  \bigr) \nonumber \\
        + \bar{\mathcal{M}}^{f, j}_{t_{k+1}, t_k} 
		+ \bar{\mathcal{N}}^{f, j}_{t_{k+1}, t_k},  
	\end{align*}
	where the remainder terms are such that  
	$ \mathbb{E}  \bigl[  \bar{\mathcal{M}}^{f, j}_{t_{k+1}, t_k}  | \mathcal{F}_{t_{k}}  \bigr] = 0
	$,  
	and for any $p \ge 2$ there exist constants $C_1, C_2 > 0$ so that: 
    $$ 
		\max_{0 \le k \le 2^{\ell-1} - 1} 
		\mathbb{E} \bigl[ | \bar{\mathcal{M}}^{f, j}_{t_{k+1}, t_k}|^p \bigr] 
		\le C_1 \Delta_{\ell-1}^{3p/2}, \quad 
		\max_{0 \le k \le 2^{\ell-1} - 1} 
		\mathbb{E} \bigl[ | \bar{\mathcal{N}}^{f, j}_{t_{k+1}, t_k} |^p \bigr] 
		\le C_2 \Delta_{\ell -1}^{2p}. 
	$$ 
	Similarly, it holds that:
	\allowdisplaybreaks
	\begin{align*}  
		& \widetilde{X}^{f, [\ell], j}_{t_{k+1}} 
		= \widetilde{X}^{f, [\ell], j}_{t_k} 
		+ \sum_{0 \le m \le d} 
		\sigma_{m}^j 
		\bigl( \widetilde{X}^{f, [\ell]}_{t_k} \bigr) 
		\Delta B_{t_{k+1}, t_k}^{m} 
		+  \sum_{0 \le m_1, m_2 \le d} 
		\mathcal{L}_{m_1} \sigma_{m_2}^j 
		\bigl(
		\widetilde{X}^{f, [\ell]}_{t_k}
		\bigr) 
		\Delta \eta_{t_{k+1}, t_k}^{m_1 m_2} 
		\nonumber  \\ 
		& \quad 
		+ \tfrac{1}{2} \sum_{1 \le m_1, m_2 \le d}
		\mathcal{L}_{m_1} \sigma_{m_2}^j 
		\bigl( \widetilde{X}^{f, [\ell]}_{t_k} \bigr) 
		\bigl( 
		\Delta B_{t_{k+1}, t_{k+1/2}}^{m_1} 
		\Delta B_{t_{k+1/2}, t_k}^{m_2} 
		- \Delta B_{t_{k+1/2}, t_k}^{m_1} 
		\Delta B_{t_{k+1}, t_{k+1/2}}^{m_2} 
		\bigr) \nonumber \\ 
		& \quad 
		- \tfrac{1}{2} \sum_{1 \le  m_1 <  m_2 \le d} 
		\bigl[ \sigma_{m_1}, \sigma_{m_2} \bigr]^j
		\bigl( \widetilde{X}^{f, [\ell]}_{t_k} \bigr)  
		\bigl( 
%		\Delta B^{m_1}_{t_{k+1/2}, t_k} 
%		\Delta \widetilde{B}^{m_2}_{t_{k+1/2}, t_k} 
		\widetilde{A}_{t_{k+1/2}, t_k}^{m_1 m_2}
		+ 
		\widetilde{A}_{t_{k+1}, t_{k+1/2}}^{m_1 m_2}
%		\Delta B^{m_1}_{t_{k+1}, t_{k+1/2}} 
%		\Delta \widetilde{B}^{m_2}_{t_{k+1}, t_{k+1/2}}
		\bigr) 
		% & \quad + \tfrac{1}{2} \sum_{1 \le  k_2 <  k_1 \le d} 
		% \mathcal{L}_{k_1} \sigma_{k_2, \theta}^j (\bar{X}^{a}_{i})
		% \bigl( \delta B^{k_1}_{(i+1/2, i)} \delta \widetilde{B}^{k_2}_{(i+1/2, i)} 
		% + \delta B^{k_1}_{(i+1, i+1/2)} \delta \widetilde{B}^{k_2}_{(i+1, i+1/2)}  \bigr) \nonumber \\
        + \widetilde{\mathcal{M}}_{t_{k+1}, t_k}^{\, f, j} +  \widetilde{\mathcal{N}}_{t_{k+1}, t_k}^{\, f, j},
	\end{align*}
	where the remainder terms 
	$\widetilde{\mathcal{M}}_{t_{k+1}, t_k}^{\, f, j}$ and $\widetilde{\mathcal{N}}_{t_{k+1}, t_k}^{\, f, j}$ 
	satisfy the same properties as 
	$\bar{\mathcal{M}}_{t_{k+1}, t_k}^{f, j}$ and 
	$\bar{\mathcal{N}}_{t_{k+1}, t_k}^{f, j}$, respectively. 
\end{lemma}  

\begin{proof}
	From the definition of the fine discretization scheme (\ref{eq:X_f}), we have: 
	\begin{align} \label{eq:X_f_diff}
		& \bar{X}_{t_{k+1}}^{{f}, [\ell], j}
		= \bar{X}_{t_k}^{{f}, [\ell], j}
		% + \sigma_{0, \theta}^j (\bar{X}_{i}^{{f}}) \, \tfrac{h}{2}
		% + \sigma_{0, \theta}^j (\bar{X}_{i+1/2}^{{f}}) \, \tfrac{h}{2}  \nonumber \\
		% & \quad 
		+ \sum_{0 \le m \le d} 
		\Bigl\{ 
		\sigma_{m}^j (\bar{X}_{t_k}^{f, [\ell]}) 
		\Delta B_{t_{k+1/2},t_k}^{m}
		+ 
		\sigma_{m}^j (\bar{X}_{t_{k+1/2}}^{f, [\ell]}) 
		\Delta B_{t_{k+1}, t_{k+1/2}}^{m}
		\Bigr\}  \\
		& \qquad 
		+ \sum_{0 \le m_1, m_2 \le d}
		\Bigl\{ 
		\mathcal{L}_{m_1} \sigma_{m_2}^j 
		\bigl(\bar{X}_{t_k}^{f, [\ell]} \bigr) 
		\Delta \eta_{t_{k+1/2}, t_k}^{m_1 m_2}
		+ 
		\mathcal{L}_{m_1} \sigma_{m_2}^j 
		\bigl( \bar{X}_{t_{k+1/2}}^{f, [\ell]} \bigr) 
		\Delta \eta_{t_{k+1}, t_{k+1/2}}^{m_1 m_2}
		\Bigr\}   \nonumber \\ 
		& \qquad 
		+ \tfrac{1}{2} 
		\sum_{1 \le m_1 < m_2 \le d} 
		\Bigl\{ 
		[ \sigma_{m_1}, \sigma_{m_2}]^j 
		\bigl(\bar{X}_{t_k}^{f, [\ell]} \bigr) 
		\widetilde{A}_{t_{k+1/2}, t_k}^{m_1 m_2} 
%		\Delta B^{m_1}_{t_{k+1/2}, t_k}
%		\Delta \widetilde{B}^{m_2}_{t_{k+1/2}, t_k}  
		+ 
		[ \sigma_{m_1}, \sigma_{m_2}]^j 
		\bigl( \bar{X}_{t_{k+1/2}}^{f, [\ell]} \bigr) 
		\widetilde{A}_{t_{k+1}, t_{k+1/2}}^{m_1 m_2}
%		\Delta B^{m_1}_{t_{k+1}, t_{k+1/2}}
%		\Delta \widetilde{B}^{m_2}_{t_{k+1}, t_{k+1/2}}
		\Bigr\}. \nonumber 
	\end{align} 
	The It\^o-Taylor expansion gives, for $0 \le m \le d$:  
	\begin{align} 
		\sigma_{m}^j (\bar{X}_{t_{k+1/2}}^{f, [\ell]}) 
		% = \sigma_{k}^j (\bar{X}_{i}^{f, [\ell]})  
		% + \sum_{m = 1}^N \int_{t_i}^{t_{i}+h_{\ell}} 
		% \partial_m \sigma_{k, \theta} (\bar{X}_{s}^{f}) d \bar{X}_{s}^{f, m} \nonumber \\ 
		% & \qquad + 
		% \tfrac{1}{2} 
		% \sum_{m_1, m_2 = 1}^N \int_{t_i}^{t_{i}+h/2} 
		% \partial_{m_1} \partial_{m_2} 
		% \sigma_{k, \theta}^j (\bar{X}_{s}^{f}) 
		% d \langle  \bar{X}_{\cdot}^{f, m_1}, \bar{X}_{\cdot}^{f, m_2}  \rangle_s \nonumber \\  
		= 
		\sigma_{m}^j (\bar{X}_{t_k}^{f, [\ell]})  
		+ \sum_{0 \le m_1 \le d} \mathcal{L}_{m_1} \sigma^j_{m} (\bar{X}_{t_k}^{f, [\ell]}) 
		\Delta B_{t_{k+1/2}, t_k}^{m_1} 
		+ \mathcal{E}_{t_{k+1/2}, t_k}^{f, j},  \label{eq:ito_taylor}
	\end{align}
	where under Assumption \ref{ass:coeff} the remainder term $\mathcal{E}_{t_{k+1/2}, t_k}^{f, j}$ is such that, for any $p \ge 2$, there exists a constant $C > 0$ so that 
	$
		\max_{0 \le k \le 2^{\ell-1} -1} \mathbb{E} 
		\bigl[ 
		| \mathcal{E}_{t_{k+1/2}, t_k}^{f, j} |^p 
		\bigr] 
		\le C \Delta_{\ell}^p. 
	$ 
	Furthermore, we note that the standard Taylor expansion gives that for any $f \in C_b^1 (\mathbb{R}^N)$: 
	\begin{align} \label{eq:taylor_1}
		f (\bar{X}_{t_{k+1/2}}^{f, [\ell]})  
		& =f (\bar{X}_{t_k}^{f, [\ell]}) 
		+ \sum_{1 \le i \le N} \partial_i  f  (\xi) 
		\bigl( \bar{X}_{t_{k+1/2}}^{f,[\ell], i} 
		- \bar{X}_{t_k}^{f, [\ell], i} \bigr) 		
	\end{align}
	% 
%	\begin{align}
%		\mathcal{L}_{m_1} \sigma_{m_2}^j (\bar{X}_{t_{k+1/2}}^{f, [\ell]})  
%		& = \mathcal{L}_{m_1} \sigma_{m_2}^j (\bar{X}_{t_k}^{f, [\ell]}) 
%		+ \sum_{1 \le i \le N} \partial_i \bigl( \mathcal{L}_{m_1} \sigma_{m_2}^j \bigr) (\xi_1) 
%		\bigl( \bar{X}_{t_{k+1/2}}^{f,[\ell], i} 
%		- \bar{X}_{t_k}^{f, [\ell], i} \bigr); 
%		\label{eq:taylor_1}  \\ 
%		[\sigma_{m_1}, \sigma_{m_2}]^j 
%		(\bar{X}_{t_{k+1/2}}^{f, [\ell]}) 
%		& = [ \sigma_{m_1}, \sigma_{m_2}]^j  
%		(\bar{X}_{t_k}^{f, [\ell]})  \nonumber \\ 
%		& \qquad + \sum_{1 \le i \le  N} 
%		\partial_i \bigl( [ \sigma_{m_1}, \sigma_{m_2}]^j \bigr) (\xi_2)  (\bar{X}_{t_{k+1/2}}^{f, [\ell], i} 
%		- \bar{X}_{t_k}^{f, [\ell], i}), 
%		\label{eq:taylor_2} 
%	\end{align}
	% 
for some variable $\xi  \in \mathbb{R}^N$, and it holds that: 
%Notice that under Assumption \ref{ass:coeff}
%	it holds that for any $p \ge 2$ and $1 \le j \le N$, there exists a constant $C > 0$ such that  
%	$$
%	\max_{0 \le k \le 2^{\ell -1} -1} \mathbb{E} 
%	\bigl[
%	| 
%	\bar{X}_{t_{k+1/2}}^{f, [\ell], j} 
%	- \bar{X}_{t_k}^{f, [\ell], j} 
%	|^p 
%	\bigr] 
%	\leq C \Delta_{\ell-1}^{p/2}. 
%	$$ 
	% 
	\begin{align} \label{eq:bm_incre} 
		\Delta B^{m_1}_{t_{k+1}, t_k} 
		\Delta B^{m_2}_{t_{k+1}, t_k} 
		& = \Delta B^{m_1}_{t_{k+1}, t_{k+1/2}} 
		\Delta B^{m_2}_{t_{k+1}, t_{k+1/2}}  
		+ \Delta B^{m_1}_{t_{k+1}, t_{k+1/2}} 
		\Delta B^{m_2}_{t_{k+1/2}, t_k}  \nonumber  
		\\ 
		& \quad + 
		\Delta B^{m_1}_{t_{k+1/2}, t_k} 
		\Delta B^{m_2}_{t_{k+1}, t_{k+1/2}} 
		+  
		\Delta B^{m_1}_{t_{k+1/2}, t_k} 
		\Delta B^{m_2}_{t_{k+1/2}, t_k}.  
	\end{align} 
Thus, applying (\ref{eq:ito_taylor}), (\ref{eq:taylor_1}) and (\ref{eq:bm_incre}) to (\ref{eq:X_f_diff}), we obtain that: 
	\begin{align*}
		& \bar{X}_{t_{k+1}}^{{f}, [\ell], j}
		= \bar{X}_{t_k}^{{f}, [\ell], j}
		+ \sum_{0 \le m \le d} 
		\sigma_{m}^j (\bar{X}_{t_k}^{f, [\ell]}) 
		\Delta B_{t_{k+1}, t_k}^{m} 
		+ \sum_{0 \le m_1, m_2 \le d} 
		\mathcal{L}_{m_1} \sigma_{m_2}^j 
		\bigl( \bar{X}_{t_k}^{f, [\ell]} \bigr) 
		\Delta \eta_{t_{k+1}, t_k}^{m_1 m_2}
		\nonumber \\ 
		& \quad 
		- \tfrac{1}{2} \sum_{1 \le m_1, m_2 \le  d} \mathcal{L}_{m_1} \sigma_{m_2}^j 
		\bigl(\bar{X}^{f, [\ell]}_{t_k} \bigr)
		\bigl( 
		\Delta B_{t_{k+1}, t_{k+1/2}}^{m_1} 
		\Delta B_{t_{k+1/2}, t_k}^{m_2} 
		- \Delta B_{t_{k+1/2, k}}^{m_1} 
		\Delta B_{t_{k+1}, t_{k+1/2}}^{m_2} 
		\bigr) \nonumber \\  
		& \quad + 
		\tfrac{1}{2} 
		\sum_{1 \le m_1 < m_2 \le d} 
		[ \sigma_{m_1}, \sigma_{m_2}]^j
		\bigl(\bar{X}_{t_k}^{f, [\ell]} \bigr)
		\bigl( 
		\widetilde{A}_{t_{k+1/2}, t_k}^{m_1 m_2}
%		\Delta B^{m_1}_{t_{k+1}, t_{k+1/2}}
%		\Delta \widetilde{B}^{m_2}_{t_{k+1}, t_{k+1/2}}
		+ 
		\widetilde{A}_{t_{k+1}, t_{k+1/2}}^{m_1 m_2}
%		\Delta B^{m_1}_{t_{k+1/2}, t_k}
%		\Delta \widetilde{B}^{m_2}_{t_{k+1/2}, t_k} 
		\bigr) 
        + \bar{\mathcal{M}}_{t_{k+1}, t_k}^{f, j} + \bar{\mathcal{N}}_{t_{k+1}, t_k}^{f, j},  
	\end{align*}
	where the remainder terms $\bar{\mathcal{M}}_{t_{k+1}, t_k}^{f, j}$ and $\bar{\mathcal{N}}_{t_{k+1}, t_k}^{f, j}$ have the properties stated in Lemma \ref{lemma:diff_fine}. The assertion for $\widetilde{X}^{f, [\ell]}$ follows from the same discussion above, and the proof is now complete.   
\end{proof}

\begin{lemma} \label{lemma:f_anti_diff}
	%Let $1 \le j \le N$, $1 \le \ell \le L$, $0 \le k \le 2^{\ell - 1} - 1$ and $t_k = k \Delta_{\ell-1}$. 
	It holds that: 
	\begin{align} \label{eq:X_hat_f}
		\begin{aligned}
			& \hat{X}^{f, [\ell], j}_{t_{k+1}} 
			= \hat{X}^{f, [\ell], j}_{t_k}
			+ \sum_{0 \le m \le d} 
			\sigma_{m}^j (\hat{X}^{f, [\ell]}_{t_k}) 
			\Delta B_{t_{k+1}, t_k}^{m} 
			+ \sum_{1 \le m_1, m_2 \le d} 
			\mathcal{L}_{m_1} \sigma_{m_2}^j 
			( \hat{X}_{t_k}^{f, [\ell]} ) 
			\Delta \eta_{t_{k+1}, t_k}^{m_1 m_2}  \\
			& \qquad 
			+ \hat{\mathcal{M}}_{t_{k+1}, t_k}^{f, j} 
			+ \hat{\mathcal{N}}_{t_{k+1}, t_k}^{f, j}, 
		\end{aligned}  
	\end{align}
	where the remainder terms $\hat{\mathcal{M}}_{t_{k+1}, t_k}^{f, j}$ and    $\hat{\mathcal{N}}_{t_{k+1}, t_k}^{f, j}$ satisfy the same properties as 
	$\bar{\mathcal{M}}_{t_{k+1}, t_k}^{f, j}$ and 
	$\bar{\mathcal{N}}_{t_{k+1}, t_k}^{f, j}$ in Lemma \ref{lemma:diff_fine}, respectively.
%	 follows:
%	% 
%	\begin{align} 
%		\mathbb{E}
%		\bigl[\hat{\mathcal{M}}_{t_{k+1}, t_k}^{f,  j} | \mathcal{F}_{t_{k}} \bigr] = 0, \quad 0 \le k \le 2^{\ell-1} - 1, 
%	\end{align}
%	% 
%	and for any $p \ge 2$, there exist constants $C_1, C_2 > 0$ such that 
%	% 
%	\begin{align}
%		\max_{0 \le k \le 2^{\ell-1} - 1} 
%		\mathbb{E} 
%		\bigl[ 
%		| \hat{\mathcal{M}}_{t_{k+1}, t_k}^{f, j}|^p 
%		\bigr] 
%		\le C_1 \Delta_{\ell-1}^{3p/2}, 
%		\quad 
%		\max_{0 \le k \le 2^{\ell-1} - 1} 
%		\mathbb{E}
%		\bigl[ 
%		| \hat{\mathcal{N}}_{t_{k+1}, t_k}^{f, j} |^p
%		\bigr] 
%		\le C_2 \Delta_{\ell-1}^{2p}. 
%	\end{align}
	% 
\end{lemma}

\begin{proof}
	For notational simplicity, we omit the subscript ``$[\ell]$" during the proof. 
	Due to Lemma \ref{lemma:diff_fine}, we get:
	\begin{align}
		& \hat{X}_{t_{k+1}}^{f,  j} 
		= \hat{X}_{t_k}^{f, j} 
		+ \sum_{0 \le m \le d}  
		\sigma_{m}^j (\hat{X}_{t_k}^{f}) 
		\Delta B_{t_{k+1}, t_k}^{m} 
		+ \sum_{1 \le m_1, m_2 \le d} 
		\mathcal{L}_{m_1} \sigma_{m_2}^j 
		( \hat{X}_{t_k}^{f} ) 
		\Delta \eta_{t_{k+1}, t_k}^{m_1 m_2} 
		\nonumber  \\
		& \qquad 
		+ \sum_{1 \le i \le 6} 
		{\mathcal{E}}_{t_{k+1}, t_k}^{(i), j}
		+ \tfrac{1}{2} 
		\bigl\{ 
		\bar{\mathcal{M}}_{t_{k+1}, t_k}^{f,  j}
		+ \widetilde{\mathcal{M}}_{t_{k+1}, t_k}^{\, f,  j} 
		+ \bar{\mathcal{N}}_{t_{k+1}, t_k}^{f,  j}
		+ \widetilde{\mathcal{N}}_{t_{k+1}, t_k}^{\, f,  j}  
		\bigr\}, 
	\end{align}
	where we have set:
	\allowdisplaybreaks
	\begin{align}
		{\mathcal{E}}_{t_{k+1}, t_k}^{(1), j} 
		& = \sum_{1 \le m \le d} 
		\bigl( 
		\tfrac{1}{2} \sigma_m^j (\bar{X}^{f}_{t_k}) 
		+ \tfrac{1}{2} \sigma_m^j (\widetilde{X}^{f}_{t_k})  
		- \sigma_m^j (\hat{X}_{t_k}^{f}) 
		\bigr) \Delta B_{t_{k+1}, t_k}^m;    
		\nonumber 
		\\ 
		{\mathcal{E}}_{t_{k+1}, t_k}^{(2), j} 
		& = 
		\bigl(
		\tfrac{1}{2}  \sigma_{0}^j (\bar{X}_{t_k}^{f})
		+ \tfrac{1}{2} \sigma_{0}^j (\widetilde{X}_{t_k}^{f}) 
		- \sigma_{0}^j (\hat{X}_{t_k}^{f})
		\bigr) \Delta_{\ell-1} 
		+ \bigl( \mathcal{L}_0 \sigma_{0}^j (\bar{X}_{t_k}^{f}) 
		+ \mathcal{L}_0 \sigma_{0}^j (\widetilde{X}_{t_k}^{f})  
		\bigr) 
		\tfrac{\Delta_{\ell-1}^2}{4}; 
		\nonumber  \\[0.1cm] 
		{\mathcal{E}}_{t_{k+1}, t_k}^{(3), j} 
		& = \sum_{1 \le m_1, m_2 \le d} 
		\Bigl(
		\tfrac{1}{2} \mathcal{L}_{m_1} \sigma_{m_2}^j 
		(\bar{X}_{t_k}^{f})
		+ \tfrac{1}{2} 
		\mathcal{L}_{m_1} \sigma_{m_2}^j (\widetilde{X}_{t_k}^{f}) 
		- \mathcal{L}_{m_1} \sigma_{m_2}^j 
		(\hat{X}_{t_k}^{f}) \Bigr)
		\Delta \eta_{t_{k+1}, t_k}^{m_1 m_2}; 
		\nonumber  \\[0.1cm] 
		{\mathcal{E}}_{t_{k+1}, t_k}^{(4), j} 
		& \! \! =  \!  \tfrac{1}{2} \! \! 
		\sum_{1 \le m \le d}  \! \! 
		\bigl\{ 
		\bigl(
		\mathcal{L}_m \sigma_{0}^j 
		(\bar{X}_{t_k}^{f}) + 
		\mathcal{L}_m \sigma_{0}^j (\widetilde{X}_{t_k}^{f}) 
		\bigr) \Delta \eta_{t_{k+1}, t_k}^{m 0} 
		\!  +  \! 
		\bigl(
		\mathcal{L}_0 \sigma_{m}^j (\bar{X}_{t_k}^{f}) + 
		\mathcal{L}_0 \sigma_{m}^j 
		(\widetilde{X}_{t_k}^{f}) 
		\bigr) 
		\Delta \eta_{t_{k+1}, t_k}^{0m} 
		\bigr\} ; 
		\nonumber 
		\\[0.1cm] 
		{\mathcal{E}}_{t_{k+1}, t_k}^{(5), j}  
		& = - \tfrac{1}{4} \sum_{1 \le m_1, m_2 \le d}
		\bigl\{ \bigl(
		\mathcal{L}_{m_1} \sigma_{m_2}^j (\bar{X}_{t_k}^{f}) 
		- \mathcal{L}_{m_1} \sigma_{m_2}^j (\widetilde{X}_{t_k}^{f}) \bigr)  \nonumber \\ 
		& \qquad \qquad \qquad \qquad 
		\times \bigl( 
		\Delta B_{t_{k+1}, t_{k+1/2}}^{m_1}
		\Delta B_{t_{k+1/2}, t_{k}}^{m_2} 
		- \Delta B_{t_{k+1/2}, t_{k}}^{m_1}
		\Delta B_{t_{k+1}, t_{k+1/2}}^{m_2}  
		\bigr) 
		\bigr\}; \nonumber \\  
		{\mathcal{E}}_{t_{k+1}, t_k}^{(6), j}  
		& = \tfrac{1}{4} \sum_{1 \le m_1 < m_2 \le d}
		\bigl( 
		[\sigma_{m_1}, \sigma_{m_2}]^j (\bar{X}_{t_k}^{f}) 
		- [\sigma_{m_1}, \sigma_{m_2}]^j (\widetilde{X}_{t_k}^{f}) 
		\bigr) 
		\bigl( 
%		\Delta B^{m_1}_{t_{k+1}, t_{k+1/2}} 
%		\Delta \widetilde{B}^{m_2}_{t_{k+1}, t_{k+1/2}}
       \Delta \widetilde{A}_{t_{k+1}, t_{k+1/2}}^{m_1 m_2}
		+ 
		\Delta \widetilde{A}_{t_{k+1/2}, t_{k}}^{m_1 m_2} 
%		\Delta B^{m_1}_{t_{k+1/2}, t_{k}} 
%		\Delta \widetilde{B}^{m_2}_{t_{k+1/2}, t_{k}} 
		\bigr). \nonumber 
	\end{align}
	We immediately have that, 
	%for any $1 \le j \le N$ and $0 \le k \le 2^{\ell-1} -1$, 
	% 
	$ \mathbb{E} 
		\bigl[ {\mathcal{E}}_{t_{k+1}, t_k}^{(i), j}  | \mathcal{F}_{t_{k}} 
		\bigr] = 0$, $i \in \{1, 3, 4, 5, 6\}$.
	% 
	% 
	% Since $\max \bigl\{ \mathbb{E} \bigl|  
	% \eta^{k0}_{(i+1,i)} \bigr|^p, 
	% \mathbb{E} \bigl|  \eta^{0k}_{(i+1,i)} \bigr|^p  
	% \bigr\} = \mathcal{O}(h^{3p/2})$, we immediately see that 
	% % 
	% \begin{align}
		% \max_{0 \le i \le n-1} \mathbb{E} \bigl[| \hat{\mathcal{E}}_{i}^{(5), j} |^p \bigr] 
		% =  \mathcal{O} (h^{3p/2}). 
		% \end{align}
	% 
Applying second order Taylor expansion around  $\hat{X}_{t_k}^{f}$, 
	% 
%	\begin{align*}
%		& \tfrac{1}{2} 
%		\bigl( 
%		g (\bar{X}_{t_k}^{f, [\ell]})
%		+  g (\widetilde{X}_{t_k}^{f, [\ell]})
%		\bigr)
%		- g (\hat{X}_{t_k}^{f, [\ell]}) 
%		%\\
%		% = \tfrac{1}{2} \sum_{1 \le j \le N}  
%		% \partial_j g (\hat{X}_{t_k}^{f, [\ell]}) \cdot 
%		% \bigl(
%		%   \bar{X}_{t_k}^{f, [\ell], j} 
%		%   + \widetilde{X}_{t_k}^{f, [\ell], j}
%		% - 2 \hat{X}_{t_k}^{f, [\ell], j} \bigr) \\
%		% & \qquad + \tfrac{1}{16} \sum_{1 \le j_1, j_2 \le N} 
%		% \bigl\{ \partial_{j_1} \partial_{j_2} g (\xi_1)
%		%  + \partial_{j_1} \partial_{j_2} g (\xi_2) 
%		% \bigr\} \cdot
%		% \bigl(
%		% \bar{X}_{t_k}^{f, [\ell], j_1} 
%		% - \widetilde{X}_{t_k}^{f, [\ell], j_1}
%		% \bigr) 
%		% \bigl(
%		% \bar{X}_{t_k}^{f, [\ell], j_2} 
%		% - \widetilde{X}_{t_k}^{f, [\ell], j_2}
%		% \bigr)  \\
%		= \tfrac{1}{16} \sum_{1 \le j_1, j_2 \le N} 
%		\bigl\{ \partial_{j_1} \partial_{j_2} g (\xi_1)
%		+ \partial_{j_1} \partial_{j_2} g (\xi_2) 
%		\bigr\} \cdot
%		\bigl(
%		\bar{X}_{t_k}^{f, [\ell], j_1} 
%		- \widetilde{X}_{t_k}^{f, [\ell], j_1}
%		\bigr) 
%		\bigl(
%		\bar{X}_{t_k}^{f, [\ell], j_2} 
%		- \widetilde{X}_{t_k}^{f, [\ell], j_2}
%		\bigr); \\ 
%		% 
%		& g (\bar{X}_{t_k}^{f, [\ell]}) 
%		- g (\widetilde{X}_{t_k}^{f, [\ell]}) 
%		= \sum_{1 \le j \le N} 
%		\partial_j g (\xi_3) \cdot 
%		\bigl( 
%		\bar{X}_{t_k}^{f, [\ell], j} 
%		- \widetilde{X}_{t_k}^{f, [\ell], j}
%		\bigr),   
%	\end{align*}
	%
%	for some variables $\xi_1, \xi_2, \xi_3 \in \mathbb{R}^N$. Thus, 
we have  under Assumption \ref{ass:coeff} that, for $g \in C_b^2 (\mathbb{R}^N; \mathbb{R})$ and $p \ge 2$, there exist constants $C_1, C_2 > 0$ such that for all $0 \le k \le 2^{\ell -1} -1$: 
	\begin{gather*}
		\mathbb{E} \bigl[ 
		\bigl|
		\tfrac{1}{2} 
		\bigl( 
		g (\bar{X}_{t_k}^{f})
		+  g (\widetilde{X}_{t_k}^{f})
		\bigr) - g (\hat{X}_{t_k}^{f}) 
		\bigr|^p \bigr] 
		\le C_1 \Delta_{\ell - 1}^p, \quad 
		\mathbb{E} \bigl[ 
		\bigl|
		g (\bar{X}_{t_k}^{f})
		-  g (\widetilde{X}_{t_k}^{f}) 
		\bigr|^p \bigr] 
		\le C_2 \Delta_{\ell - 1}^{p/2},   
	\end{gather*} 
	where we made use of the following result: 
	\textcolor{black}{for any $p \ge 2$, there exists $C > 0$ such that  
	\begin{align} \label{eq:m_bd_diff}
		\max_{0 \le k \le 2^{\ell-1}-1}
		\mathbb{E} 
		\bigl[ 
		\| \bar{X}_{t_k}^{f} 
		- \widetilde{X}_{t_k}^{f} \|^p 
		\bigr]
		\le C \Delta_{\ell -1}^{p/2}.
	\end{align} 
	The bound (\ref{eq:m_bd_diff}) is obtained by noticing that
	$$
	\max_{0 \le k \le 2^{\ell-1}-1}
	\mathbb{E} 
	\bigl[ 
	\| \bar{X}_{t_k}^{f} 
	- \widetilde{X}_{t_k}^{f} \|^p 
	\bigr] 
	\le \mathbb{E} 
	\bigl[ \max_{0 \le k \le 2^{\ell-1}-1} 
	\| \bar{X}_{t_k}^{f} 
	- \widetilde{X}_{t_k}^{f} \|^p 
	\bigr]
	$$
	and applying the same argument used in the proof of \cite[Lemma 4.6]{ml_anti} with the strong convergence result (Proposition \ref{prop:s_rate}) to the right-hand-side of the above inequality.}
	Then, we have that:   
	$
		\max_{0 \le k \le 2^{\ell -1} -1}
		\mathbb{E} 
		\bigl[ 
		| \mathcal{E}_{t_{k+1}, t_k}^{(2), j} |^p 
		\bigr]
		\le C_1 \Delta_{\ell-1}^{2p}, 
        \ \ 
		\max_{0 \le k \le 2^{\ell -1} -1} 
		\mathbb{E} 
		\bigl[ 
		| {\mathcal{E}}_{t_{k+1}, t_k}^{(i), j} |^p
		\bigr] = C_2 \Delta_{\ell-1}^{3p/2},  \ \  j \in \{1, 3, 4, 5, 6\} 
	$
	for some positive constants $C_1, C_2$. Setting   
	$
		\hat{\mathcal{M}}_{t_{k+1}, t_k}^{f,  j} 
		\equiv  \sum_{i \in \{1, 3, 4, 5, 6\}} 
		\mathcal{E}_{t_{k+1}, t_k}^{(i), j} 
		+ \tfrac{1}{2} 
		\bigl( 
		\bar{\mathcal{M}}_{t_{k+1}, t_k}^{f, j} 
		+ \widetilde{\mathcal{M}}_{t_{k+1}, t_k}^{\,f, j} 
		\bigr)$ and 
     $
		\hat{\mathcal{N}}_{t_{k+1}, t_k}^{f,  j} 
		\equiv \mathcal{E}_{t_{k+1}, t_k}^{(2), j} 
		+ \tfrac{1}{2} 
		\bigl( 
		\bar{\mathcal{N}}_{t_{k+1}, t_k}^{f, j} 
		+ \widetilde{\mathcal{N}}_{t_{k+1}, t_k}^{\, f, j}
		\bigr),   
	$ 
	we conclude. 
\end{proof}

\begin{lemma} \label{lemma:coare} 
%	Let $1 \le j \le N$, $1 \le \ell \le L$, $0 \le k \le 2^{\ell - 1} - 1$ and $t_k = k \Delta_{\ell-1}$. 
	It holds that: 
	\begin{align}  \label{eq:X_hat_c}
		\begin{aligned} 
			& \hat{X}_{t_{k+1}}^{c, [\ell-1], j} 
			= \hat{X}_{t_k}^{c, [\ell-1], j} 
			+ \sum_{0 \le m \le d} \sigma_{m}^j (\hat{X}_{t_k}^{c, [\ell-1]}) \Delta B_{t_{k+1}, t_k}^{m}  \\ 
			& \qquad + \sum_{1 \le m_1, m_2 \le d} \mathcal{L}_{m_1} \sigma_{m_2}^j (\hat{X}_{t_k}^{c, [\ell-1]}) 
			\Delta \eta^{m_1 m_2}_{t_{k+1}, t_k}  
			+ \hat{\mathcal{M}}_{t_{k+1}, t_k}^{c,  j}  
			+  \hat{\mathcal{N}}_{t_{k+1}, t_k}^{c,  j},   
		\end{aligned}
	\end{align}  
	where the remainder terms $\hat{\mathcal{M}}_{t_{k+1}, t_k}^{c,  j}$ and   
	$\hat{\mathcal{N}}_{t_{k+1}, t_k}^{c,  j}$ 
	satisfy the same properties as 
	$\bar{\mathcal{M}}_{t_{k+1}, t_k}^{f, j}$ and 
	$\bar{\mathcal{N}}_{t_{k+1}, t_k}^{f, j}$ in Lemma \ref{lemma:diff_fine}, respectively. 
%	are specified as follows: 
%	% 
%	\begin{align}  
%		\mathbb{E} 
%		\bigl[
%		\hat{\mathcal{M}}_{t_{k+1}, t_k}^{c, j} | \mathcal{F}_{t_{k}}
%		\bigr]
%		= 0, \quad 0 \le k \le 2^{\ell-1} - 1, 
%	\end{align}
%	% 
%	and for any $p \ge 2$ there exists a constant $C > 0$ such that
%	% 
%	\begin{align} 
%		\max_{0 \le k \le 2^{\ell - 1} - 1} 
%		\mathbb{E} \bigl[ 
%		| \hat{\mathcal{M}}_{t_{k+1}, t_k}^{c, j} |^p  
%		\bigr] \le C \Delta_{\ell-1}^{3p/2}, 
%		\quad 
%		\max_{0 \le k \le 2^{\ell - 1} - 1}  
%		\mathbb{E} \bigl[ 
%		| \hat{\mathcal{N}}_{t_{k+1}, t_k}^{c, j} |^p  
%		\bigr] \le C \Delta_{\ell-1}^{2p}.  
%	\end{align}
	% 
\end{lemma}

\begin{proof}
	For notational simplicity, we omit the subscript ``$[\ell - 1]$" during the proof.  From the discretizations (\ref{eq:X_c}) and (\ref{eq:X_c_a}), we have: 
	\begin{align*}
		& \hat{X}_{t_{k+1}}^{c, j}  
		= \hat{X}_{t_k}^{c, j}  
		+ \sum_{0 \le m \le d} 
		\sigma_{m}^j (\hat{X}_{t_k}^{c}) 
		\Delta B_{t_{k+1},t_k}^m 
		+ \sum_{1 \le m_1, m_2 \le d} 
		\mathcal{L}_{m_1} \sigma_{m_2}^j (\hat{X}_{t_k}^{c}) 
		\Delta \eta^{m_1 m_2}_{t_{k+1},t_k}  
		+ \sum_{1 \le i \le 5} \mathcal{R}_{t_{k+1}, t_k}^{(i), j}, 
	\end{align*}
	where we have defined:
	\begin{align*}
		& \mathcal{R}_{t_{k+1}, t_k}^{(1), j} 
		 = 
		\bigl( 
		\tfrac{1}{2} \sigma_{0}^j (\bar{X}_{t_k}^{c}) 
		+ \tfrac{1}{2} \sigma_{0}^j (\widetilde{X}_{t_k}^{c})  
		- \sigma_{0}^j (\hat{X}_{t_k}^{c})  
		\bigr) \Delta_{\ell-1}  
		 + \bigl( 
		\mathcal{L}_{0} \sigma_{0}^j (\bar{X}_{t_k}^{c}) 
		+ \mathcal{L}_{0} \sigma_{0}^j (\widetilde{X}_{t_k}^{c}) 
		\bigr) \tfrac{\Delta_{\ell-1}^2}{4}; 
		\\[0.2cm]
		& \mathcal{R}_{t_{k+1}, t_k}^{(2), j}   
		= \sum_{1 \le m \le d} 
		\bigl( 
		\tfrac{1}{2} 
		\sigma_{m}^j ( \bar{X}_{t_k}^{c}) 
		+ \tfrac{1}{2} \sigma_{m}^j (\widetilde{X}_{t_k}^{c}) 
		- \sigma_{m}^j (\hat{X}_{t_k}^{c})  
		\bigr) \Delta B_{t_{k+1}, t_k}^{m}; 
		\\
		& \mathcal{R}_{t_{k+1}, t_k}^{(3), j} 
		= \sum_{1 \le m_1, m_2 \le d}  \!  \! 
		\bigl( 
		\tfrac{1}{2} \mathcal{L}_{m_1} \sigma_{m_2}^j 
		( \bar{X}_{t_k}^{c}) 
		+ \tfrac{1}{2} \mathcal{L}_{m_1} \sigma_{m_2}^j
		(\widetilde{X}_{t_k}^{c})  
		- \mathcal{L}_{m_1} \sigma_{m_2}^j  (\hat{X}_{t_k}^{c})  
		\bigr) 
		\Delta \eta_{t_{k+1}, t_k}^{m_1 m_2}; 
		\\
		& \mathcal{R}_{t_{k+1}, t_k}^{(4), j}  
		= \tfrac{1}{4}  \sum_{1 \le m_1 < m_2 \le d} 
		\bigl(
		[\sigma_{m_1}, \sigma_{m_2}]^j 
		(\bar{X}_{t_k}^{c})
		-
		[\sigma_{m_1}, \sigma_{m_2}]^j 
		(\widetilde{X}_{t_k}^{c}) 
		\bigr) 
		\Delta B_{t_{k+1}, t_k}^{m_1} 
		\Delta \widetilde{B}_{t_{k+1}, t_k}^{m_2}; 
		\\
		& \mathcal{R}_{t_{k+1}, t_k}^{(5), j} 
		= \tfrac{1}{2} \! \! \sum_{1 \le m \le d} \! \! \! 
		\bigl\{ (
		\mathcal{L}_{m} \sigma_{0}^j 
		(\bar{X}_{t_k}^{c}) 
		\! + \! \mathcal{L}_{m} \sigma_{0}^j  
		(\widetilde{X}_{t_k}^{c}) ) \Delta \eta_{t_{k+1}, t_k}^{m0} + 
        ( 
		\mathcal{L}_{0} \sigma_{m}^j (\bar{X}_{t_k}^{c}) 
		+ \mathcal{L}_{0} \sigma_{m}^j (\widetilde{X}_{t_k}^{c}) 
		) 
		\Delta \eta_{t_{k+1}, t_k}^{0 m} \bigr\}.   
	\end{align*}
From the argument used in proof of Lemma \ref{lemma:f_anti_diff}, we have that $
\hat{\mathcal{M}}_{t_{k+1}, t_k}^{c, j} 
= \sum_{2 \le i \le 5}
\mathcal{R}_{t_{k+1}, t_k}^{(i), j}$ and 
$\hat{\mathcal{N}}_{t_{k+1}, t_k}^{c, j} = \mathcal{R}_{t_{k+1}, t_k}^{(1), j}
$  satisfy the properties in the statement of Lemma \ref{lemma:coare}, and then we conclude. 
\end{proof} 

\section{Supporting algorithms for multilevel particle filtering}
\label{sec:supp_algo}
We present three algorithms to support the standard/antithetic multilevel particle filtering provided in the main text.  
\allowdisplaybreaks
\begin{algorithm}[H]
	\begin{enumerate}
		\item{Input $M\in\mathbb{N}$ the cardinality of the state-space and two positive probability mass functions
			$W^1(1),\dots,W^1(M)$ and $W^2(1),\dots,W^2(M)$ on $\{1,\dots,M\}$.  Go to 2..}
		\item{Sample $U\sim\mathcal{U}_{[0,1]}$ (continuous uniform distribution on $[0,1]$). If $U<\sum_{i=1}^M\min\{W^1(i),W^2(i)\}$ go to 3.~otherwise go to 4..}
		\item{Sample an index $i_1$ using the probability mass function:
			$$
			\mathbb{P}(i_1) = \tfrac{\min\{W^1(i_1),W^2(i_1)\}}{\sum_{j_1=1}^M\min\{W^1(j_1),W^2(j_1)\}}
			$$
			set $i_2=i_1$ and go to 5..}
		\item{Sample the indices $(i_1,i_2)$ using the probability mass function:
			$$
			\mathbb{P}(i_1,i_2) = \tfrac{W^1(i_1)-\min\{W^1(i_1),W^2(i_1)\}}{1-\sum_{j_1=1}^M\min\{W^1(j_1),W^2(j_1)\}}
			\tfrac{W^2(i_2)-\min\{W^1(i_2),W^2(i_2)\}}{1-\sum_{j_1=1}^M\min\{W^1(j_1),W^2(j_1)\}}
			$$
			and go to 5..}
		\item{Return the indices $(i_1,i_2)\in\{1,\dots,M\}^2$.}
	\end{enumerate}
	\caption{Simulating a Maximal Coupling.}
	\label{alg:max_couple}
\end{algorithm} 

\begin{algorithm}[H]
	\begin{enumerate}
		\item{Input $M\in\mathbb{N}$ the cardinality of the state-space and four positive probability mass functions
			$W^1(1),\dots,W^1(M),\dots,W^4(1),\dots,W^4(M)$ on $\{1,\dots,M\}$.  Go to 2..}
		\item{Sample $U\sim\mathcal{U}_{[0,1]}$. If $U<\sum_{i=1}^M\min\{W^1(i),W^2(i),W^3(i),W^4(i)\}$ go to 3.~otherwise go to 4..}
		\item{Sample an index $i_1$ using the probability mass function:
			$$
			\mathbb{P}(i_1) = \tfrac{\min\{W^1(i_1),W^2(i_1),W^3(i_1),W^4(i_1)\}}{\sum_{j_1=1}^M\min\{W^1(j_1),W^2(j_1),W^3(j_1),W^4(j_1)\}}
			$$
			set $i_4=i_3=i_2=i_1$ and go to 5..}
		\item{Sample the indices $(i_1,\dots,i_4)$ using the probability mass function:
			$$
			\mathbb{P}(i_1,\dots,i_4) = \prod_{1 \le j \le 4} \tfrac{W^j(i_j)-\min\{W^1(i_j),W^2(i_j),W^3(i_j),W^4(i_j)\}}{1-\sum_{j_1=1}^M\min\{W^1(j_1),W^2(j_2),W^3(j_1),W^4(j_1)\}}
			$$
			and go to 5..}
		\item{Return the indices $(i_1,\dots,i_4)\in\{1,\dots,M\}^4$.}
	\end{enumerate}
	\caption{Simulating a Four-Way Coupling.}
	\label{alg:max_couple_four}
\end{algorithm} 

\allowdisplaybreaks 
\begin{algorithm}[H]
	\begin{enumerate}
		\item{Input: level of discretization $\ell\in\mathbb{N}$, final time $T\in\mathbb{N}$ and number of samples $M$. Set $\bar{X}_0^{c,[\ell-1]}(i)=\bar{X}_0^{f,[\ell]}(i)=x$, $i=1,\dots,M$ and $k=1$. Go to 2..}
		\item{Sampling: For $i=1,\dots,M$, simulate 
			$
			\bigl(\bar{X}_k^{c,[\ell-1]}(i),\bar{X}_k^{f,[\ell]}(i)\bigr)|
			\bigl(\bar{x}_{k-1}^{c,[\ell-1]}(i),\bar{x}_{k-1}^{f,[\ell]}(i)\bigr)
			$
			using the coupled dynamics \eqref{eq:X_c} and \eqref{eq:X_f} up-to time 1, with 
			starting point $\bar{x}_{k-1}^{c,[\ell-1]}(i)$, step-size $\Delta_{\ell-1}$ for \eqref{eq:X_c} and
			starting point $\bar{x}_{k-1}^{f,[\ell]}(i)$ and step-size $\Delta_{\ell}$ for \eqref{eq:X_f}.
			Go to 3..}
		\item{Resampling: For $i=1,\dots,M$ compute:
			$$
			w_k^{c,[\ell-1]}(i) := \tfrac{g(\bar{X}_{k}^{c,[\ell-1]}(i),y_k)}{\sum_{1 \le j \le M} g(\bar{X}_{k}^{c,[\ell-1]}(j),y_k)},  \  \ 
			w_k^{f,[\ell]}(i) := \tfrac{g(\bar{X}_{k}^{f,[\ell]}(i),y_k)}{\sum_{1 \le j \le M} g(\bar{X}_{k}^{f,[\ell]}(j),y_k)}.
			$$
			For any $\varphi\in\mathcal{B}_b(\mathbb{R}^{N})$ we have the estimate:
			\begin{equation} \label{eq:cpf_est}
				\pi_k^{[\ell],M}(\varphi) - \pi_k^{[\ell-1],M}(\varphi) 
				:= \sum_{i  =1}^M 
				w_k^{f,[\ell]}(i)\varphi(\bar{X}_k^{f,[\ell]}(i))-
				\sum_{i = 1}^M w_k^{c,[\ell-1]}(i)\varphi(\bar{X}_k^{c,[\ell-1]}(i)).
			\end{equation}
			For $i=1,\dots,M$ sample 
			indices $\left(j^{c,[\ell-1]}(i),j^{f,[\ell]}(i)\right)\in\{1,\dots,M\}^2$
			using Algorithm \ref{alg:max_couple}  with probability mass functions $(w_k^{c,[\ell-1]}(\cdot),w_k^{f,[\ell]}(\cdot))$, cardinality $M$
			and set $\check{X}_k^{c,[\ell-1]}(i)=\bar{X}_k^{c,[\ell-1]}(j^{c,[\ell-1]}(i))$,  $\check{X}_k^{f,[\ell]}(i)=\bar{X}_k^{f,[\ell]}(j^{f,[\ell]}(i))$.
			For $i=1,\dots,M$, set $\bar{X}_k^{c,[\ell-1]}(i)=\check{X}_k^{c,[\ell-1]}(i)$, 
			$\bar{X}_k^{f,[\ell]}(i)=\check{X}_k^{f,[\ell]}(i)$, 
			$k=k+1$, if $k=T+1$ go to 4.~otherwise go to 2..}
		\item{Return the estimates $\pi_1^{[\ell],M}(\varphi) - \pi_1^{[\ell-1],M}(\varphi),\dots,\pi_T^{[\ell],M}(\varphi) - \pi_T^{[\ell-1],M}(\varphi)$ from \eqref{eq:cpf_est}.}
	\end{enumerate}
	\caption{Coupled Particle Filter using the weak second order scheme \eqref{eq:scheme}. The algorithm is stopped at a time $T$, but need not be.}
	\label{alg:cpf}
\end{algorithm}

\section{Case study for small diffusions} \label{sec:B}
%\subsection{Effectiveness of the Proposed Antithetic Scheme} \label{sec:sd}
%  
% We have proposed a new antithetic estimator based on the weak second order scheme (\ref{eq:scheme}). 
We discuss the benefit of the proposed AMLMC estimator over the original one based upon the truncated Milstein scheme proposed by \cite{ml_anti}. Throughout this subsection, we refer to the AMLMC estimator using the weak second order scheme and the truncated-Milstein scheme as AW2 and ATM, respectively. We then compare these two estimators in terms of bias (weak error) and variance. As for the bias,  AW2 achieves weak error 2, while ATM has weak error 1 as we explained in the main text. Regarding variance, Theorem \ref{thm:strong_err_coupl} and \cite[Theorem 4.10]{ml_anti} state that the variance of coupling at level $\ell \in \{1, \ldots, L\}$ is of size $\mathcal{O} (\Delta_{\ell-1}^{2})$ for both AW2 and ATM. However, for the class of small-noise SDEs, which appear very often in applications, we will show analytically that AW2 has a smaller variance compared to ATM. %due to the presence of the small diffusion parameter. 
% 
%\subsubsection{Set-up} \label{sec:sd_set}
% 

We introduce the following small-noise diffusion model as a subclass of (\ref{eq:sde}) in the main text:
\begin{align} \label{eq:small_diffusion} 
	d X_t = 
	\begin{bmatrix}
		d X_{S, t} \\ 
		d X_{R, t} 
	\end{bmatrix}
	=
	\begin{bmatrix}
		\sigma_{S, 0} (X_t)  \\[0.1cm] 
		\sigma_{R, 0} (X_t) 
	\end{bmatrix}  dt 
	+ \sum_{1 \le j \le d}
	\begin{bmatrix}
		\mathbf{0}_{N_S} \\[0.1cm] 
		\mu \, \sigma_{R, j} (X_t) 
	\end{bmatrix} 
	dB_t^j, \quad X_0 = x \in \mathbb{R}^N, 
\end{align}
where the coefficients $\sigma_{S, 0}$, $\sigma_{R, j}, \, 0 \le j \le d,$ are defined as in (\ref{eq:sde}), and the small parameter $\mu \in (0,1)$ is now indicated in the diffusion coefficient. 
%For comparison of the variance of AW2 and ATM, we rewrite here the definition of standard/antithetic discretization of truncated Milstein scheme introduced in \cite{ml_anti}.
We write the standard and antithetic discretization of truncated Milstein scheme introduced in \cite{ml_anti} as $\{ \bar{X}^{\textrm{TM},  [\ell]}_{t} \}_{t \in \mathbf{g}^{f, [\ell]}}$ and 
$\{ \widetilde{X}^{\textrm{TM},  [\ell]}_{t} \}_{t \in \mathbf{g}^{f, [\ell]}}$, respectively,  where  the grids $\mathbf{g}^{f, [\ell]}, \, \ell \le L$ is defined in Section \ref{sec:weak_mlmc_scheme} in the main text.  
%The explicit form of the schemes can be found in Appendix \ref{sec:TM}. 
% 
Then, for a suitable test function $\varphi: \mathbb{R}^N \to \mathbb{R}$ and an integer $L \ge 1$, ATM is defined via the following identity: 
\begin{align*}
	\mathbb{E} [\varphi (\bar{X}_T^{\textrm{TM}, [L]})] = \mathbb{E} [\varphi (\bar{X}_T^{\textrm{TM}, [0]})]
	+ \sum_{1 \le \ell \le L} 
	\mathbb{E} 
	\bigl[ 
	\mathcal{P}^{\textrm{TM}, \varphi}_{f, \ell} - \mathcal{P}^{\textrm{TM}, \varphi}_{c, \ell-1} \bigr], 
\end{align*}
where  
$ 
	\mathcal{P}^{\textrm{TM}, \varphi}_{f, \ell} 
	= \tfrac{1}{2}
	\bigl( 
	\varphi (\bar{X}_{T}^{\textrm{TM}, [\ell]})
	+ \varphi (\widetilde{X}_{T}^{\textrm{TM}, [\ell]})
	\bigr),  \  
	\mathcal{P}_{c, \ell-1}^{\textrm{TM}, \varphi} 
	= \varphi (\bar{X}_{T}^{\textrm{TM}, [\ell-1]}).    
$ 
We also define, for $\ell = 0,1, \ldots, L$, 
$
\hat{X}^{\textrm{TM}, [\ell]}_t 
\equiv \tfrac{1}{2} 
\bigl( 
\bar{X}^{\textrm{TM}, [\ell]}_t  
+
\widetilde{X}^{\textrm{TM}, [\ell]}_t 
\bigr),  \  t \in \mathbf{g}^{f, [\ell]}.   
$ 
For AW2, we make use of the same notation as in (\ref{eq:cpl_weak2}) in Section \ref{sec:weak_mlmc_scheme}, but now the scheme is applied to the SDE (\ref{eq:small_diffusion}) and then the small parameter $\mu$ is incorporated in the definition of the antithetic couplings and the discretization.  
%
%\subsubsection{Analytic results}
% 

We first state the strong convergence of the weak second order scheme and the truncated Milstein scheme with an emphasis on $\mu \in (0,1)$. The proof of the next result is provided in Appendix \ref{app:s_err_sd}. 
\begin{proposition} \label{prop:s_err_sd}
	Work under the model class (\ref{eq:small_diffusion}). 
	Let $\ell = 0, \ldots, L$. For any $p \ge 1$, there exist constants $C_1, C_2 > 0$ independent of $\mu \in (0,1)$ such that:
	\begin{align*}
		\mathbb{E} 
		\bigl[ \max_{t \in \mathbf{g}^{f, [\ell]}} 
		\| X_{t} - \bar{X}_{t}^{f, [\ell]} \|^{2p} 
		\bigr] 
		\le C_1 \mu^{4p} \Delta_\ell^{p}, 
		\qquad  
		\mathbb{E} 
		\bigl[ 
		\max_{t \in \mathbf{g}^{f, [\ell]}}   
		\| 
		X_{t} - \bar{X}_{t}^{\mathrm{TM}, [\ell]} 
		\|^{2p} 
		\bigr] 
		\le C_2 \mu^{4p} \Delta_\ell^{p}. 
	\end{align*}  
\end{proposition}
Notice that both strong error bounds are of order 1 w.r.t.~the step-size $\Delta_\ell$ in agreement with Table~\ref{table:literature} in the main text, but now the effect of  $\mu$ on the bounds is made clear. We move on to moment estimates for the antithetic couplings of AW2 and ATM. We focus on the couplings at a fixed level $\ell = 1, \ldots, L$. Then, it holds from \cite[Lemma 2.2.]{ml_anti} and Lemma \ref{lemma:s_bd_cpl} that for any $\varphi \in C_b^2 (\mathbb{R}^N; \mathbb{R})$ and $p \ge 2$,  
\begin{align}
	& \begin{aligned} 
		& \mathbb{E} 
		\bigl[ 
		\bigl( 
		\mathcal{P}_{f, \ell}^{\textrm{TM}, \varphi} 
		- \mathcal{P}_{c, \ell -1}^{\textrm{TM},\varphi} 
		\bigr)^p 
		\bigr] 
		\le C_1 \mathbb{E} \bigl[ 
		\| \hat{X}_{T}^{\textrm{TM},  [\ell]} - \bar{X}_{T}^{\textrm{TM},  [\ell-1]} \|^p  \bigr]  \\
		& \qquad \qquad  + C_2 \mathbb{E} \bigl[ \| \bar{X}_{T}^{\textrm{TM}, [\ell]}
		- \widetilde{X}_{T}^{\textrm{TM}, [\ell]} \|^{2p} \bigr]; 
	\end{aligned} 
	\label{eq:cpl_mil_sd} \\[0.1cm] 
	& \begin{aligned} 
		& \mathbb{E} \bigl[ 
		\bigl( 
		\mathcal{P}_{f, \ell}^{\varphi} 
		- \mathcal{P}_{c, \ell -1}^{\varphi} 
		\bigr)^p \bigr]  
		\le \widetilde{C}_1 \mathbb{E} 
		\bigl[ \| \hat{X}_{T}^{f, [\ell]} - \hat{X}_{T}^{c, [\ell-1]} \|^p  \bigr]
		+\widetilde{ C}_2 \mathbb{E} \bigl[ \| \bar{X}_{T}^{f, [\ell]}
		- \widetilde{X}_{T}^{f, [\ell]} \|^{2p} \bigr]  \\[0.1cm] 
		& \qquad  \qquad 
		+ \widetilde{C}_3 \mathbb{E} \bigl[ \| \bar{X}_{T}^{c, [\ell-1]}
		- \widetilde{X}_{T}^{c, [\ell-1]} \|^{2p} \bigr], 
	\end{aligned} 
	\label{eq:cpl_w2_sd}
\end{align}
for some positive constants $C_1, C_2, \widetilde{C}_1, \widetilde{C}_2, \widetilde{C}_3$ independent of $\mu \in (0,1)$. Due to Proposition \ref{prop:s_err_sd}, the second term in the R.H.S. of (\ref{eq:cpl_mil_sd}) and the second/third term in the R.H.S. of (\ref{eq:cpl_w2_sd}) are bounded by $C \mu^{4p} \Delta_{\ell-1}^p$, where $C > 0$ is a constant independent of $\mu$. The next result is critical in highlighting analytically a difference between AW2 and ATM, with its proof in Section \ref{app:m_bd_sd}. 
\begin{theorem} \label{thm:m_bd_sd}
	Work under the model class (\ref{eq:small_diffusion}). 
	Let $1 \le \ell \le L$. For any $p \ge 2$, there exist constants $C_1, C_2 > 0$ independent of $\mu \in (0,1)$ such that:
	\begin{align}
		&		\mathbb{E} 
		\bigl[ \, 
		\max_{t \in \mathbf{g}^{c, [\ell-1]}} 
		\| \hat{X}^{f, [\ell]}_{t} - \hat{X}^{c, [\ell-1]}_{t} 
		\|^p
		\bigr] \le C_1 \mu^{p} \, \Delta_{\ell-1}^p; 
		\label{eq:w2_sd} \\ 
		& \mathbb{E} 
		\bigl[ \, 
		\max_{t \in \mathbf{g}^{c, [\ell-1]}} 
		\| 
		\hat{X}^{\mathrm{TM}, [\ell]}_{t}
		- \bar{X}^{\mathrm{TM}, [\ell-1]}_{t} 
		\|^p
		\bigr] \le C_2 \, \Delta_{\ell-1}^p. 
		\label{eq:tmil_sd} 
	\end{align} 
\end{theorem}
The bounds (\ref{eq:cpl_mil_sd}), (\ref{eq:cpl_w2_sd}) and Theorem \ref{thm:m_bd_sd} provide the following result.
\begin{corollary}
	Work under the model class (\ref{eq:small_diffusion}).  Let $\varphi \in C_b^2 (\mathbb{R}^N; \mathbb{R})$ and $1 \le l \le L$. For any $p \ge 2$, there exist constants $C_{\mathrm{ATM}}, C_{\mathrm{AW2}} > 0$ independent of $\mu \in (0,1)$ such that: 
	\begin{align*}
		\mathbb{E} 
		\bigl[ 
		\bigl(
		\mathcal{P}^{\varphi}_{f, \ell} 
		- \mathcal{P}^{\varphi}_{c, \ell-1}  \bigr)^p 
		\bigr] \le C_{\mathrm{AW2}} \, \mu^{p} 
		\Delta_{\ell -1}^p, \  \ \ 
		\mathbb{E} 
		\bigl[
		\bigl(\mathcal{P}^{\mathrm{TM}, \varphi}_{f, \ell} - \mathcal{P}^{\mathrm{TM}, \varphi}_{c, \ell-1}  \bigr)^p 
		\bigr] \le C_{\mathrm{ATM}} \, \Delta_{\ell -1}^p.  
	\end{align*} 
\end{corollary}
From the above result, we conclude that within the small-noise SDE class  (\ref{eq:small_diffusion}) the variance of couplings for AW2 can be smaller than that of ATM due to the diffusion parameter $\mu \in (0,1)$. Thus, for such a class, it is expected that the smaller variance together with the smaller bias, i.e. the second order weak convergence, contributes to effective reduction of computational cost in the AMLMC with the weak second order scheme compared to the AMLMC based on the truncated Milstein scheme. We here briefly remark that the improvement in the variance bound for AW2 in Theorem \ref{thm:m_bd_sd} comes from the inclusion of higher-order stochastic Taylor expansion terms from the drift coefficient $\sigma_0$ in scheme (\ref{eq:scheme}) (note that the \ref{eq:t_mil} scheme contains higher order terms from the diffusion coefficients only). 

\subsection{Proof of Proposition \ref{prop:s_err_sd}} 
\label{app:s_err_sd}
\begin{proof} 
	We will show the strong convergence bound only for the \ref{eq:scheme} scheme under the small diffusion regime (\ref{eq:small_diffusion}) because the same discussion applies to the case of the truncated Milstein scheme. We follow the proof of Proposition \ref{prop:s_rate} in Appendix \ref{app:s_rate} and introduce  
$  \mathcal{S}_i^{(\mu)} \equiv  \mathbb{E}  \bigl[ \max_{0 \le k \le i} 	\| X_{t_k} - \bar{X}_{t_k}  \|^p \bigr],  \  1 \le i \le 2^{\ell}$.  
	Then, we show that for any $p \ge 2$, there exists a constant $C > 0$ independent of $\mu \in (0,1)$ such that: 
	\begin{align} \label{eq:gw_sd}
		\mathcal{S}_i^{(\mu)} 
		\le C 
		\Bigl( 
		\mu^{2p} \Delta_{\ell}^{p/2}
		+ 
		\sum_{0 \le n \le i-1} \mathcal{S}_n^{(\mu)} \cdot \Delta_\ell 
		\Bigr).  
	\end{align}  
	(\ref{eq:gw_sd}) indeed holds by adjustment of $\mu \in (0,1)$  for the diffusion coefficients $\sigma_j, \, 1 \le j \le d$. In particular, now the  term $	\mathcal{T}_i^{(4), j}$ appearing in the proof of Proposition \ref{prop:s_rate} (Appendix \ref{app:s_rate} in the main text) is bounded as: 
$
\mathcal{T}_i^{(4), j} \le C \mu^{2p} \Delta_{\ell}^{p/2},
$ 
	for some constant $C > 0$ independent of $\mu \in (0,1)$. Thus, we obtain (\ref{eq:gw_sd}) and conclude from the discrete Gr\"onwall's inequality. 
\end{proof} 

\subsection{Technical results for Theorem \ref{thm:m_bd_sd}}
\label{app:tech_sd}
Throughout this section, for notational simplicity, we write 
%we omit the notation `$[\ell-1]$' and `$[\ell]$' appearing in the original presentation of the numerical schemes w.r.t to \ref{eq:scheme}, e.g.,  
$\bar{X}^{f}_{t} \equiv \bar{X}^{f, [\ell]}_{t}, \ \bar{X}^{c}_{t} \equiv \bar{X}^{c, [\ell-1]}_{t}, \hat{X}^{f}_{t} \equiv \hat{X}^{f, [\ell]}_{t}$ and $\hat{X}^{c}_{t} \equiv \hat{X}^{c, [\ell-1]}_{t}$. We also set $1 \le j \le N$, $1 \le \ell \le L$,  $0 \le k \le 2^{\ell - 1} - 1$ and $t_k = k \Delta_{\ell -1}$. 
\begin{lemma} \label{lemma:diff_fine_w2} 
	Work under the model class (\ref{eq:small_diffusion}).  
%	Let $1 \le j \le N$, $1 \le \ell \le L$, $0 \le k \le 2^{\ell - 1} - 1$ and $t_k = k \Delta_{\ell -1}$. 
	It holds that:
	\begin{align}
		& \bar{X}^{f, j}_{t_{k+1}} 
		= \bar{X}^{f,  j}_{t_k} 
		+  \sum_{0 \le m \le d} \mu^{\mathbf{1}_{m > 0}}
		\sigma_{m}^j 
		\bigl( \bar{X}^{f}_{t_k} \bigr) 
		\Delta B_{t_{k+1}, t_k}^{m}  \nonumber \\ 
		& \quad 
		+ \mathcal{L}_{0} \sigma_{0}^j 
		\bigl(
		\bar{X}^{f}_{t_k}
		\bigr) 
		\Delta \eta_{t_{k+1}, t_k}^{00}
		+ \mu^2  \sum_{1  \le m_1, m_2 \le d} 
		 \mathcal{L}_{m_1} \sigma_{m_2}^j  (\bar{X}_{t_k}^{f}) \Delta \eta_{t_{k+1}, t_k}^{m_1 m_2} \nonumber \\ 
		& \quad 
		- \tfrac{\mu^2}{2}  \sum_{1 \le m_1, m_2 \le d} 
		\mathcal{L}_{m_1} \sigma_{m_2}^j 
		\bigl( \bar{X}^{f}_{t_k} \bigr)
		\bigl( 
		\Delta B_{t_{k+1}, t_{k+1/2}}^{m_1}  
		\Delta B_{t_{k+1/2}, t_k}^{m_2}
		-
		\Delta B_{t_{k+1/2}, t_k}^{m_1} 
		\Delta B_{t_{k+1}, t_{k+1/2}}^{m_2} 
		\bigr) \nonumber \\ 
		& \quad 
		+ \tfrac{\mu^2}{2} 
		\sum_{1 \le  m_1 <  m_2 \le d} 
		\bigl[\sigma_{m_1}, \sigma_{m_2} \bigr]^j
		\bigl( \bar{X}^{f}_{t_k} \bigr) 
		\bigl( 
%		\Delta B^{m_1}_{t_{k+1/2}, t_k} 
%		\Delta \widetilde{B}^{m_2}_{t_{k+1/2}, t_k} 
		\Delta \widetilde{A}_{t_{k+1/2}, t_k}^{m_1 m_2}
		+ 
%		\Delta B^{m_1}_{t_{k+1}, t_{k+1/2}}
%		\Delta \widetilde{B}^{m_2}_{t_{k+1}, t_{k+1/2}} 
        \Delta \widetilde{A}_{t_{k+1}, t_{k+1/2}}^{m_1 m_2} 
		 \bigr)  + \bar{\mathcal{M}}^{f, j}_{t_{k+1}, t_k} 
		+ \bar{\mathcal{N}}^{f, j}_{t_{k+1}, t_k},  \nonumber 
%		\label{eq:X_f_sd}
	\end{align}
	where the remainder terms are such that 
	$\mathbb{E} \bigl[ \bar{\mathcal{M}}^{f, j}_{t_{k+1}, t_k} | \mathcal{F}_{t_{k}} \bigr] = 0$,  
	and for any $p \ge 2$ there exist constants $C_1, C_2 > 0$ independent of $\mu \in (0,1)$ so that: 
	\begin{align}
		\max_{0 \le k \le 2^{\ell-1} - 1} 
		\mathbb{E} \bigl[ | \bar{\mathcal{M}}^{f, j}_{t_{k+1}, t_k}|^p \bigr] 
		\le C_1 \mu^{p} \Delta_{\ell-1}^{3p/2}, \quad  
		\max_{0 \le k \le 2^{\ell-1} - 1} 
		\mathbb{E} \bigl[ | \bar{\mathcal{N}}^{f, j}_{t_{k+1}, t_k} |^p \bigr] 
		\le C_2 \mu^{2p} \Delta_{\ell -1}^{2p}. 
	\end{align}
	Similarly, it holds that:
	\begin{align}  
		& \widetilde{X}^{f, j}_{t_{k+1}} 
		= \widetilde{X}^{f,  j}_{t_k} 
		+ \sum_{0 \le m \le d} \mu^{\mathbf{1}_{m > 0}}
		\sigma_{m}^j \bigl(\widetilde{X}^{f}_{t_k} \bigr) 
		\Delta B_{t_{k+1}, t_k}^{m}  \nonumber \\ 
		& \quad 
		+ \mathcal{L}_{0} \sigma_{0}^j 
		\bigl(
		\widetilde{X}^{f}_{t_k}
		\bigr) 
		\Delta \eta_{t_{k+1}, t_k}^{00} 
		+ \mu^2  \sum_{1 \le m_1, m_2 \le d}  
		 \mathcal{L}_{m_1} \sigma_{m_2}^j  (\widetilde{X}_{t_k}^{f}) \Delta \eta_{t_{k+1}, t_k}^{m_1 m_2} \nonumber \\ 
		& \quad 
		+ \tfrac{\mu^2}{2} \sum_{1 \le m_1, m_2 \le d} 
		\mathcal{L}_{m_1} \sigma_{m_2}^j 
		\bigl( \widetilde{X}^{f}_{t_k} \bigr)
		\bigl( 
		\Delta B_{t_{k+1}, t_{k+1/2}}^{m_1}  
		\Delta B_{t_{k+1/2}, t_k}^{m_2}
		-
		\Delta B_{t_{k+1/2}, t_k}^{m_1} 
		\Delta B_{t_{k+1}, t_{k+1/2}}^{m_2} 
		\bigr) 
		\nonumber \\ 
		& \quad 
		- \tfrac{\mu^2}{2} 
		\sum_{1 \le  m_1 <  m_2 \le d} 
		\bigl[\sigma_{m_1}, \sigma_{m_2} \bigr]^j
		\bigl( \widetilde{X}^{f}_{t_k}  \bigr) 
		\bigl( 
%		\Delta B^{m_1}_{t_{k+1/2}, t_k} 
%		\Delta \widetilde{B}^{m_2}_{t_{k+1/2}, t_k} 
        \Delta \widetilde{A}_{t_{k+1/2}, t_k}^{m_1 m_2}
		+ 
%		\Delta B^{m_1}_{t_{k+1}, t_{k+1/2}}
%		\Delta \widetilde{B}^{m_2}_{t_{k+1}, t_{k+1/2}} 
        \Delta \widetilde{A}_{t_{k+1}, t_{k+1/2}}^{m_1 m_2} 
		\bigr) + \widetilde{\mathcal{M}}^{f,  j}_{t_{k+1}, t_k} 
		+ \widetilde{\mathcal{N}}^{f,  j}_{t_{k+1}, t_k},  \nonumber 
%		\label{eq:X_f_anti_sd}
	\end{align}
	where the remainder terms 
	$\widetilde{\mathcal{M}}_{t_{k+1}, t_k}^{f, j}$ and $\widetilde{\mathcal{N}}_{t_{k+1}, t_k}^{f, j}$ 
	satisfy the same properties as 
	$\bar{\mathcal{M}}_{t_{k+1}, t_k}^{f, j}$ and 
	$\bar{\mathcal{N}}_{t_{k+1}, t_k}^{f, j}$, respectively. 
\end{lemma} 
\begin{proof}
We show only for $\bar{X}^{f, j}$ because the latter follows from a similar argument. Following the argument used in the proof of Lemma \ref{lemma:diff_fine} in Appendix \ref{app:tech}, we obtain: 
	\begin{align} 
		& \bar{X}_{t_{k+1}}^{{f}, j}
		= \bar{X}_{t_k}^{{f}, j}
		+ \sum_{0 \le m \le d} 
		\mu^{\mathbf{1}_{m > 0}} \sigma_{m}^j (\bar{X}_{t_k}^{f}) 
		\Delta B_{t_{k+1},t_k}^{m}
		+ \mathcal{L}_{0} \sigma_{0}^j
		(\bar{X}_{t_{k}}^{f}) 
		\bigl\{ 
		\Delta_\ell^2
		+ 
		\Delta \eta_{t_{k+1/2}, t_k}^{00}
		+  
		\Delta \eta_{t_{k+1}, t_{k+1/2}}^{00} 
		\bigr\}
		\nonumber \\ 
		& \quad  
		+ \mu^2 \sum_{1 \le m_1, m_2 \le d} 
		\mathcal{L}_{m_1} \sigma_{m_2}^j
		(\bar{X}_{t_{k}}^{f}) 
		\bigl\{ 
		\Delta B_{t_{k+1/2}, t_{k}}^{m_1}
		\Delta B_{t_{k+1}, t_{k+1/2}}^{m_2}
		+ 
		\Delta \eta_{t_{k+1/2}, t_k}^{m_1 m_2}
		+  
		\Delta \eta_{t_{k+1}, t_{k+1/2}}^{m_1 m_2}  
		\bigr\}
		\nonumber \\  
		& \quad 
		+ \tfrac{\mu^2}{2} 
		\sum_{1 \le m_1 < m_2 \le d} 
		[ \sigma_{m_1}, \sigma_{m_2}]^j 
		\bigl(\bar{X}_{t_k}^{f} \bigr)  
		\bigl\{ 
%		\Delta B^{m_1}_{t_{k+1/2}, t_k}
%		\Delta \widetilde{B}^{m_2}_{t_{k+1/2}, t_k}  
        \Delta \widetilde{A}_{t_{k+1/2}, t_k}^{m_1 m_2}
		+ 
%		\Delta B^{m_1}_{t_{k+1}, t_{k+1/2}}
%		\Delta \widetilde{B}^{m_2}_{t_{k+1}, t_{k+1/2}}
        \Delta \widetilde{A}_{t_{k+1}, t_{k+1/2}}^{m_1 m_2} 
		\bigr\} 
		+ \mathcal{M}_{t_{k+1}, t_k}^{f, j} 
		+ \mathcal{N}_{t_{k+1}, t_k}^{f, j}, \nonumber 
	\end{align}  
	where the residuals $\mathcal{M}_{t_{k+1}, t_k}^{f, j}$ and $\mathcal{N}_{t_{k+1}, t_k}^{f, j}$ satisfy the same properties as $\bar{\mathcal{M}}_{t_{k+1}, t_k}^{f, j}$ and $\bar{\mathcal{N}}_{t_{k+1}, t_k}^{f, j}$, respectively. 
%	: $\mathbb{E} \bigl[\mathcal{M}_{t_{k+1}, t_k}^{f, (\mu), j} | \mathcal{F}_{t_k} \bigr] = 0$ and for any $p \ge 2$, 
%	% 
%	\begin{align}
%		\mathbb{E} 
%		\bigl[
%		\bigl| 
%		\mathcal{M}_{t_{k+1}, t_k}^{f, (\mu), j} 
%		\bigr|^p 
%		\bigr] 
%		\le C_1 \, \mu^{p} \Delta_{\ell-1}^{3p/2}, 
%		\qquad    
%		\mathbb{E} 
%		\bigl[
%		\bigl| 
%		\mathcal{N}_{t_{k+1}, t_k}^{f, (\mu), j} 
%		\bigr|^p 
%		\bigr] \le C_2 \, \mu^{2p} \Delta_{\ell-1}^{2p},  
%	\end{align} 
%	% 
%	for some constants $C_1, C_2 > 0$. 
	Noticing that: 
	\allowdisplaybreaks 
	\begin{gather*}
		\Delta_\ell^2 + \Delta \eta_{t_{k+1/2}, t_k}^{00} + \Delta \eta_{t_{k+1}, t_{k+1/2}}^{00}  
		= 2 \Delta_\ell^2 = \tfrac{\Delta_{\ell-1}^2}{2} 
		= \Delta \eta_{t_{k+1}, t_k}^{00};
		\nonumber \\[0.2cm] 
		\Delta B_{t_{k+1/2}, t_k}^{m_1}
		\Delta B_{t_{k+1}, t_{k+1/2}}^{m_2}  
		+ \tfrac{1}{2} \Delta B^{m_1}_{t_{k+1/2}, t_k}
		\Delta B^{m_2}_{t_{k+1/2}, t_k} 
		+ \tfrac{1}{2} 
		\Delta B^{m_1}_{t_{k+1}, t_{k+1/2}}
		\Delta B^{m_2}_{t_{k+1}, t_{k+1/2}}  \\
		= \tfrac{1}{2} 
		\Delta B^{m_1}_{t_{k+1}, t_{k}}  
		\Delta B^{m_2}_{t_{k+1}, t_{k}} 
		+ 
		\tfrac{1}{2} 
		\Delta B^{m_1}_{t_{k+1/2}, t_{k}}  
		\Delta B^{m_2}_{t_{k+1}, t_{k+1/2}}  
		- 
		\tfrac{1}{2} 
		\Delta B^{m_1}_{t_{k+1}, t_{k+1/2}}  
		\Delta B^{m_2}_{t_{k+1/2}, t_{k}},     
	\end{gather*} 
	we obtain the difference equation for $\bar{X}^{f, j}$ given in the statement and thus conclude. 
%	we obtain (\ref{eq:X_f_sd}). (\ref{eq:X_f_anti_sd}) is deduced from a similar argument and then we omit the details. The proof is now complete. 
	% 
\end{proof}

\begin{lemma} \label{lemma:diff_fine_mil} 
	Work under the model class (\ref{eq:small_diffusion}).   
%	Let $1 \le j \le N$, $1 \le \ell \le L$, $0 \le k \le 2^{\ell - 1} - 1$ and $t_k = k \Delta_{\ell -1}$. 
	It holds that:
	\allowdisplaybreaks
	\begin{align*}
	        & \begin{aligned} 
		& \bar{X}^{\mathrm{TM}, [\ell],  j}_{t_{k+1}} 
		= \bar{X}^{\mathrm{TM}, [\ell],  j}_{t_k} 
		+ \sum_{0 \le m \le d} 
		\mu^{\mathbf{1}_{m > 0}} \sigma_{m}^j 
		\bigl(
		\bar{X}^{\mathrm{TM}, [\ell]}_{t_k}
		\bigr) 
		\Delta B_{t_{k+1}, t_k}^{m} 
		+ \tfrac{1}{2} 
		\mathcal{L}_{0} \sigma_{0}^j 
		\bigl(
		\bar{X}^{\mathrm{TM}, [\ell]}_{t_k} 
		\bigr) 
		\Delta \eta_{t_{k+1}, t_k}^{00}  \nonumber  \\ 
		& \quad + \mu^2 \sum_{1 \le m_1, m_2 \le d}
		\mathcal{L}_{m_1} \sigma_{m_2}^j 
		\bigl(
		\bar{X}^{\mathrm{TM}, [\ell]}_{t_k} 
		\bigr) 
		\Delta \eta_{t_{k+1}, t_k}^{m_1 m_2} \nonumber \\ 
		& \quad 
		- \tfrac{\mu^2}{2} \sum_{1 \le m_1, m_2 \le d} 
		\mathcal{L}_{m_1} \sigma_{m_2}^j 
		\bigl( 
		\bar{X}^{\mathrm{TM}, [\ell]}_{t_k} 
		\bigr)
		\bigl( 
		\Delta B_{t_{k+1}, t_{k+1/2}}^{m_1}  
		\Delta B_{t_{k+1/2}, t_k}^{m_2}
		-
		\Delta B_{t_{k+1/2}, t_k}^{m_1} 
		\Delta B_{t_{k+1}, t_{k+1/2}}^{m_2} 
		\bigr) \nonumber \\ 
		& \quad 
		+ \bar{\mathcal{M}}^{\mathrm{TM},  j}_{t_{k+1}, t_k} 
		+ \bar{\mathcal{N}}^{\mathrm{TM}, j}_{t_{k+1}, t_k};  
		\end{aligned} \\[0.1cm] 
%
	% 
%	are specified as follows: 
%	% 
%	\begin{align}
%		\mathbb{E} 
%		\bigl[ 
%		\bar{\mathcal{M}}^{\mathrm{T}\textrm{-}\mathrm{Mil}, (\mu), j}_{t_{k+1}, t_k} 
%		| \mathcal{F}_{t_{k}} 
%		\bigr] = 0,  
%		\quad  0 \le k \le 2^{\ell-1} - 1,
%	\end{align}
%	% 
%	and for any $p \ge 2$ there exist constants $C_1, C_2 > 0$ independent of $\mu \in (0,1)$ such that 
%	% 
%	\begin{gather*}
%		\max_{0 \le k \le 2^{\ell-1} - 1} 
%		\mathbb{E} 
%		\bigl[ 
%		| 
%		\bar{\mathcal{M}}^{\mathrm{T}\textrm{-}\mathrm{Mil}, (\mu), j}_{t_{k+1}, t_k}
%		|^p 
%		\bigr] 
%		\le C_1 \mu^p \Delta_{\ell-1}^{3p/2}, \quad 
%		\max_{0 \le k \le 2^{\ell-1} - 1} 
%		\mathbb{E} 
%		\bigl[ 
%		| \bar{\mathcal{N}}^{\mathrm{T}\textrm{-}\mathrm{Mil}, (\mu), j}_{t_{k+1}, t_k} |^p 
%		\bigr] 
%		\le C_2 \mu^{2p} \Delta_{\ell -1}^{2p}. 
%	\end{gather*}
	% 
	% 
%	\allowdisplaybreaks
%	\begin{align*}
		&
		\begin{aligned} 
		 & \widetilde{X}^{\mathrm{TM}, [\ell], j}_{t_{k+1}} 
		= \widetilde{X}^{\mathrm{TM}, [\ell], j}_{t_k} 
		+ \sum_{0 \le m \le d} 
		\mu^{\mathbf{1}_{m > 0}} \sigma_{m}^j 
		\bigl(
		\widetilde{X}^{\mathrm{TM}, [\ell]}_{t_k}
		\bigr) 
		\Delta B_{t_{k+1}, t_k}^{m} 
		+ \tfrac{1}{2} 
		\mathcal{L}_{0} \sigma_{0}^j 
		\bigl(
		\widetilde{X}^{\mathrm{TM}, [\ell]}_{t_k} 
		\bigr) 
		\Delta \eta_{t_{k+1}, t_k}^{00}  \\ 
		& \quad + \mu^2 \sum_{1 \le m_1, m_2 \le d}
		\mathcal{L}_{m_1} \sigma_{m_2}^j 
		\bigl(
		\widetilde{X}^{\mathrm{TM}, [\ell]}_{t_k}  
		\bigr) 
		\Delta \eta_{t_{k+1}, t_k}^{m_1 m_2} \nonumber \\ 
		& \quad 
		+ \tfrac{\mu^2}{2} \sum_{1 \le m_1, m_2 \le d} 
		\mathcal{L}_{m_1} \sigma_{m_2}^j 
		\bigl( 
		\widetilde{X}^{\mathrm{TM}, [\ell]}_{t_k}  
		\bigr)
		\bigl( 
		\Delta B_{t_{k+1}, t_{k+1/2}}^{m_1}  
		\Delta B_{t_{k+1/2}, t_k}^{m_2}
		-
		\Delta B_{t_{k+1/2}, t_k}^{m_1} 
		\Delta B_{t_{k+1}, t_{k+1/2}}^{m_2} 
		\bigr) \nonumber \\ 
		& \quad 
		+ \widetilde{\mathcal{M}}^{\mathrm{TM},  j}_{t_{k+1}, t_k} 
		+ \widetilde{\mathcal{N}}^{\mathrm{TM}, j}_{t_{k+1}, t_k},  \nonumber  
		\end{aligned} 
	\end{align*}
	where the remainder terms $ \bar{\mathcal{M}}^{\mathrm{TM}, j}_{t_{k+1}, t_k}$ and $\widetilde{\mathcal{M}}_{t_{k+1}, t_k}^{\mathrm{TM}, j}$ satisfy the same properties as  
	$\bar{\mathcal{M}}_{t_{k+1}, t_k}^{f, j}$ in Lemma \ref{lemma:diff_fine_w2},  and $\bar{\mathcal{N}}^{\mathrm{TM},  j}_{t_{k+1}, t_k}$  and $\widetilde{\mathcal{N}}_{t_{k+1}, t_k}^{\mathrm{TM}, j}$ 
	do as $\bar{\mathcal{N}}_{t_{k+1}, t_k}^{f, j}$ given in Lemma \ref{lemma:diff_fine_w2}. 
\end{lemma}  

\begin{proof}
	The proof follows the same argument of \cite[Lemma 4.7]{ml_anti} and Lemma \ref{lemma:diff_fine_w2}. 
	We here only mention that in the case of truncated Milstein scheme, the fine discretization produces the deterministic $\mathcal{O} (\Delta_n^2)$-term as  
	$ \mathcal{L}_0 \sigma_0 (\bar{X}_{t_k}^{\textrm{TM}, [\ell]}) \Delta_\ell^2  = \tfrac{1}{2} \mathcal{L}_0 \sigma_0 (\bar{X}_{t_k}^{\textrm{TM}, [\ell]}) \Delta \eta_{t_{k+1}, t_k}^{00}$, 
	whereas in the case of the weak second order scheme, we have that  
	$\mathcal{L}_0 \sigma_0 (\bar{X}_{t_k}^{f})  \times \bigl\{ \Delta_\ell^2  + \Delta \eta_{t_{k+1/2}, t_k}^{00} +  \Delta \eta_{t_{k+1}, t_{k+1/2}}^{00} \bigr\}
	= \mathcal{L}_0 \sigma_0 (\bar{X}_{t_k}^{f}) \Delta \eta_{t_{k+1}, t_k}^{00}$. 
	The rest of the proof relies on the same argument as in the proof of Lemma \ref{lemma:diff_fine_w2}. 
\end{proof}

\begin{lemma} \label{lemma:f_anti_diff_sd}
	Work under the model class (\ref{eq:small_diffusion}). 
%	 Let $1 \le j \le N$, $1 \le \ell \le L$, $0 \le k \le 2^{\ell - 1} - 1$ and $t_k = k \Delta_{\ell -1}$. 
	It holds that: 
	\begin{align}
		\begin{aligned}
			& \hat{X}^{f,  j}_{t_{k+1}} 
			= \hat{X}^{f, j}_{t_k}
			+ 
			\sum_{0 \le m \le d} 
			\mu^{\mathbf{1}_{m > 0}}
			\sigma_{m}^j (\hat{X}^{f}_{t_k}) 
			\Delta B_{t_{k+1}, t_k}^{m} 
			+ \mathcal{L}_{0} \sigma_{0}^j 
			(\hat{X}_{t_k}^{f}) \Delta \eta_{t_{k+1}, t_k}^{00}  \\ 
			& \quad 
			+ \mu^2 \sum_{1 \le m_1, m_2 \le d} 
			\mathcal{L}_{m_1} \sigma_{m_2}^j 
			( \hat{X}_{t_k}^{f} ) 
			\Delta \eta_{t_{k+1}, t_k}^{m_1 m_2} 
			+ \hat{\mathcal{M}}_{t_{k+1}, t_k}^{f, j} 
			+ \hat{\mathcal{N}}_{t_{k+1}, t_k}^{f, j};  
		\end{aligned} \label{eq:hat_X_f_sd} \\ 
		\begin{aligned}
			& \hat{X}_{t_{k+1}}^{c, j} 
			= \hat{X}_{t_k}^{c, j} 
			+ \sum_{0 \le m \le d} \mu^{\mathbf{1}_{m>0}}
			\sigma_{m}^j (\hat{X}_{t_k}^{c}) \Delta B_{t_{k+1}, t_k}^{m}
			+ \mathcal{L}_{0} \sigma_{0}^j (\hat{X}_{t_k}^{c}) 
			\Delta \eta^{00}_{t_{k+1}, t_k} \\
			& \quad 
			+ \mu^2 \sum_{1 \le m_1, m_2 \le d} 
			\mathcal{L}_{m_1} \sigma_{m_2}^j (\hat{X}_{t_k}^{c}) 
			\Delta \eta^{m_1 m_2}_{t_{k+1}, t_k}  
			+ \hat{\mathcal{M}}_{t_{k+1}, t_k}^{c, j}  
			+  \hat{\mathcal{N}}_{t_{k+1}, t_k}^{c,  j},    
		\end{aligned} \label{eq:hat_X_c_sd} 
	\end{align}  
	where the remainder terms $ \hat{\mathcal{M}}^{f, j}_{t_{k+1}, t_k}$ and $\hat{\mathcal{M}}_{t_{k+1}, t_k}^{c, j}$ satisfy the same properties as  
	$\bar{\mathcal{M}}_{t_{k+1}, t_k}^{f, j}$ in Lemma \ref{lemma:diff_fine_w2},  and $ \hat{\mathcal{N}}^{f, j}_{t_{k+1}, t_k}$ and $\hat{\mathcal{N}}_{t_{k+1}, t_k}^{c, j}$  
	do as $\bar{\mathcal{N}}_{t_{k+1}, t_k}^{f, j}$ given in Lemma \ref{lemma:diff_fine_w2}.
%	where the remainder terms $\hat{\mathcal{M}}_{t_{k+1}, t_k}^{c, j}$ and  $\hat{\mathcal{N}}_{t_{k+1}, t_k}^{c, j}$ have the same properties as  $ \bar{\mathcal{M}}^{f, j}_{t_{k+1}, t_k}$ and $\bar{\mathcal{N}}^{f,  j}_{t_{k+1}, t_k}$, respectively provided in Lemma \ref{lemma:diff_fine_w2}.  
	% 
%	where the remainder terms $\hat{\mathcal{M}}_{t_{k+1}, t_k}^{f, j}$ and  $\hat{\mathcal{N}}_{t_{k+1}, t_k}^{f, j}$ have the same properties as  $ \bar{\mathcal{M}}^{f, j}_{t_{k+1}, t_k}$ and $\bar{\mathcal{N}}^{f,  j}_{t_{k+1}, t_k}$, respectively provided in Lemma \ref{lemma:diff_fine_w2}. 
%	are specified as follows:
%	% 
%	\begin{align} 
%		\mathbb{E}
%		\bigl[\hat{\mathcal{M}}_{t_{k+1}, t_k}^{f, (\mu),  j} | \mathcal{F}_{t_{k}} \bigr] = 0, \quad 0 \le k \le 2^{\ell-1} - 1, 
%	\end{align}
%	% 
%	and for any $p \ge 2$, there exist constants $C_1, C_2 > 0$ independent of $\mu \in (0,1)$ such that 
%	% 
%	\begin{align}
%		\max_{0 \le k \le 2^{\ell-1} - 1} 
%		\mathbb{E} 
%		\bigl[ 
%		| \hat{\mathcal{M}}_{t_{k+1}, t_k}^{f, (\mu), j}|^p 
%		\bigr] 
%		\le C_1 \mu^{p} \Delta_{\ell-1}^{3p/2}, 
%		\quad 
%		\max_{0 \le k \le 2^{\ell-1} - 1} 
%		\mathbb{E}
%		\bigl[ 
%		| \hat{\mathcal{N}}_{t_{k+1}, t_k}^{f, (\mu),  j} |^p
%		\bigr] 
%		\le C_2 \mu^{2p} \Delta_{\ell-1}^{2p}. 
%	\end{align}
	% 
\end{lemma}
\begin{proof}
(\ref{eq:hat_X_f_sd}) is obtained from the argument in the proof of Lemma \ref{lemma:f_anti_diff} in Appendix \ref{app:tech} together with Lemma \ref{lemma:diff_fine_w2}. 
We also deduce (\ref{eq:hat_X_c_sd}) by applying the argument used in the proof of Lemma \ref{lemma:coare} in Appendix \ref{app:tech}.    
%, from the definition of the standard/antithetic coarse discretizations (\ref{eq:X_c})/(\ref{eq:X_c_a}) in the main text with the dependence of the small parameter $\mu$, we deduce the result 
\end{proof}

%\begin{lemma} \label{lemma:f_coarse_diff_sd}
%	Work under the model class (\ref{eq:small_diffusion}).  
%	Let $1 \le j \le N$, $1 \le \ell \le L$, $0 \le k \le 2^{\ell - 1} - 1$ and $t_k = k \Delta_{\ell -1}$. 
%	It holds that 
	% 
%	where the remainder terms are specified as follows: 
%	% 
%	\begin{align}  
%		\mathbb{E} 
%		\bigl[
%		\hat{\mathcal{M}}_{t_{k+1}, t_k}^{c, (\mu), j} | \mathcal{F}_{t_{k}}
%		\bigr]
%		= 0, \quad 0 \le k \le 2^{\ell-1} - 1, 
%	\end{align}
%	% 
%	and for any $p \ge 2$ there exist constants $C_1, C_2 > 0$ independent of $\mu \in (0,1)$ such that
%	% 
%	\begin{align} 
%		\max_{0 \le k \le 2^{\ell - 1} - 1} 
%		\mathbb{E} \bigl[ 
%		| \hat{\mathcal{M}}_{t_{k+1}, t_k}^{c, (\mu),  j} |^p  
%		\bigr] \le C_1 \mu^p \Delta_{\ell-1}^{3p/2}, 
%		\quad 
%		\max_{0 \le k \le 2^{\ell - 1} - 1}  
%		\mathbb{E} \bigl[ 
%		| \hat{\mathcal{N}}_{t_{k+1}, t_k}^{c, (\mu),  j} |^p  
%		\bigr] \le C_2 \mu^{2p} \Delta_{\ell-1}^{2p}.  
%	\end{align} 
%\end{lemma}

\begin{lemma} \label{lemma:mil_anti_diff_sd}
	Work under the model class (\ref{eq:small_diffusion}).  
%	Let $1 \le j \le N$, $1 \le \ell \le L$,  $0 \le k \le 2^{\ell - 1} - 1$ and $t_k = k \Delta_{\ell -1}$. 
	It holds that: 
	\begin{align*}
		& \hat{X}^{\mathrm{TM}, [\ell], j}_{t_{k+1}} 
		= \hat{X}^{\mathrm{TM}, [\ell] , j}_{t_k}  
		+ \sum_{0 \le m \le d}  \mu^{\mathbf{1}_{m > 0}} 
		\sigma_{m}^j (\hat{X}^{\mathrm{TM} , [\ell]}_{t_k}) 
		\Delta B_{t_{k+1}, t_k}^{m} 
		+ \tfrac{1}{2} \mathcal{L}_{0} \sigma_{0}^j 
		(\hat{X}_{t_k}^{\mathrm{TM}, [\ell]} )  \Delta \eta_{t_{k+1}, t_k}^{00}  \\
		& \qquad + \mu^2 \sum_{1 \le m_1, m_2 \le d} 
		\mathcal{L}_{m_1} \sigma_{m_2}^j 
		( \hat{X}_{t_k}^{\mathrm{TM}, [\ell]} ) 
		\Delta \eta_{t_{k+1}, t_k}^{m_1 m_2} 
		+ \hat{\mathcal{M}}_{t_{k+1}, t_k}^{\mathrm{TM}, j} 
		+ \hat{\mathcal{N}}_{t_{k+1}, t_k}^{\mathrm{TM}, j}, 
	\end{align*}
	where the remainder terms $\hat{\mathcal{M}}_{t_{k+1}, t_k}^{\mathrm{TM}, j}$ and  $\hat{\mathcal{N}}_{t_{k+1}, t_k}^{\mathrm{TM}, j}$ have the same properties as  $ \bar{\mathcal{M}}^{f, j}_{t_{k+1}, t_k}$ and $\bar{\mathcal{N}}^{f,  j}_{t_{k+1}, t_k}$, respectively provided in Lemma \ref{lemma:diff_fine_w2}. 
%	where the remainder terms are specified as follows:
%	% 
%	\begin{align*} 
%		\mathbb{E}
%		\bigl[\hat{\mathcal{M}}_{t_{k+1}, t_k}^{\mathrm{T}\textrm{-}\mathrm{Mil}, (\mu),  j} | \mathcal{F}_{t_{k}} \bigr] = 0, \quad 0 \le k \le 2^{\ell-1} - 1, 
%	\end{align*}
%	% 
%	and for any $p \ge 2$, there exist constants $C_1, C_2 > 0$ independent of $\mu \in (0,1)$ such that 
%	% 
%	\begin{align*}
%		\max_{0 \le k \le 2^{\ell-1} - 1} 
%		\mathbb{E} 
%		\bigl[ 
%		| \hat{\mathcal{M}}_{t_{k+1}, t_k}^{\mathrm{T}\textrm{-}\mathrm{Mil}, (\mu), j}|^p 
%		\bigr] 
%		\le C_1 \mu^{p} \Delta_{\ell-1}^{3p/2}, 
%		\quad 
%		\max_{0 \le k \le 2^{\ell-1} - 1} 
%		\mathbb{E}
%		\bigl[ 
%		| 
%		\hat{\mathcal{N}}_{t_{k+1}, t_k}^{\mathrm{T}\textrm{-}\mathrm{Mil}, (\mu), j} 
%		|^p
%		\bigr] 
%		\le C_2 \mu^{2p} \Delta_{\ell-1}^{2p}. 
%	\end{align*}
	% 
\end{lemma}
\begin{proof}
	The result is deduced from Lemma \ref{lemma:diff_fine_mil} and the argument used in the proof of \cite[Lemma 4.9]{ml_anti}, and we omit the detailed proof. 
\end{proof}

\subsection{Proof of Theorem \ref{thm:m_bd_sd}} 
\label{app:m_bd_sd}
\begin{proof}
%	Let $t_k = k \Delta_{\ell -1}, \, 0 \le k \le 2^{\ell - 1}$. 
	We first show (\ref{eq:w2_sd}). From Lemma \ref{lemma:f_anti_diff_sd}, we obtain: 
%	for $1 \le k \le 2^{\ell - 1}$ and $1 \le j \le N$,  
	% 
	\begin{align} \label{eq:diff_sd}
		\begin{aligned}
		& \hat{X}^{f, j}_{t_k} - \hat{X}^{c, j}_{t_k}
		= 
		\sum_{0 \le i \le k-1} \sum_{0 \le m \le d} 
		\mu^{\mathbf{1}_{m > 0}}
		\bigl( \sigma_{m}^j (\hat{X}^{f}_{t_i})  - \sigma_{m}^j (\hat{X}^{c}_{t_i}) \bigr) 
		\Delta B_{t_{i+1}, t_i}^{m}  \\ 
		& \qquad \qquad
		+ \sum_{0 \le i \le k-1} 
		\bigl( 
		\mathcal{L}_{0} \sigma_{0}^j 
		(\hat{X}^{f}_{t_i}) 
		- \mathcal{L}_{0} \sigma_{0}^j 
		(\hat{X}^{c}_{t_i})
		\bigr) 
		\Delta \eta_{t_{i+1}, t_i}^{00}   \\
		& \qquad \qquad
		+ \mu^2 \sum_{0 \le i \le k-1} \sum_{1 \le m_1, m_2 \le d}
		\bigl( 
		\mathcal{L}_{m_1} \sigma_{m_2}^j 
		(\hat{X}^{f}_{t_i}) 
		- \mathcal{L}_{m_1} \sigma_{m_2}^j 
		(\hat{X}^{c}_{t_i})
		\bigr) 
		\Delta \eta_{t_{i+1}, t_i}^{m_1 m_2}  \\ 
		& \qquad  \qquad 
		+ \sum_{0 \le i \le k-1}
		\bigl\{ 
		\hat{\mathcal{M}}_{t_{i+1}, t_i}^{f, j} 
		+ \hat{\mathcal{N}}_{t_{i+1}, t_i}^{f, j}
		+ \hat{\mathcal{M}}_{t_{i+1}, t_i}^{c, j}
		+ \hat{\mathcal{N}}_{t_{i+1}, t_i}^{c, j}  
		\bigr\},  
		\end{aligned} 
	\end{align} 
	where the terms 
	$
	\hat{\mathcal{M}}_{t_{i+1}, t_i}^{f, j}, \, \hat{\mathcal{N}}_{t_{i+1}, t_i}^{f,  j}, \, \hat{\mathcal{M}}_{t_{i+1}, t_i}^{c,  j}, \, \hat{\mathcal{N}}_{t_{i+1}, t_i}^{c, j} 
	$
	are defined in Lemma \ref{lemma:f_anti_diff_sd}. 
	Then, we have from \eqref{eq:diff_sd} that, for $0 \le n \le 2^{\ell - 1}$ and $p \ge 2$: 
	\begin{align}
		\mathcal{S}_{n}^{\textrm{Weak-2}} 
		& \equiv  \mathbb{E} 
		\bigl[ \max_{0 \le k \le n} 
		\bigl\|
		\hat{X}_{t_k}^{f} 
		- 
		\hat{X}_{t_k}^{c}  
		\bigr\|^p  
		\bigr] 
%		\le   c_1  \sum_{1 \le j \le N} 
%		\mathbb{E} 
%		\bigl[ \max_{0 \le k \le n} 
%		\bigl|
%		\hat{X}_{t_k}^{f, j} 
%		- 
%		\hat{X}_{t_k}^{c, j}  
%		\bigr|^p  
%		\bigr]  \nonumber  \\ 
		 \le   C 
		\Bigl( 
		\mu^{p} \Delta_{\ell-1}^{p} 
		+ \Delta_{\ell-1} \sum_{0 \le k \le n-1} 
		\mathcal{S}_{k}^{\textrm{Weak-2}} 
		\Bigr) \label{eq:w2_iter_sd} 
	\end{align} 
	% 
	%We have from (\ref{eq:diff_sd}) that 
	% 
%	\begin{align}
%		\mathcal{S}_{n}^{\textrm{Weak-2}, (\mu)}  
%		& \le c_1 \sum_{1 \le j \le N} 
%		\mathbb{E} 
%		\Bigl[ \max_{0 \le k \le n} 
%		\bigl|
%		\hat{X}_{t_k}^{f, (\mu), [\ell], j} 
%		- 
%		\hat{X}_{t_k}^{c, (\mu), [\ell -1], j}  
%		\bigr|^p  
%		\Bigr]  
%		\le c_2 
%		\Bigl( 
%		\mu^{p} \Delta_{\ell-1}^{p} 
%		+ \Delta_{\ell-1} \sum_{0 \le k \le n-1} 
%		\mathcal{S}_{k}^{\textrm{Weak-2}, (\mu)} 
%		\Bigr) \label{eq:w2_iter}
%	\end{align} 
	% 
	for some constant $C > 0$ independent of $\mu \in (0,1)$, where we have used the same argument in the proof of \cite[Lemma 4.10]{ml_anti}. Applying the discrete Gr\"onwall inequality to (\ref{eq:w2_iter_sd}), we obtain (\ref{eq:w2_sd}).

	Subsequently, we prove (\ref{eq:tmil_sd}). It follows from Lemma \ref{lemma:mil_anti_diff_sd} that: 
%	for $1 \le k \le 2^{\ell - 1}$ and $1 \le j \le N$,  
	% 
	\begin{align} \label{eq:diff_mil_sd}
		\begin{aligned} 
		& \hat{X}^{\textrm{TM}, [\ell], j}_{t_k} 
		- \bar{X}^{\textrm{TM}, [\ell-1], j}_{t_k} 
		 = \! \! \! 
		\sum_{\substack{0 \le i \le k-1 \\ 0 \le m \le d}}  
		\mu^{\mathbf{1}_{m > 0}}
		\bigl(
		\sigma_{m}^j (\hat{X}^{\textrm{TM}, [\ell]}_{t_i})
		- \sigma_{m}^j (\bar{X}^{\textrm{TM}, [\ell-1]}_{t_i})  \bigr) 
		\Delta B_{t_{i+1}, t_i}^{m} \\ 
		& \qquad \qquad  
		+ \sum_{0 \le i \le k-1} 
		\Bigl( \tfrac{1}{2} \mathcal{L}_{0} \sigma_{0}^j 
		(\hat{X}^{\textrm{TM}, [\ell]}_{t_i}) 
		-   \mathcal{L}_{0} \sigma_{0}^j 
		(\bar{X}^{\textrm{TM}, [\ell-1]}_{t_i}) \Bigr) 
		\Delta \eta_{t_{i+1}, t_i}^{00}   \\  
		& \qquad \qquad 
		+ \mu^2 \sum_{\substack{0 \le i \le k-1 \\ 1 \le m_1, m_2 \le d }}
		\bigl( 
		\mathcal{L}_{m_1} \sigma_{m_2}^j 
		(\hat{X}^{\textrm{TM}, [\ell]}_{t_i}) 
		- \mathcal{L}_{m_1} \sigma_{m_2}^j 
		(\bar{X}^{\textrm{TM}, [\ell-1]}_{t_i})
		\bigr) 
		\Delta \eta_{t_{i+1}, t_i}^{m_1 m_2}   \\ 
		&  \qquad \qquad    + \sum_{0 \le i \le k-1}
		\bigl\{ 
		\hat{\mathcal{M}}_{t_{i+1}, t_i}^{\textrm{TM},  j} 
		+ \hat{\mathcal{N}}_{t_{i+1}, t_i}^{\textrm{TM},  j}
		\bigr\},  
		\end{aligned} 
	\end{align} 
	where the terms $\hat{\mathcal{M}}_{t_{i+1}, t_i}^{\textrm{TM}, j}, \, 
	\hat{\mathcal{N}}_{t_{i+1}, t_i}^{\textrm{TM},  j}$ are defined in Lemma \ref{lemma:mil_anti_diff_sd}. 
%		We define: for $0 \le n \le 2^{\ell - 1}$ and $p \ge 2$, 
	% 
%	\begin{align*}
%		\mathcal{S}_{n}^{\textrm{TM}} 
%		= \mathbb{E} 
%		\Bigl[ \max_{0 \le k \le n} 
%		\bigl\|
%		\hat{X}_{t_k}^{\textrm{TM}, (\mu), [\ell]} 
%		- 
%		\bar{X}_{t_k}^{\textrm{TM}, (\mu), [\ell -1]}  
%		\bigr\|^p  
%		\Bigr]. 
%	\end{align*}
	% 
	Then, it follows from (\ref{eq:diff_mil_sd}) that, for $0 \le n \le 2^{\ell - 1}$ and $p \ge 2$: 
	\begin{align} \label{eq:tmil_iter}
	\mathcal{S}_{n}^{\textrm{TM}} 
	\equiv  \mathbb{E} 
	\Bigl[ \max_{0 \le k \le n} 
	\bigl\|
	\hat{X}_{t_k}^{\textrm{TM}, [\ell]} 
	- 
	\bar{X}_{t_k}^{\textrm{TM}, [\ell -1]}  
	\bigr\|^p  
	\Bigr]  \le C \Bigl( \Delta_{\ell-1}^p +  
		\Delta_{\ell-1}  \sum_{0 \le k \le n-1} \mathcal{S}_k^{\textrm{TM}} \Bigr) 
	\end{align}
	for some $C > 0$ independent of $\mu \in (0,1)$, where we have used again the argument in the proof of \cite[Lemma 4.10]{ml_anti}. We here note that the first term of the R.H.S. of (\ref{eq:tmil_iter}) is now independent of $\mu \in (0,1)$ and comes from:  
	\begin{align}
		& \mathbb{E} \left[ 
		\max_{0 \le k \le n} 
		\Bigl| 
		\sum_{0 \le i \le k-1} 
		\Bigl( 
		\tfrac{1}{2} \mathcal{L}_{0} \sigma_{0}^j 
		(\hat{X}^{\textrm{TM}, [\ell]}_{t_i}) 
		-  \mathcal{L}_{0} \sigma_{0}^j 
		(\bar{X}^{\textrm{TM}, [\ell-1]}_{t_i})
		\Bigr) 
		\Delta \eta_{t_{i+1}, t_i}^{00}   
		\Bigr|^p  \right] \label{eq:tmil_half}
		\\ 
		& \quad 
		\le c_1 n^{p-1} \sum_{0 \le i \le n-1} | \Delta \eta_{t_{i+1}, t_i}^{00} |^p
		\le  c_2 n^p \Delta_{\ell - 1}^{2p}
		\le c_2 T^p \Delta_{\ell - 1}^p, \nonumber 
	\end{align} 
	for some $c_1, c_2 > 0$ independent of $\mu \in (0,1)$. In particular, due to `$\tfrac{1}{2}$', the term (\ref{eq:tmil_half}) is not bounded as the second term of the R.H.S. of (\ref{eq:tmil_iter}) in the case of the truncated Milstein scheme. Finally, applying the discrete Gr\"onwall inequality to (\ref{eq:tmil_iter}), we conclude. 
\end{proof}

\bibliographystyle{siamplain}
%\bibliography{references}

\end{document}